 \documentclass[draft]{article}

\usepackage{amsmath,amsfonts,amsthm,amssymb,amscd,cancel,color}
\usepackage{enumitem}
\usepackage{verbatim}
\usepackage{ulem,color}

\providecommand{\keywords}[1]
{
  \small	
  \textbf{\textit{Keywords---}} #1
}

\renewcommand{\Re}{\mathop{\rm Re}\nolimits}
\renewcommand{\Im}{\mathop{\rm Im}\nolimits}

\def\uno{{\kern+.3em {\rm 1} \kern -.22em {\rm l}}}

\theoremstyle{plain} \newtheorem{theorem}{Theorem}[section]
\newtheorem{lemma}[theorem]{Lemma}
\newtheorem{proposition}[theorem]{Proposition}
\newtheorem{corollary}[theorem]{Corollary} \theoremstyle{definition}
\newtheorem{definition}[theorem]{Definition}
\newtheorem{assumption}[theorem]{Assumption} \theoremstyle{remark}
\newtheorem{remark}[theorem]{Remark}

\newcommand{\R}{{\mathbb R}}
\newcommand{\U}{{\mathcal U}}

\newcommand{\Z}{{\mathbb Z}}

\newcommand{\N}{{\mathbb N}}

\def\im{{\rm i}}

\newcommand{\C}{\mathbb{C}}

\def\({\left(}
\def\){\right)}
\def\<{\left\langle}
\def\>{\right\rangle}

\newcommand{\Hrad}{H_{\mathrm{rad}}}
\newcommand{\Lrad}{L_{\mathrm{rad}}^2}
\newcommand{\stz}{\mathrm{Stz}}

\newcommand{\stzo}{\mathrm{Stz}^1}

\newcommand{\stzso}{\mathrm{Stz}^{*,1}}

\numberwithin{equation}{section}

\begin{document}

\title{Long time oscillation of solutions of nonlinear Schr\"odinger equations near minimal mass ground state}

\author {Scipio Cuccagna, Masaya Maeda}

\maketitle

\begin{abstract}
In this paper, we consider the long time dynamics of radially symmetric solutions of nonlinear Schr\"odinger equations (NLS) having a minimal mass ground state.
In particular, we show that there exist solutions with initial data near the minimal mass ground state that oscillate for long time.
More precisely, we introduce a coordinate defined near the minimal mass ground state which consists of finite and infinite dimensional part associated to the discrete and continuous part of the linearized operator.
Then, we show that the finite dimensional part,   two dimensional, approximately obeys Newton's equation of motion for a particle in an anharmonic potential well.
Showing that the infinite dimensional part is well separated from the finite dimensional part, we will have long time oscillation.

\keywords{Nonlinear Schr\"odinger equation; ground states; oscillation}
\end{abstract}

\tableofcontents

\section{Introduction}
\label{sec:introduction}

In this paper we consider a   nonlinear Schr\"odinger equation (NLS)
\begin{equation}\label{1}
\im \partial_t u  = -\Delta u  + g(|u|^2) u ,\quad (t,x)\in \R\times \R^3,
\end{equation}
with $u(0)=u_0 \in H^1_{\mathrm{rad}}(\R^3, \C ):=\{u\in H^1(\R^3, \C )\ |\ u\text{ is radially symmetric}\}$ and $g \in C^\infty([0,\infty) ,\R)$ with  $g(0)=0$ and   $|g ^{(n)}(s) |\leq C_n s^{p-n}$  in $s\geq 1$ for some $p<2$ and $C_n>0$ for all $n\in\N\cup\{0\}$.

We assume the existence of a one parameter family of ground states, i.e.\ there exists an open interval $\mathcal O\subset (0,\infty)$ and a $C^2$  map
\begin{align}\label{2}
\mathcal O \ni \omega  \mapsto \phi_\omega \in \Hrad^1(\R^3,\R),
\end{align}
s.t.\ $\phi_\omega$ are positive and solve
\begin{align}\label{3}
0=-\Delta \phi_\omega + \omega \phi_\omega + g(\phi_\omega^2) \phi_\omega.
\end{align}
Notice that $e^{\im \omega t}\phi_\omega$ are solutions of \eqref{1}, which are also called ground states.
We further assume that the map $\omega\mapsto \|\phi_\omega\|_{L^2}$ has a nondegenerate local minimum, i.e.\
\begin{align}\label{4}
\text{$\exists\omega _*\in \mathcal O$ s.t.\   $q'(\omega_*)=0$ and $q''(\omega_*)>0$, where $q(\omega):=\frac 1 2 \|\phi_\omega\|_{L^2(\R^3, \C )}^2$.}
\end{align}

\begin{remark}
The above hypotheses have been numerically verified for various equations involving
 \textit{saturated} nonlinearities such as $g(s)=\frac{s}{1+s}$, relevant in optics, see \cite{TCL96PRL},
or $g(s)=\frac{s^3}{1+s^2}$, discussed by Marzuola \textit{et al.} \cite{MRS10JNS}.  Other examples are in  Buslaev--Grikurov \cite{BG01MCS}.
For an analytical result for double power nonlinearity in one dimension, see \cite{Maeda08KMJ}.
\end{remark}

\noindent
The above set up provides some interesting patterns.  First of all, by \eqref{4}, we have $q'(\omega)<0$ for $\omega\in (\omega_*-\delta,\omega_*)$ and $q'(\omega)>0$ for $\omega\in (\omega_*, \omega_*+\delta)$ for some $\delta>0$.
It is well known that under standard nondegeneracy assumption (1.\ of Assumption \ref{a:1} below), the ground state $e^{\im \omega t}\phi_\omega$ is orbitally stable (resp.\ unstable) if $q'(\omega)>0$ ($q'(\omega)<0$), see e.g.\ \cite{GSS}.


\noindent
Therefore, $\omega_*$ is the critical frequency dividing the stable and unstable ground states.
Comech--Pelinovsky \cite{CP03CPAM} proved that this critical  ground state $e^{\im \omega_* t}\phi_{\omega_*}$
  is orbitally instable (see also \cite{Ohta11JFA, Maeda12JFA}).  The proof uses an appropriate    system of   coordinates, involving a mixture of three scalar  coordinates (related to the discrete spectrum of an appropriate linearization of the NLS) and one infinite dimensional coordinate (related to the continuous spectrum of the linearization).
By this coordinates, \eqref{1} is viewed
as a perturbation of   a finite dimensional Hamiltonian system involving the scalar coordinates, with
$e^{\im \omega_* t}\phi_{\omega_*}$ shown to be unstable  for the finite dimensional system. Then, by controlling the continuous spectrum coordinate by means of classical methods involving the energy, like in \cite{Weinstein85,Weinstein86,GSS},  Comech--Pelinovsky \cite{CP03CPAM} produces some solutions of
\eqref{1} whose discrete coordinates   \textit{shadow} the unstable solutions of the
finite dimensional system.
More precisely, it is shown that for  a sufficiently large   period of time the discrete coordinates have motion similar  to that of  the finite dimensional system while the continuous coordinate of the  NLS   remains small.
This yields
the desired  instability   for the full NLS.

Marzuola \textit{et al.} \cite{MRS10JNS} considered the framework of  Comech--Pelinovsky \cite{CP03CPAM}
and developed a systematic numerical exploration of solutions of both the finite dimensional approximation and of
the full NLS.  While observing the patterns of   Comech--Pelinovsky \cite{CP03CPAM},  Marzuola \textit{et al.}     \cite{MRS10JNS}    identified  another class of  initial data
  near the minimal mass ground state $e^{\im \omega_* t}\phi_{\omega_*}$.
The corresponding solutions look like $e^{\im \vartheta (t)}\phi_{\omega (t)}$ locally in space with
$\omega (t)$  displaying an oscillating motion, which appears periodic. A similar pattern had previously been observed also by
   Buslaev--Grikurov \cite[fig. 2 case $\alpha =0.5$]{BG01MCS}.

The aim of this paper is to rigorously justify the observation of Marzuola \textit{et al.} \cite{MRS10JNS} and provide a theoretical explanation of the oscillation phenomena.
Very roughly, our main result is the following.

\begin{theorem}[informal]\label{main:informal}
For any $M\in\N$, $M\geq 2$, there exists an  open set $\mathcal U_M\subset\Hrad^1$ near $\phi_{\omega_*}$ contained in $\{u\in \Hrad^1\ |\ \|u\|_{L^2}^2/2>q(\omega_*)\}$ s.t.\  the solutions  of \eqref{1} with initial datum $u_0\in \mathcal U_M$  can be expressed as $u(t)=e^{\im \vartheta(t)}\phi_{\omega(t)}+\mathrm{error}$  and $\omega(t)$ is oscillating  for long time.
In particular, if $u\in \mathcal U_M$, we have  $\inf _{\vartheta \in \R}\| u_0 - e^{\im \vartheta}\phi_{\omega_*}\| _{\Hrad^1}\sim \epsilon  $,
the oscillations occur over a time period of length $  T\sim \epsilon ^{-M} $ and
each oscillation occurs in a time approximately equal to $  \epsilon ^{-\frac{1}{2}}$, where  $\epsilon^2:=\frac12\|u\|_{L^2}^2-q(\omega_*)$.
\end{theorem}

\begin{remark}
The open sets $\mathcal U_M\subset \Hrad^1(\R^3,\C)$  are not neighborhoods of $\phi_{\omega_*}$. As stated in Theorem \ref{main:informal}, they belong  to the region of $ \frac12 \|u\|_{L^2}^2>q(\omega_*)=\frac12 \|\phi_{\omega_*}\|_{L^2}^2$.
When $\frac12\|u\|_{L^2}^2<q(\omega_*)$, then $\omega(t)$ does not oscillate and the solution  escapes from the tubular neighborhood of $\mathcal T_{\omega_*}:=\{e^{\im \vartheta}\phi_{\omega}\ |\ \vartheta\in\R\}$.
This is the instability proved by Comech-Pelinovsky  \cite{CP03CPAM}.
Notice that the escape from $\phi_{\omega_*}$ of the solutions
$\frac12\|u\|_{L^2}^2<q(\omega_*)$  is   natural from the point of view of the Soliton Resolution conjecture because if the solution stayed  near $\mathcal T_{\omega_*}$ for all time, then we would expect  it to converge, modulo phase and a linear wave, to $\phi_\omega$ for $\omega$ near $\omega_*$.
 And since the linear wave would have nonnegative $L^2$ norm     scattered to spatial infinity, we would get a  contradiction with the $L^2$--norm conservation of equation \eqref{1}, because $\|\phi_{\omega}\|_{L^2}\geq \|\phi_{\omega_*}\|_{L^2}>\|u_0\|_{L^2}$. See the related discussion in the proof of Theorem 1.4 of \cite{CM16JNS}.
See also, \cite{KOPV17ARMA}.
\end{remark}

\noindent
The precise statement of Theorem \ref{main:informal} will be given in Theorems \ref{thm:normal}, \ref{thm:dynamics} and \ref{thm:shadow}.
In Theorem \ref{thm:normal}, we will make the meaning ``the solution $u(t)$ can expressed as $u(t)=e^{\im \vartheta(t)}\phi_{\omega(t)}+\mathrm{error}$" clear by introducing an appropriate coordinate system.
In Theorem \ref{thm:dynamics} and \ref{thm:shadow}, we will make the meaning "long time" clear and deduce the effective equation which governs the motion of $\omega(t)$.
This will provide the meaning of ``$\omega(t)$ is oscillating for long time".

Similar results on persistence of oscillating patterns for long periods of time have been
 proved, for different but analogous problems,  by Marzuola and Weinstein \cite{MW10DCDS}, see also \cite{GMW15DCDS,PP12JDE}.
 However here we are able to prove the persistence of the oscillations for much longer times.
 In fact, in \cite{MW10DCDS} the number of oscillations  is of the order of $\epsilon ^{-a _0} $ for a fixed $a _0$, while in this paper  the number of oscillations can be arranged to be
 $\epsilon ^{-N} $  for any $N$.  A somewhat similar result to {Theorem} \ref{main:informal}
 is in \cite{BM16CMP},  which deals with the effects that an arbitrarily small  but classically trapping
   potential exerts on solitons of a translation invariant NLS.  In  \cite{BM16CMP}
   the time frame of the phenomena is as our's, not coincidentally since  in fact we use similar methods.
   However here we examine our oscillating patterns in more detail than what is done for possibly analogous patterns in \cite{BM16CMP}. For the trapping of a soliton in a well
   we refer also to \cite{Bonanno15JDE,soffer, FGJS04CMP,HZ08IMRN,JLFGS06AHP}   and to the references therein.

We now explain the strategy of the proof.
As we mentioned above, for problems related to the dynamics near ground states, following the classical works such as \cite{Weinstein85,SW90CMP,SW92JDE}, it  is quite natural to introduce
coordinates which are in part discrete and in part  continuous coordinates and which are related to the spectral decomposition of the linearization of \eqref{1}.
For example, when we study the asymptotic stability of ground states (i.e.\ showing solutions near ground states decompose to ground state and linear wave) for the case $q'(\omega)>0$, we can decompose the solution $u(t)$ as
\begin{align*}
u(t)=e^{\im \vartheta(t)}(\phi_{\omega(t)}+r(t)).
\end{align*}
Here, $r(t)$ is symplectic orthogonal to the kernel of the linearized operator.
If we assume that there are no internal modes (non-zero discrete spectrum) for the linearized operator, then $(\vartheta, \omega)$ can be viewed as the discrete coordinates and $r $ as the continuous coordinate, with $(\vartheta, \omega ,r)$ a complete system of independent coordinates.
Therefore, the dynamics of $u(t)$ reduces to the dynamics of $\vartheta(t)$, $\omega(t)$ and $r(t)$.
As in     \cite{Cuccagnatrans,Bambusi13CMPas}  it is natural to  replace the coordinate system $(\vartheta, \omega, r)$ with $(\vartheta,Q,r)$, where $Q=\frac12 \|u\|_{L^2}^2$ is a first integral of   motion, and
by an elementary application of  Noether's principle,  to decrease the number of coordinates, reducing to a NLS--like equation on $r$. Then   the proof of asymptotic stability results from the proof of the scattering of $r$.

The above argument,  is based on the fact that the generalized kernel of the linearized operator
    is $2$ dimensional,   which is a consequence of  $q'(\omega)\neq 0$.
In the case $q'(\omega)=0$ the generalized kernel becomes $4$ dimensional, and instead of the previous ansatz with coordinates
$(\vartheta, \omega, r)$, we have
\begin{align*}
u(t)=e^{\im \omega(t)}\(\phi_{\omega(t)}+\lambda(t)\psi_3+\mu(t)\psi_4+r(t)\),
\end{align*}
where $\psi_3$ and $\psi_4$ are additional generalized eigenfunctions (see Proposition \ref{prop:1}).
Thus, under the assumption that there exist no internal modes, our problem reduces to the study of the dynamics of $(\vartheta,\omega,\lambda,\mu,r)$.
Unlike in the case $q'(\omega)\neq 0$, we cannot replace $\omega$ by $Q$.
But  we can replace $\mu$ by $Q$.
Then, by an elementary application of Noether's principle, we are left with $(\omega,\lambda,r)$. The equations for modulation parameters $(\omega,\lambda )$
have already been studied in the literature, see   (4.11) in \cite{CP03CPAM} or (3.5) in \cite{MRS10JNS}, but are not well understood. Here we add to their understanding
by
following an approach initiated in \cite{Cuccagna11CMP}. That is,   using the Hamilton structure of NLS \eqref{1},   we  move to a   Darboux system of coordinates, see  Proposition \ref{prop:3} in Subsection \ref{sec:Darboux} (and we further simplify them by canonical transformations, see Sect.\ \ref{sec:normalf}).
In Darboux   coordinates we have
\begin{align*}
\dot \omega= A^{-1}\partial_\lambda E,\quad \dot \lambda=-A^{-1}\partial_\omega E
\end{align*}
where $E$ is the energy (Hamiltonian) of the  NLS and $A=A(\omega)>0$ is a function that here we can  assume constant.
If we expand $E=E_f(Q,\omega,\lambda ,0)+``\text{terms with $r$}"$,
 where for $E_f $ see \eqref{eq:fenergy},
  and think the second term as an error, we   obtain
\begin{align*}
\dot \omega = A^{-1}\partial_\lambda E_f(Q,\omega,\lambda , 0)+\mathrm{error},\quad \dot \lambda =-A^{-1}\partial_{\omega}E_f(Q,\omega,\lambda ,0)+\mathrm{error}.
\end{align*}
Ignoring errors, this is a  2 dimensional Hamiltonian system with energy $\sim 2^{-1}A \lambda^2+ d(\omega)-\omega Q$, where $d(\omega)=E(\phi_\omega)+\omega q(\omega)$ satisfies $d'(\omega)=q(\omega)$.
So, thinking $2^{-1}A \lambda^2$ as the kinetic energy and $V_Q(\omega)=d(\omega)-\omega Q$ as potential energy, we see that $\omega$ approximately   Newton's equation in a potential well $V_Q$, with $\lambda$   the momentum.
If $Q>q(\omega_*)$, the well has a local minimum at a point $\omega=\omega_+>\omega_*$.
Thus, in this case, even starting from $\omega(0)=\omega_*$, $\omega $ will oscillate around $\omega_+$, consistently  with the numerical observations   by Marzuola \textit{et al.} \cite[Fig.\ 4--6]{MRS10JNS}.
Similarly, we see that if $Q<q(\omega_*)$, then $V(Q,\omega)$ is monotically increasing and $\omega$ will fall to the unstable side ($\omega<\omega_*$) without oscillation.

As we mentioned,  Marzuola and Weinstein \cite{MW10DCDS}, see also \cite{GMW15DCDS,PP12JDE}, considered similar problems,  proving long time   shadowing of solutions of an appropriate
finite dimensional systems by solutions of NLS's.
However,  as we mentioned above, we prove the persistence of the
 oscillating patterns   for much longer   than in \cite{MRS10JNS}, because we
are more careful in the choice of the coordinates, as we now explain.

Because of the fact that the coordinates are derived from the linearization of \eqref{1}, at a linear level
the various coordinates evolve in time  independently from each other,   they interact
at a nonlinear level only.  The reason why the oscillations last for a  long time,  but presumably not for ever,
is that the continuous coordinate, if initially very small,   remains small for a long time. However, because of nonlinear interactions, the  continuous coordinate   grows and
after a sufficiently long time,   disrupts the motion of the discrete coordinates.
The latter coordinates, for a long time  and as long as the continuous coordinate
   is very small, evolve like solutions of the finite dimensional system, but,
   from a certain moment on, evolve   quite differently. What we do in this paper,
   which is not done in related papers such as \cite{MW10DCDS}, is to choose the coordinates (i.e.\ a normal forms argument)
   in such a way that it is  quite clear
   that the interaction between the continuous coordinate with the various discrete coordinates occurs at arbitrarily  high order: the higher this  order, the longer   the time it takes
   for the continuous coordinate to disrupt the discrete coordinates.
   This is analogous  to the approach
   by Bambusi and Maspero \cite{BM16CMP}.

While we prove oscillations for very long times, the question on what happens
asymptotically to these oscillating patterns remains open.
In   \cite[Conjecture 1.1]{MRS10JNS} it is
   suggested  that asymptotically these oscillating solutions should scatter to a stable ground state. Reference is made to possible  {radiation damping}  phenomena and to the  damped oscillations observed numerically in some cases for the mass critical saturated NLS
by LeMesurier \textit{et al.} \cite[figures 7 and 16]{LemPapSulSul88}. For further comments and references see also the discussion in \cite[Sect.\ 9.3.2--9.3.3] {SulemSulem}.

\noindent
Damped oscillations similar to the ones in LeMesurier \textit{et al.} \cite{LemPapSulSul88}  are observed
   in Buslaev--Grikurov \cite[fig. 2]{BG01MCS}.  However
only the case $\alpha =0.5$ in  \cite[fig. 3]{BG01MCS}, which has an unknown asymptotic behavior,  is in the same regime  of Marzuola \textit{et al. }   The other cases in \cite[fig. 3]{BG01MCS}, all displaying damping and convergence, have $\alpha \ll 1$, which means that  the initial value  $u_0$ has large distance  in $\Hrad^1 (\R ^3)$  from the orbit of $\phi_{\omega_*}$, with distance$ \stackrel{\alpha\to 0^+ }{\rightarrow}\infty$. Ultimately the regime $\alpha \ll 1$ in \cite[fig. 3]{BG01MCS}
seems to be very different from the  regime considered in   \cite{MRS10JNS}.

   We think that
   it is   plausible that a radiation damping phenomenon like in \cite{SW99IM,BP95,Cuccagna11CMP,CM15APDE} will prove  Conjecture 1.1 in  \cite {MRS10JNS}.  However, it is almost certain   that the type of  coordinates used
   in the present paper, which originate from the analysis in Comech and Pelinovsky \cite{CP03CPAM}, are
   inadequate to prove the conjecture. There  should exist    more "nonlinear"
     coordinates, possibly related to the ones used in the study of
   the blow up by Perelman  \cite{perelman}  and Merle and Raphael  \cite{merle}.  For  few more comments on the conjecture we refer
to Remark \ref{rem:FGR}.

This paper is organized as follows.
In section \ref{sec:setup}, we set up the framework and give the precise statement of the main results (Theorems \ref{thm:normal}, \ref{thm:dynamics} and \ref{thm:shadow}).
Section \ref{sec:normal} will be devoted to the proof of Theorem \ref{thm:normal}, which consists of Darboux theorem (Proposition \ref{prop:3}) and Birkhoff normal forms  in subsection \ref{sec:normalf}.
In section \ref{sec:contspec}, we prove Theorems \ref{thm:dynamics} and \ref{thm:shadow}.
In the appendix of this paper, we provide the proof of Proposition \ref{prop:1}, Lemmas \ref{lem:Rexpand} and Lemma \ref{lem:errorstzbound}.

\section{Set up and Main result}\label{sec:setup}

\subsection{Notation}
We collect notations which we use throughout this paper.
\begin{itemize}
\item
By $f \lesssim g$, we mean there exists a constant $C>0$ such that $f\leq C g$.
Here, $C$ only depends on parameters which are irrelevant to us.
Further, by $f\lesssim_a g$, we mean that there exists a constant $C=C_a$ (depending on parameter $a$) such that $f\leq C_a g$.
If $f\lesssim g$ and $g\lesssim f$, then we write $f\sim g$.
\item
We use Pauli matrices
\begin{align}\label{17.3} \sigma_0=\begin{pmatrix} 1 & 0 \\ 0 & 1 \end{pmatrix}, \sigma_1=\begin{pmatrix} 0 & 1 \\ 1 & 0 \end{pmatrix},\ \sigma_2=\begin{pmatrix} 0 & -\im \\ \im & 0 \end{pmatrix},\  \sigma_3=\begin{pmatrix} 1 & 0 \\ 0 & -1 \end{pmatrix}.\end{align}
\item
For Banach spaces $X,Y$ we denote by $\mathcal L(X,Y)$   the Banach space of all bounded operators from $X$ to $Y$. If $X=Y$, we write $\mathcal L(X,X)=\mathcal L(X)$.
Further, we define $\mathcal L^n(X,Y)$   by $\mathcal L^0(X,Y)=Y$, $\mathcal L^1(X,Y)=\mathcal L(X,Y)$,  $\mathcal L^n(X,Y)=\mathcal L^{n-1}(X,\mathcal L(X,Y))$ and we set $\mathcal L^n(X):=\mathcal L^n(X,X)$.

\item Given a Banach space $X$, $x_0\in X$ and $r>0$, we set $D_X(x_0,r):=\{ x\in X:\|x-x_0\| _X <r \}$.
Further, for $A\subset X$, we set
\begin{align}\label{def:tubnbd}
D_X(A,r):=\{ x\in X\ |\ \mathrm{dist}_X(x,A)<r\},\quad \mathrm{dist}_X(x,A):=\inf_{y\in A}\|x-y\|_X.
\end{align}

\item
We embed $\C \hookrightarrow \C^2$
  using the natural identification
\begin{equation}\label{eq:identif}
    u\in \C \to \widetilde{u}:= {^t(u \ \bar u )}\in \widetilde{\C}:=\{ {^t (z \ \bar z )}\in \C^2 : z\in \C  \} \subset \C^2.
\end{equation}
We set $\<U_,V\>_{\C^2}:=
2^{-1} (u_1  v_1+u_2   v_2 ) $ for  $U=(u_1\   u_2)^t $ and $V=(v_1\   v_2)^t$ in $\C^2$ and
\begin{equation}\label{eq:bilform}
   \<U,V\>:=\int_{\R^3}\<U(x),V(x)\>_{\C^2}\,dx \text{ for $U,V\in L^2(\R^3,\C^2)$. }
\end{equation}
We will also use $\<U,V\>$ as the dual coupling between $H^{-1}(\R^3,\C^2)$ and $H^1(\R^3,\C^2)$.
 By our definition  we have $\< \widetilde{u},\sigma _1\widetilde{v}\>_{\C^2}=\Re u\bar v$, and in particular $\<\widetilde{u},\sigma _1\widetilde{u}\>_{\C^2}=|u|^2$.
For $\widetilde u\in \widetilde \C$, we define $|\widetilde u|^2:=\<\widetilde{u},\sigma _1\widetilde{u}\>_{\C^2}$.

\item
We define the skew-symmetric form $\Omega$ by
\begin{align}
\Omega(\cdot,\cdot):=\<\im \sigma_3 \cdot, \sigma_1  \cdot\>.\label{17.6d}
\end{align}
Notice that $\Omega(\widetilde u, \widetilde v)=\Im \int_{\R^3} u(x)\overline{v(x)}\,dx$.

\item
For $F\in C^1( \Hrad^1(\R^3,\widetilde{\C} ),\R)$, we define the gradient $\nabla
F(\widetilde{u})\in  \Hrad^{-1}(\R^3,\widetilde{\C} )$ by
\begin{align}\label{def:grad}
dF(\widetilde{u})\tilde v =\<\nabla F(\widetilde{u}),\sigma _1\widetilde{v}\>  \text{  for all $\widetilde{v}\in \Hrad^1(\R^3,\widetilde{\C} )$}
\end{align}
where $dF(\widetilde u) \in \mathcal L(  \Hrad^1(\R^3,\widetilde{\C} ), \R)$ is the Fr\'echet
derivative of $F$.
For $F\in C^1(L^2_{rad}(\R^3,\widetilde{\C} ),\R)$, the gradient $\nabla F(\widetilde{u})$ will be similarly defined and it will belong in $L^2_{rad}(\R^3,\widetilde{\C} )$.
Further, for $F\in C^n( \Hrad^1(\R^3,\widetilde{\C} ),\R)$    define
 \begin{equation*}
    \nabla^m F \in C^{n-m}( \Hrad^1(\R^3,\widetilde{\C} ),\mathcal L^{m-1}( \Hrad^1(\R^3,\widetilde{\C} ), \Hrad^{-1}(\R^3,\widetilde{\C} )))
 \end{equation*} inductively by $\nabla^{m} F =d \nabla^{m-1}F$ ($2\leq m\leq n$).

\item
For $s\in \R$ and for $X=\widetilde \C$ or $\C^2$, we set
\begin{equation} \label{eq:Sigma}
    \Sigma^s(\R^3,X):=\{U\in \mathcal S_{\mathrm{rad}}'(\R^3,X):\  \ \|U\|_{\Sigma^s}<\infty\} \text{ where } \|U\|_{\Sigma^s}:=\|(-\Delta + |x|^2+1)^{s/2} U\|_{L^2}.
\end{equation}
We further set $\Sigma^s_c(\R^3,\widetilde \C):=P(\omega_*)\Sigma^s(\R^3,\widetilde \C)$, where $P(\omega)$ is defined in \eqref{55}.

\item
For an operator $A$, we set $\mathcal N(A)$ to be the kernel of $A$ and $\mathcal N_g(A)=\cup_{n=1}^\infty \mathcal N(A^n) $ to be the generalized kernel.

\end{itemize}

\subsection{The linearized operator}

It is a traditional approach to study the stability of solitons by appropriate choices of coordinates
where the problem is reframed as the stability  of 0 and where nonlinear stability is
derived by some form of Lyapunov or linear stability argument, see for example   \cite{Weinstein85,GSS,BP92StP}. In all these approaches a key role is played by
the linearization of the NLS at the soliton $e^{\im t \omega }\phi _{\omega}$,
related to the operator \eqref{36.6} below. Such operator is    $\R$--linear but  not $\C$--linear.
Since it is very important to consider spectral decompositions of this operator, it is quite natural to ``complexify", which leads to viewing a $\C$--valued function $u$ as a
$\widetilde \C(\hookrightarrow \C^2)$--valued function $\widetilde u={^t(u\ \bar u) }$ (see \eqref{eq:identif}).
The map $u\to \widetilde{u}$ yields a natural identification of
$\Hrad^s(\R^3,\C )$   with $  \Hrad^s(\R^3,\widetilde{\C} )\subset \Hrad^s(\R^3,\C^2)$.
By this identification, NLS \eqref{1} can be written as
\begin{align}\label{1tilde}
\im \sigma_3 \partial_t \widetilde u =-\Delta \widetilde u + g(|\widetilde u|^2)\widetilde u.
\end{align}


We introduce the energy $E$ and mass $Q$, which are the constant under the flow of \eqref{1tilde}.

\begin{definition}
We define $E\in C^6( \Hrad^1(\R^3,\widetilde{\C} ),\R)$ and $Q \in C^\infty( \Lrad(\R^3,\widetilde{\C} ) ,\R)$ by
\begin{align}
E(\widetilde u) &:= \frac{1}{2}\<(-\Delta)\widetilde{u},\sigma _1\widetilde{u}\>+E_P(\widetilde u) \text{ with }       E_P(\widetilde u):=\frac 1 2 \int_{\R^3}G(|\widetilde{u}(x)|^2)\,dx,\label{15}\\
Q(\widetilde u) &:=\frac{1}{2}\< \widetilde{u},\sigma _1\widetilde{u}\>.\label{16}
\end{align}
where $G(s)=\int_0^s g(\tau)\,d\tau$.
Recall that $|\tilde u(x)|^2=\<\tilde u(x), \sigma_1 \tilde u(x)\>_{\C^2}$.
\end{definition}
\begin{remark}
The fact $E$ is only $C^6$ comes from $H^1(\R^3, {\C} )\hookrightarrow L^p(\R^3, {\C} )$ for $p=6$ but not  for $p>6$.
For example, $F(\widetilde{u})=\int_{\R^3}|\widetilde{u}(x)|^8\,dx$ is not defined for all $\widetilde{u}\in  \Hrad^1(\R^3,\widetilde{\C} )$.
However, $E$ is $C^\infty$ in for example $H^s$ or $\Sigma^s$ for $s>3/2$.
\end{remark}

We further set the action $S_\omega(u)$ by
\begin{align}\label{33}
S_\omega(\widetilde{u} ):=E(\widetilde{u})+\omega Q(\widetilde{u})
\end{align}
and for later use we introduce also
\begin{align}\label{eq:def d}
d(\omega):=S_\omega(\widetilde{\phi}_{\omega} ).
\end{align}
Notice that \eqref{3} is equivalent to $\nabla S_{\omega}(\widetilde \phi_\omega)=0$.
The ``linearized Hamiltonian" $\nabla^2 S_\omega(\widetilde \phi_\omega)$ will be an important operator in our analysis.
However, it lies in $\mathcal{L} ( \Hrad^1(\R^3,\widetilde{\C} ), \Hrad^{-1}(\R^3,\widetilde{\C} ))$ which is not a vector space with scalar field $\C$.
Therefore, we use another operator $H_\omega$ whose restriction in
$\Hrad^1(\R^3,\widetilde \C)$ equals $\nabla^2 S_\omega(\widetilde \phi_{\omega})$ but  which extends  to a $\C$-linear element  in $\mathcal{L} (\Hrad^{ 1}(\R ^3, \C ^2);\Hrad^{- 1}(\R ^3, \C ^2))$ and furthermore is self-adjoint in $L^2(\R^3;\C^2)$.

\begin{lemma}\label{lem:2}
We have $\left.\nabla^2 S_\omega \right|_{\Hrad^1(\R^3,\widetilde \C)} =H_\omega\in\mathcal{L} (\Hrad^1(\R^3,\widetilde{\C} ), \Hrad^{-1}(\R^3,\widetilde{\C} ))$, where
\begin{align}\label{36}
H_\omega:=-\Delta+ \omega +g(\phi_\omega ^2)+g'(\phi_\omega^2)\phi_\omega^2
+  g'(\phi_\omega ^2)\phi_\omega^2 \sigma_1.
\end{align}
\end{lemma}

\begin{proof}
Differentiating \eqref{33}, we have
$
\nabla^2 S_\omega(\widetilde{u})(\cdot)=-\Delta+ \omega+g(|\widetilde u|^2)\cdot + 2g'(|\widetilde u|^2)\< \sigma _1\widetilde{u},\cdot\>_{\C^2}\widetilde{u}$.
Thus, substituting, $\widetilde u=\widetilde \phi_\omega$, we have
\begin{align}\label{36.6}
\nabla^2 S_\omega(\widetilde \phi_\omega)(\cdot)= \(-\Delta+ \omega +g( \phi_\omega ^2)\)
+2g'( \phi_\omega ^2)\<\widetilde{\phi}_\omega,\cdot\>_{\C^2}\widetilde{\phi}_\omega.
\end{align}
Since for $\widetilde{r}={^t( r  \ \bar r )} $  we have
\begin{align*}
\<\widetilde{\phi}_\omega,\widetilde{r}\>_{\C^2}\widetilde{\phi}_\omega=
\frac 1 2 \phi_\omega\Re \( \bar r +  r\)
\begin{pmatrix} \phi_\omega \\  \phi_\omega \end{pmatrix}=
\frac 1 2  \phi_\omega^2
\begin{pmatrix}  1 & 1 \\ 1 &  1 \end{pmatrix}
\begin{pmatrix} r \\ \bar r
\end{pmatrix}=\frac 1 2 (1+\sigma_1)\phi_\omega^2 \widetilde{r},
\end{align*}
we obtain $ \left.\nabla^2 S_\omega \right|_{\Hrad^1(\R^3,\widetilde \C)} =H_\omega$.
\end{proof}

We now introduce the linearized operator $\mathcal H_\omega$.
\begin{definition}
We set
\begin{align}\label{39}
\mathcal H_\omega &:= \sigma_3 H_\omega =\sigma_3 (-\Delta+ \omega )   +\mathcal{V}(\omega)  \text{ with } \mathcal{V}(\omega):=(g(\phi_\omega ^2)+g'(\phi_\omega^2) \phi_\omega^2)\sigma_3
+  g'(\phi_\omega ^2)\phi_\omega^2 \im\sigma_2 .
\end{align}
Notice that $\mathcal H_\omega$   with domain $\Hrad^2(\R ^3, \C^2)$ defines a closed $\C$--linear operator in $\Lrad(\R ^3, \C^2)$.
\end{definition}

\begin{remark}
When we substitute $\widetilde u = e^{\im \omega t \sigma_3}\(\widetilde \phi_\omega+\widetilde r\)$, we obtain
\begin{align*}
\im \partial_t \widetilde r =\mathcal H_\omega \widetilde r + ``\text{higher order of  $\widetilde r$}".
\end{align*}
Thus, actually $-\im \mathcal H_\omega$ is the linearized operator.
However, we adopt this terminology because we would like to handle the linearized equation as a Schr\"odinger equation.
\end{remark}
Since $\mathcal V(\omega)$ decreases  rapidly  at spatial infinity (see Lemma \ref{lem:4.3} below), we have $\sigma_{\mathrm{ess}}(\mathcal H_{\omega})=(-\infty,-\omega]\cup [\omega,\infty)$, where $\sigma_{\mathrm{ess}}(\mathcal H_{\omega})\subset \C$ is the set of essential spectrum of $\mathcal H_{\omega}$.
For the discrete spectrum $\sigma_{\mathrm{d}}(\mathcal H_{\omega})$, we know that $0\in \sigma_{\mathrm{d}}(\mathcal H_{\omega})$ because we have
\begin{align}\label{41}
\mathcal H_\omega \im \sigma_3 \widetilde{\phi}_\omega=0,\quad -\im \mathcal H_\omega \partial_\omega \widetilde{\phi}_\omega=\im \sigma_3 \widetilde{\phi}_\omega.
\end{align}
We make the following assumptions on the spectrum of  $\mathcal H_{\omega_*}$.
\begin{assumption}\label{a:1}
We assume
\begin{enumerate}
\item
$\sigma_{\mathrm{d}}(\mathcal H_{\omega_*})=\{0\}$ and $\mathcal N(\mathcal H_{\omega_*})=\mathrm{span}\{\im \sigma_3 \widetilde \phi_{\omega_*}\}$.
\item
$\mathcal H_{\omega_*}$ has no edge resonance nor embedded eigenvalues.
\end{enumerate}
\end{assumption}

\begin{remark}
The assumption $\mathrm{dim}\mathcal N(\mathcal H_{\omega_*})=1$ is equivalent to the standard assumption (see, e.g.\ \cite{Weinstein85,Weinstein86,GSS})
$\mathcal N(\left.L_+\right|_{\Lrad(\R^3,\C)})=\{0\}$, where $L_{+} :=-\Delta + \omega_*
+g(\phi_{\omega_*}^2)
+2g'(\phi_{\omega_*}^2)\phi_{\omega_*}^2$.
Notice that $L_+ \partial_{x_j}\phi_{\omega_*}=0$ but $\partial_{x_j}\phi_{\omega_*}\not\in \Lrad(\R^2;\C)$.
The hypothesis $\sigma_{\mathrm{d}}(\mathcal H_{\omega_*})=\{0\}$ is here assumed for simplicity, but we expect to be able to relax it, by allowing   the existence of internal modes that we should be able to treat like in \cite{Cuccagna11CMP}.
The 2nd assumption is   technical and is   needed to get  dispersive estimates.
The nonexistence of edge resonance is expected to hold generically and the   nonexistence of embedded eigenvalue is conjectured to be true.
\end{remark}

\noindent By Assumption \ref{a:1}, we can bootstrap the regularity of the map $\omega\mapsto \widetilde{\phi}_\omega$. Specifically, we can show:
\begin{lemma}\label{lem:4.3}
Under  Assumption \ref{a:1}, there exists $\delta>0$ s.t.\ the map $\omega\mapsto\phi_{\omega}$ is $C^\infty$ in $\Sigma^s(\R^3,\R)$ for all $s\geq 0$.
\end{lemma}
Since the proof of Lemma \ref{lem:4.3} consists of standard elliptic regularity argument, we omit it.

%
%


We set
\begin{align}\label{Psi12}
\Psi_1(\omega):=\im \sigma_3 \widetilde{\phi}_\omega,\quad \Psi_2(\omega):=\partial_\omega \widetilde{\phi}_\omega.
\end{align}

The following proposition is due to Comech--Pelinovsky \cite{CP03CPAM}.

\begin{proposition}\label{prop:1}
For $s\geq 0$, there exists $\delta_s>0$ s.t.\ for $\omega\in D_\R(\omega_*,\delta_s)$, there exist $\Psi_j(\omega)\in \Sigma^s(\R^3,\widetilde \C)$ ($j=3,4$) satisfying the following properties.
\begin{enumerate}
\item
For $1\leq j\leq 4$, $\Psi_j\in C^\infty(D_R(\omega_*,\delta_s),\Sigma^s)$.
\item
We have $\mathcal N_g(\mathcal H_{\omega_*})=\mathrm{span}\{\Psi_j(\omega_*)\ |\ j=1,2,3,4\}$.
\item
We have $A(\omega):=\Omega(\Psi_2(\omega),\Psi_3(\omega))>0$.
\item
There exists $$P_d\in \cap_{s\geq 0} C^\infty(D_\R(\omega_*,\delta_s),\mathcal L(\Sigma^{-s}(\R^3,\C^2),\Sigma^{s}(\R^3,\C^2))),$$ s.t.\ $P_d(\omega)$ is a projection symmetric w.r.t.\ $\Omega$, i.e.\ $\Omega(P_d(\omega)\cdot,\cdot)=\Omega(\cdot,P_d(\omega)\cdot)$ (recall \eqref{17.6d}). The projection $P_d(\omega)$ commutes with $\pi_0$ and $\mathcal H_{\omega}$ and satisfies $\mathcal N_g(\mathcal H_{\omega})=\mathrm{Ran}P_d(\omega)$.
\item
The operator $-\im \mathcal H_\omega P_d(\omega)$ can be represented as
\begin{align}\label{46}
-\im\mathcal H_\omega P_d(\omega)=
\begin{pmatrix}
0 & 1 & 0 & 0\\
0 & 0 & 1 & 0\\
0 & 0 & 0 & 1\\
0 & 0 & a(\omega) & 0
\end{pmatrix}
\end{align}
in the frame $\{\Psi_j(\omega),\ j=1,2,3,4\}$, where $a(\omega)=-A(\omega)^{-1}q'(\omega)$.
\item
We have
\begin{equation}
\begin{aligned}
\Omega(\Psi_1(\omega),\Psi_2(\omega))=-q'(\omega), \Omega(\Psi_1(\omega),\Psi_3(\omega))=0,\quad \Omega(\Psi_1(\omega),\Psi_4(\omega))=-A(\omega),\\
\Omega(\Psi_2(\omega),\Psi_4(\omega))=0,\quad \Omega(\widetilde \phi_{\omega},\Psi_2(\omega))=\Omega(\widetilde \phi_{\omega},\Psi_4(\omega))=0.\label{50}
\end{aligned}
\end{equation}
\end{enumerate}
\end{proposition}

\begin{remark}
The matrix in \eqref{46} has eigenvalues $0, \pm\sqrt{a(\omega)}$.
\end{remark}

We will give the proof of Proposition \ref{prop:1} in the appendix of this paper.

\subsection{Modulation coordinates}
In this subsection, we give the first choice of the coordinates related to the linearized operator $\mathcal H_{\omega}$.
We set
\begin{align}\label{55}
P (\omega):=1-P_d(\omega).
\end{align}

We set $\mathcal T_{\omega}=\{e^{\im \theta \sigma_3}\widetilde{\phi}_{\omega}\ |\ \theta\in\R\}$ and recall \eqref{def:tubnbd}.
\begin{lemma}\label{lem:mod1}
For any $s\ge 0 $, there exists $\delta_s>0$ s.t.\ there exist $\vartheta,\omega, \lambda,\mu \in C^\infty (D_{\Sigma^{-s}}(\mathcal T_{\omega_*},\delta_s), \R)$ s.t.\ for $\tilde u \in D_{\Sigma^{-s}}(\mathcal T_{\omega_*},\delta_s)$,
\begin{align}\label{61}
  R(\widetilde{u}):=e^{-\im \vartheta(\widetilde{u})\sigma_3}\widetilde{u} -\widetilde{\phi}_{\omega(\widetilde{u})}-\lambda(\widetilde{u})\Psi_3(\omega(\widetilde{u}))
  -\mu(\widetilde{u})\Psi_4(\omega (\widetilde{u})) \in P (\omega(\widetilde{u}))\Sigma_{-s}.
\end{align}
Furthermore, $\Sigma^s$ can by replaced by $\Hrad^s$.

\end{lemma}

\begin{proof}
Set $  R(\widetilde{u},\vartheta,\omega,\lambda,\mu):=e^{-\im \vartheta\sigma_3}\widetilde{u} -\widetilde{\phi}_{\omega}-\lambda\Psi_3(\omega)-\mu \Psi_4(\omega)$ and
\begin{align}\label{64}
  \mathrm{F}_j(\widetilde{u},\vartheta,\omega,\lambda,\mu)=\Omega(  R(\widetilde{u},\vartheta,\omega,\lambda,\mu),\Psi_j(\omega)).
\end{align}
Then, by \eqref{50}, we have
\begin{align}\label{67}
\left.\frac{\partial (\mathrm{F}_1, \mathrm{F}_2, \mathrm{F}_3, \mathrm{F}_4)}{\partial(\vartheta,\omega,\lambda,\mu)}\right|_{(\widetilde{u},\vartheta,\omega,\lambda,\mu)=(e^{\im \theta\sigma_3}\widetilde\phi_{\omega_*},\theta,\omega_*,0,0)}&=\(\Omega\(\Psi_i(\omega_*),\Psi_j(\omega_*)\)\)_{1\leq i,j\leq 4}\nonumber\\&
=\begin{pmatrix}
0 & 0 & 0 & -A(\omega_*)\\
0 & 0 & A(\omega_*) & 0 \\
0 & -A(\omega_*) & 0 & B(\omega_*)\\
A(\omega_*) & 0 & -B(\omega_*) & 0
\end{pmatrix},
\end{align}
where $A(\omega)$ is given in 3.\ of Proposition \ref{prop:1} and
\begin{align}\label{70}
B(\omega):=\Omega(\Psi_3(\omega),\Psi_4(\omega)).
\end{align}
Since the determinant of the matrix given in \eqref{67} is $A(\omega_*)^4>0$, it is invertible.
Therefore, since $\mathrm{F}_j(e^{\im \theta \sigma_3}\widetilde{\phi}_{\omega_*},\theta,\omega_*,0,0)=0$, we have the claim by the implicit function theorem.
\end{proof}

\begin{proposition}\label{prop:31}
Let $\delta_1>0$ given in Lemma \ref{lem:mod1} for $\Hrad^1$.
Set
\begin{align*}
\mathcal F_1\in &C^\infty(\R\times  D_{\R^3\times P(\omega_*)\Hrad^1(\R ^3, \widetilde{\C} )}((\omega_*,0,0,0),\delta_1), \Hrad^1(\R ^3, \widetilde{\C} )):\\&(\vartheta,\omega,\lambda,\mu, {\widetilde{r}})\mapsto \mathcal F_1(\vartheta,\omega,\lambda,\mu, {\widetilde{r}}):=e^{\im \vartheta \sigma_3}(\widetilde{\phi}_\omega+\lambda\Psi_3(\omega)+\mu \Psi_4(\omega)+P (\omega) {\widetilde{r}}).
\end{align*}
Then, $\mathcal F_1$ is a $C^\infty$ diffeomorphism onto a neighborhood of $\mathcal T_{\omega_*}$ in $\Hrad^1$.

\end{proposition}

\begin{proof}
It suffices to show $\mathcal F_1^{-1}$ exists and it is $C^\infty$.
To show it, since we already have the map $\tilde u \mapsto (\vartheta(\tilde u), \omega(\tilde u), \lambda(\tilde u),\mu(\tilde u), R(\tilde u))$, it suffices to replace $R(\tilde u) \in P(\omega(\tilde u))\Hrad^1$ by $\tilde r(\tilde u)\in P(\omega_*)\Hrad^1$  if the map
\begin{equation}\label{eq:prop:31}
   \left . P(\omega(\tilde u)) \right | _{P(\omega_*)\Hrad^1} =\left . \( 1 + P(\omega(\tilde u))-P(\omega_*) \) \right | _{P(\omega_*)\Hrad^1}: P(\omega_*)\Hrad^1\to P(\omega )\Hrad^1
\end{equation}
is an isomorphism. For   $\tilde u$ sufficiently close to $\mathcal T_{\omega_*}$ then $\| P(\omega(\tilde u))-P(\omega_*)\|_{\mathcal L(\Hrad^1)}<1$ and so the operator
    $1 + P(\omega(\tilde u))-P(\omega_*)$ is invertible
in the whole of $\Hrad^1$, and in particular the map in \eqref{eq:prop:31}
is invertible with  $$\left.(1+P(\omega(\tilde u))-P(\omega_*))^{-1}\right|_{P(\omega(\tilde u))\Hrad^1}=(\left.P(\omega(\tilde u))\right|_{P(\omega_*) \Hrad^1)})^{-1}.$$
Therefore, if we set $\tilde r(\tilde u):= \(\left.P(\omega(\tilde u)\)\right|_{P(\omega_*) \Hrad^1)})^{-1} R(\tilde u) $, we will have the desired property.
\end{proof}

If $q'(\omega)$ is bounded away from $0$ it is natural to    replace $\omega$ by $Q$,
see \cite{Cuccagna12Rend},
since by the conservation of $Q$   this reduces the number of unknowns in the system.
However,   in a neighborhood of $\omega=\omega_*$ this is not possible  and,
instead, we replace $\mu$ by $Q$.

\begin{lemma}\label{lem:defmu}
For any $s\geq 0$, there exists $\delta_s>0$ s.t.\ there exists $$\mu\in \cap_{s\geq 0}C^\infty(D_{\R^4\times \Sigma^{-s}_c}((q(\omega_*),\omega_*,0,0,0),\delta_s)),\R)$$
s.t.\ $\mu=\mu(Q,\omega,\lambda,\rho,\widetilde r)$ satisfies
\begin{align}\label{Qmu}
Q=Q(\mathcal F_1(\vartheta,\omega,\lambda,\mu,\widetilde r))\text{ for }\rho = Q(\widetilde r).
\end{align}
\end{lemma}

\begin{remark}
The additional variable $\rho$ is introduced to define $\mu$ for $\Sigma^s$ with $s<0$.
In the following, we will always substitute $\rho = Q(\widetilde r)$.
\end{remark}

\begin{proof}
We define $\mu=\mu(Q,\omega,\lambda,\rho,\widetilde r)$ by the implicit function theorem to be the solution of
\begin{align}\label{def:mu}
\mathcal G (\mu,Q,\omega,\lambda,\rho,\tilde r):=&Q(\widetilde \phi_\omega+\lambda\Psi_3(\omega)+\mu \Psi_4(\omega)-P_d(\omega)\widetilde{r})\\&+\<\widetilde \phi_\omega+\lambda\Psi_3(\omega)+\mu \Psi_4(\omega)-P_d(\omega)\widetilde{r},\sigma_1\widetilde r\>+\rho-Q=0.\nonumber
\end{align}
Since $\mathcal G (0,q(\omega_*),\omega_*,0,0,0)=0$ and
$$\left.\partial_\mu \mathcal G \right|_{(\mu,Q,\omega,\lambda,\rho,\widetilde r)=(0,q(\omega_*),\omega_*,0,0,0)}=\<\widetilde \phi_{\omega_*},\sigma_1 \Psi_4(\omega_*)\>=A(\omega_*)\neq 0,$$
 we have the assumptions of implicit function theorem.
From \eqref{def:mu}, it is clear that we have \eqref{Qmu}.
\end{proof}

By using $\mu=\mu(Q,\omega,\lambda,Q(\widetilde r),\widetilde r)$ we can move to coordinates $(\vartheta,Q,\omega,\lambda,\widetilde r)$ instead of the coordinates $(\vartheta, \omega,\lambda,\mu,\widetilde r)$.

\begin{proposition}\label{prop:32}
Consider the $\delta_1>0$   in Lemma \ref{lem:mod1} for $\Hrad^1$.
For $\mathcal F_1$ the  function in Proposition \ref{prop:31}, set
\begin{align*}
\mathcal F_2\in &C^\infty(\R\times  D_{\R^3\times P(\omega_*)\Hrad^1(\R ^3, \widetilde{\C} )}((q(\omega_*),\omega_*,0,0),\delta_s), \Hrad^1(\R ^3, \widetilde{\C} )):\\&(\vartheta,Q,\omega,\lambda, {\widetilde{r}})\mapsto \mathcal F_2(\vartheta,Q,\omega,\lambda, {\widetilde{r}}):=\mathcal F_1(\vartheta,\omega,\lambda,\mu(Q,\omega,\lambda,Q(\widetilde r),\widetilde r), {\widetilde{r}}).
\end{align*}
Then, $\mathcal F_2$ is a $C^\infty$ diffeomorphism onto a neighborhood of $\mathcal T_{\omega_*}$ in $\Hrad^1$.
\end{proposition}

\subsection{Symbol for errors}
Even though we have replaced $\mu$ by $Q$, as a dependent variable  $\mu$ will appear  frequently.
We define some classes of functions to handle various error terms.

\begin{definition}\label{def:symbols}
For $n=0,1,2$ and $s\geq 0$, we set,
\begin{align*}
  \widehat{\mathcal S}^n_s(\delta):=  C^\infty(D_{\R^5\times P(\omega_*)\Sigma^{-s}(\R ^3, \widetilde{\C} )}((q(\omega_*),\omega_*,0,0,0,0),\delta),X ^n_s ),
\end{align*}
where $X ^0_s:=\R$, $X^1_s:=\Sigma^s(\R^3,\widetilde \C)$ and $X^2_s:=
 \mathcal L\(\Sigma^{-s}(\R ^3, \widetilde{\C} ), \Sigma^s (\R ^3, \widetilde{\C} )\)$.
We define $\mathcal S^n$ by
\begin{align*}
f\in \mathcal S^n \Leftrightarrow \forall s\geq 0,\ \exists \delta_s>0  \text{ s.t.\ } \exists  \widehat{f} \in \cap_{s\geq 0}\widehat{\mathcal S}^n_s(\delta_s) \text{ s.t.} \\ f(Q,\omega,\lambda ,\widetilde{r})= \widehat{f}(Q,\omega,\lambda ,\mu(Q,\omega,\lambda,Q(\widetilde r),\widetilde r),Q(\widetilde{r}),\widetilde{r}).
\end{align*}
Moreover, we define $\mathcal S^n_{i,j}\subset \mathcal S^n$ by $f\in \mathcal S^n_{i,j}$ if and only if $\widehat{f}(Q,\omega,\lambda ,\mu,\rho,\widetilde{r})$ satisfies
\begin{align*}
\|\widehat{f}\|_{X ^n_s}\lesssim_s (|Q-q(\omega_*)|^{1/2}+|\omega-\omega_*|+|\lambda|+|\mu|+\rho^{1/2}+\|\widetilde{r}\|_{\Sigma^{-s}})^i (|\lambda|+|\mu|+\|\widetilde{r}\|_{\Sigma^{-s}})^j.
\end{align*}
We will always use $S^n_{i,j}$ to be a generic  element of $\mathcal S^{n}_{i,j}$.
We also write formulas like
$S^n_{i,j}=S^n_{i,j}+S^n_{i+1,j}$ where $S^n_{i,j}$ in the l.h.s.\ and r.h.s.\ are different (so the equation does not imply $S_{i+1,j}^n=0$).

 \noindent We further define $s^1_{i,j}$ similarly to $S^1_{i,j}$ but replacing
 $ \Sigma^s(\R^3,\widetilde \C)$ in the definition of  $X^1_s $
    with $ \Sigma^s(\R^3,  \R )$.
\end{definition}


In the sequel, other  error terms will appear in the energy expansion which are not of the form $S^0_{i,j}$.
These terms  are essentially like $\int V(s)r^n\,dx$,
  with $r^n$ an $n$--linear form, for and $n\geq 2$, from $\widetilde{\C}$ to $\C$ and
 with $V\in \mathcal S$.
Before introducing these class, we need some preliminary notation.
\begin{definition}\label{def:beta}
Let $\beta_{-1}(s)=1$ and $\beta_n(s)= g^{(n)}(s)$ for $n\geq 0$.
We will write $c g^{(n)}(s)=\beta_n(s)$,   ignoring constants.
With this convention we write $2g'(s)=\beta_1(s)$ and obviously this does not mean
that $2\beta_1(s)=2g'(s)=\beta_1(s)$ implies $\beta_1(s)=0$, because here
the $\beta_1$'s are not the same function.
We extend further the use of the symbol $\beta_n$ as follows. If we set
 \begin{align}\label{eq:beta-1}T_{S^1_{i,j},k}\beta _n(S^1_{0,0},\widetilde r):=\int_0^1 t^k \beta _n(|S^1_{0,0}+t\widetilde r +t S^1_{i,j}|^2)\,dt,\end{align}for $k\geq 0$ and $i,j\geq 0$,
 and if more generally we consider the composition
   \begin{align}\label{eq:beta1} T_{S^1_{i_1,j_1},k_1}T_{S^1_{i_2,j_2},k_2}\cdots T_{S^1_{i_N,j_N},k_N} \beta _n(S^1_{0,0},\widetilde r)  \end{align}
   we will denote the expression  \eqref{eq:beta1}  by $\beta _n(|S^1_{0,0}+ \widetilde r |^2)$.
\end{definition}

\begin{remark}
In Definition \ref{def:beta} we are committing an abuse of notation,
which is however harmless in terms of the estimates or of qualitative properties
like the smoothness and the decay/increasing property.
Indeed the r.h.s.\ of
\eqref{eq:beta-1}, or a more general multiple integral representing \eqref{eq:beta1},
satisfies the same bounds of $\beta _n(|S^1_{0,0}+ \widetilde r |^2)$.
In particular, for $n\geq0$ using our notation we have
\begin{equation}\label{expandbeta}
\begin{aligned}
\beta_n(|S^1_{0,0}+\widetilde r|^2)&=\beta_n(|S^1_{0,0}|^2)+\int_0^1 2\beta_{n}'(|S^1_{0,0}+t\widetilde r|^2)\<S^1_{0,0}+t\widetilde r, \sigma_1 \widetilde r\>\,dt\\&=S^0_{0,0}+\beta_{n+1}(|S^1_{0,0}+\widetilde r|^2)\<S^1_{0,0},\sigma_1\widetilde r\>_{\C^2}+\beta_{n+1}(|S^1_{0,0}+\widetilde r|^2)|\widetilde r|^2.
\end{aligned}
\end{equation}
\end{remark}

With  the above notation we define the classes $\mathbf R_k$.

\begin{definition}\label{def:Rk}
Let $\beta_n$ given in definition \eqref{def:beta}.
\begin{itemize}
\item
By $\mathbf R_2$, we   mean $\R$ valued functions consisting of finite sums of $\<S^2_{1,0}\widetilde r,\sigma_1\widetilde r\>$ and terms in the form
\begin{align}\label{r2}
\int_{\R^3}\beta_n(|S^1_{0,0}+\widetilde r|^2)s^1_{1,0}|\widetilde r|^2\,dx,\quad \int_{\R^3}\beta_n(|S^1_{0,0}+\widetilde r|^2)\<S^1_{1,0},\sigma_1\widetilde r\>_{\C^2}\<S^1_{0,0},\sigma_1\widetilde r\>_{\C^2}\,dx,
\end{align}
with $n\geq -1$.
\item
By $\mathbf R_k$ ($k=3,4,5$), we mean   $\R$ valued functions consisting of finite sums of
\begin{align}\label{r345}
\int_{\R^3}\beta_n(|S^1_{0,0}+\widetilde r|^2)\<S^1_{0,0},\sigma_1 \widetilde r\>_{\C^2}^i|\widetilde r|^{2j}\,dx,\quad i+2j=k,
\end{align}
with $n\geq k-2$.
\item
By $\mathbf R_6$, we mean    $\R$ valued functions consisting of finite sums of
\begin{align}\label{r6}
\int_0^1\int_0^1(1-t)^2\int_{\R^3}g'''(|v(t,s)|^2)\<v(t,s),\sigma_1 S^1_{0,0}\>_{\C^2}\<v(t,s),\sigma_1 S^1_{1,1}\>_{\C^2}\<v(t,s),\sigma_1 \widetilde r\>_{\C^2}^2\,dxdsdt,
\end{align}
with $v(t,s)=sS^1_{0,0}+t (\widetilde r+S^1_{0,0})$.
\item
By $\mathbf R_7$, we express an $\R$ valued function consisting of finite sum of
\begin{align}\label{R7}
\mathbf R_7=\frac12\int_0^1\int_0^1
(1-t)^2 \<\nabla^4 E_P(v(t,s))(S^1_{0,0},\widetilde r, \widetilde r),\sigma_1 \widetilde r\> dsdt
\end{align}
with  $v(t,s)=sS^1_{0,0}+t (\widetilde r+S^1_{0,0})$.
\end{itemize}

\end{definition}


\subsection{Hamilton structure of NLS and symplectic form}\label{subsec:Hamstruc}

We introduce the notion of symplectic form.

\begin{definition}
Let $X$ be a Banach space on $\R$ and let $X'$ be its dual.
A strong symplectic form is a  $2$-form $\tilde\Omega$ on $X$ s.t.\ $d \tilde \Omega=0$ (i.e.\ $\tilde \Omega$ is closed) and s.t.\ the map $X\ni x\to \tilde \Omega(x,\cdot)\in X'$
is an isomorphism.
\end{definition}
Setting $\Gamma(\widetilde u):=\frac12 \Omega(\widetilde u,\cdot)$, we have $d \Gamma=\Omega$ (recall that given two vector fields $X$ and $Y$ we have $d\Gamma (X,Y):= X \Gamma (Y)- Y \Gamma (X)-  \Gamma ([X,Y])$, see \cite{AMRBook}).
Therefore, the skew-symmetric form $\Omega$ defined in \eqref{17.6d} is a strong symplectic form on $\Hrad^j(\R^3,\widetilde{\C} )$ for $j=0,1$.

\noindent
Notice that   $\Omega (\widetilde{u},\widetilde{v})= \Im \int _{\R ^3}u(x)\overline{v}(x)dx$ and that $\Omega $ extends naturally in a symplectic form on $\Hrad^j(\R^3, {\C} ^2)$ for $j=0,1$.

\begin{definition}
Let $\tilde \Omega$ be a strong symplectic form in $\Hrad^1(\R^3,\widetilde{\C} )$  and $F,G\in C^1( \Hrad^1(\R^3,\widetilde{\C} ) ,\R)$.
We define the Hamiltonian vector field $X_F(\widetilde{u})$ with respect to $\tilde \Omega$
by
\begin{align*}
\tilde\Omega(\widetilde{u})(X_F(\widetilde{u}),\widetilde{y}):=dF(\widetilde{u})\widetilde{y}=\<\nabla F(\widetilde u),\sigma_1 \widetilde y\> \text{ for all $\widetilde{y}\in \Hrad^1(\R^3,\widetilde{\C} )$.}
\end{align*}
\end{definition}

For the case $\tilde \Omega=\Omega$ defined in \eqref{17.6d}, we have $X_F(\widetilde{u})=-\im \sigma_3 \nabla F(\widetilde{u})$.

The Hamiltonian vector field of $E$ with respect to $\Omega$     is
\begin{align}\label{18}
X_E(\widetilde{u})=-\im \sigma_3 \nabla E=-\im \sigma_3\( -\Delta \widetilde{u} + g(  |\widetilde{u}|^2)\widetilde{u}\)\in \Hrad^{-1}(\R^3,\widetilde{\C}) .
\end{align}
Thus, we see that \eqref{1} can be rewritten in the Hamiltonian formulation  as
\begin{align}\label{24}
\partial_t {\widetilde{u}}  = X_E(\widetilde{u}).
\end{align}
The symplectic form $\Omega$     looks simple.
However, when   expressed  in terms of   the modulation coordinates,     $\Omega$ appears very complicated and full of cross terms (see \eqref{Xs7} below).
For this reason, we introduce a new symplectic form $\Omega_0:= d \Gamma_0$, where
\begin{align}
\Gamma _0( \widetilde{u} ):=&   Q d \theta
-
\lambda  A(\omega)d \omega + \lambda B(\omega) \partial_Q \mu \ dQ
  + \frac 1 2 \Omega(\widetilde{r},d\widetilde{r}) ,\label{eq:gamma1}
\end{align}
For $A(\omega)$  defined in Prop. \ref{prop:1} and $B(\omega)$  defined in
\eqref{70}.

\noindent Using the notation in Definition \ref{def:symbols}, we have
\begin{align}\label{def:Omega0}
\Omega_0=  dQ\wedge d\theta +A(\omega)d \omega \wedge d \lambda +  \Omega(d\widetilde{r},d\widetilde{r})   +  \(S^0_{0,0}d \lambda + S^0_{0,1}d \omega + \<S^0_{0,1}\widetilde{r}+S^1_{0,1},\sigma_1 d\widetilde{r} \>\) \wedge dQ.
\end{align}
This represention of $\Omega_0$ in \eqref{def:Omega0} still may look complicated due to the 4th term.
However, the fourth term  affects   only  the dynamics of $\vartheta$, which we will ignore.
If $F=F(Q,\omega,\lambda,\widetilde{r})$ is $C^1$ and does not depend on $\vartheta$, like the Energy $E$, if we decompose its differential as
\begin{equation}\label{eq:1formdec}
    dF= \sum _{\tau =  Q ,\omega ,\lambda}\partial _\tau F d\tau + \< \nabla _{\widetilde{r}}F, \sigma _1 d\widetilde{r} \> ,
\end{equation}
where $\nabla  _{\widetilde{r}}F\in \text{Ran}P(\omega_*)^*$ is defined by the above formula, then
the Hamiltonian vector field of $F$ associated to $\Omega_0$ satisfies
\begin{align*}
&\Omega_0(X_F,Y)=\partial_Q F (Y)_Q+\partial_\omega F (Y)_\omega+ \partial_\lambda F (Y)_\lambda + \<\nabla_{\widetilde{r}} F,\sigma _1(Y)_{\widetilde{r}}\>\\&
= (X_F)_Q (Y)_\vartheta - (X_F)_\vartheta (Y)_Q + A(\omega) (X_F)_\omega (Y)_\lambda- A(\omega)(X_F)_\lambda (Y)_\omega +   \Omega((X_F)_{\widetilde{r}},(Y)_{\widetilde{r}})\\&\quad
+\(S^0_{0,0}(X_F)_\lambda+S^0_{0,1} (X_F)_\omega +\<S^0_{0,0}{\widetilde{r}}+S^1_{0,1},(X_F)_{\widetilde{r}}\> \)(Y)_Q\\&\quad-
\(S^0_{0,0}(Y)_\lambda+S^0_{0,1} (Y)_\omega +\<S^0_{0,0}{\widetilde{r}}+S^1_{0,1},(Y)_{\widetilde{r}}\>\) (X_F)_Q,
\end{align*}
where $(X)_\tau=d\tau X$ for $\tau=\vartheta,Q,\omega,\lambda,\widetilde r$.
Therefore, we have the following simple formulas for $X_F$
\begin{align}\label{eq:HamVector}
(X_F)_Q=0,\quad (X_F)_\omega=A(\omega)^{-1}\partial_\lambda F,\quad (X_F)_\lambda=-A(\omega)^{-1}\partial_\omega F,\quad
(X_F)_{\widetilde{r}}=-\im P(\omega_*) \sigma_3 \nabla_{\widetilde{r}} F,
\end{align}
and a more complicated formula only for the $\vartheta$--th component
\begin{align}\label{eq:HamVector1}
(X_F)_\vartheta =-\partial_Q F +S^0_{0,0}(X_F)_\lambda+S^0_{0,1} (X_F)_\omega +\<S^0_{0,0}{\widetilde{r}}+S^1_{0,1},(X_F)_{\widetilde{r}}\>.
\end{align}
Since a vector can be decomposed as
\begin{align}&
X=e^{\im \vartheta \sigma_3}\(\im \sigma_3 \underline{u}  (X) _{\vartheta} + \partial_\omega \underline{u}  (X) _ \omega   + \Psi_3 (X) _ \lambda    + \Psi_4  (X) _  \mu   + P  (X)  _{\widetilde{r}}  \)  \text{ where}\nonumber\\& (X) _{\tau}:=d\tau X  \text{ for }\tau = \vartheta ,\omega , \lambda ,\mu , \widetilde{r} \text{ and }  \underline{u}:=\widetilde{\phi}_\omega + \lambda  \Psi_3+ \mu \Psi_4+P \widetilde{r},\label{eq:vfieldcomp}
\end{align}
we can recover $X_F$ from $(X_F)_\tau$ for $\tau=\vartheta,Q,\omega,\lambda,\widetilde r$ given in \eqref{eq:HamVector} and \eqref{eq:HamVector1}.

\subsection{Main theorems}\label{subsec:main}
We are now in the position to state are main theorems precisely.
First, we will give   appropriate coordinates which will be a modification of the modulation coordinate given in Proposition \ref{prop:3}.
These coordinates will transform the symplectic form $\Omega$ into $\Omega_0$ defined in \eqref{def:Omega0} and will decouple the continuous and discrete coordinate to arbitrary high order. The proof will be given in Sect.\ \ref{sec:normal}.


\begin{theorem}\label{thm:normal}
For any $N\geq 2$, there exist $\delta_N>0$ and a $C^\infty$--diffeomorphism
onto a neighborhood of $\mathcal T_{\omega_*}$   in $\Hrad^1$ (recall $\mathcal T_{\omega_*}=\{e^{\im \theta \sigma_3}\widetilde \phi_{\omega}\ |\ \theta\in\R\}$),
$\mathfrak F_N \in C^\infty(D_{\Hrad^1}(\mathcal T_{\omega_*},\delta_N),\Hrad^1(\R^3,\widetilde \C))$,   s.t.
\begin{align}\label{error1}
\mathfrak F_N^* Q=Q,\ \mathfrak F_N^*\omega=\omega+S^0_{0,1},\ \mathfrak F_N^* \lambda=\lambda+S^0_{1,1},\ \mathfrak F_N^* \widetilde r = e^{\im S^0_{0,1} \sigma_3}(\widetilde r + S^1_{1,1})
\end{align}
and
\begin{align}\label{Darb}
\mathfrak F_N^* \Omega=\Omega_0,
\end{align}
where $\mathfrak F_N^*$ is the pull-back operator.
Further, for $E_N :=\mathfrak F_N^* E(= E\circ \mathfrak F_N)$, we have
\begin{equation}\label{eq:mmain1} \begin{aligned}
 E_N =  E_{f,N}(Q,\omega, \lambda,Q(\widetilde r))
      +\frac 1 2\<H _{\omega _*}  \widetilde{r}, \sigma _1  \widetilde{r}\>+E_P(\widetilde{r})
   + S_{0, N +1}^0+ \sum_{k=2}^7 \mathbf R_k,
\end{aligned}\end{equation}
where $\mathbf R_k$ are terms are defined in Definition \ref{def:Rk} and
\begin{align}\label{eq:fenergy}
E_{f,N}(Q,\omega, \lambda,Q(\widetilde r))=d (\omega ) -\omega Q+(\omega-\omega_*)Q(\widetilde r) + \frac{1}{2} A (\omega ) \lambda ^2     +\sum  _{m=2}^{N}   \lambda ^m     e _{m,0}^{(N)}(Q, \omega , Q(\widetilde{r}))
\end{align}
with $e_{2,0}^{(N)}=S^0_{1,0}$ and $e_{m,0}^{(N)}=S^0_{0,0}$ for $m\geq 3$ and where
 $
d(\omega)  $ was defined in \eqref{eq:def d}.
%
%
%
\end{theorem}

\begin{remark}
In \eqref{error1}, we can expand $\mathfrak F_N^*\omega=\omega+S^0_{0,1}$ as
\begin{align}\label{sigmaexpand}
\mathfrak F_N^*\omega=\omega+\frac{B(\omega)}{2A(\omega)}\mu+S^0_{1,1}=\omega-\frac{B(\omega)}{2A(\omega)^2}(q(\omega)-Q+Q(\widetilde r))+S^0_{1,1},
\end{align}
see  Lemma \ref{lem:nform1} for the 1st equality and Lemma \ref{lem:muexpand} for the 2nd.
\end{remark}

\begin{remark}
For \eqref{error1}, for example $\mathfrak F_N^* \omega =\omega+ S^0_{0,1}$, we mean that $\mathcal F_2^*\(\mathfrak F_N^* \omega-\omega\)$, initially defined in the neighborhood of $\R\times (q(\omega_*),\omega_*,0,0)$ in $\R^3\times P(\omega_*)\Hrad^1$ will not depend on $\vartheta$ and can be extended to $D_{\R^3\times P(\omega_*)\Sigma^{-s}}((q(\omega_*),\omega_*,0,0),\delta_s)$ for arbitrary $s\geq 0$ and some $\delta_s>0$ to belong in $\mathcal S^0_{0,1}$.
Actually, we will be constructing $\mathfrak F_N$ by determining $\mathcal F_2^*\(\mathfrak F_N^* \tau-\tau\)$ $(\tau=\vartheta,Q,\omega,\lambda, \widetilde r)$ in $D_{\R^3\times P(\omega_*)\Sigma^{-s}}((q(\omega_*),\omega_*,0,0),\delta_s)$ and then define $\mathfrak F_N$ by them so the there will no ambiguity for the extension.
\end{remark}

\begin{remark}
Let $v=\mathfrak F_N^{-1}(u)$.
Then, by Theorem \ref{thm:normal}, in the $v$ variable, the symplectic form $\Omega$ will be transformed to $\Omega_0$ and the energy $E$ will be transformed to $E_N$.
Thus, our coordinate will be given by $\mathcal F_2^{-1}\circ \mathfrak F_N^{-1}(u)=(\vartheta(v),Q(u),\omega(v),\lambda(v),\widetilde r(v))$.
\end{remark}

Before proceeding, we prepare another notation.
\begin{definition}\label{def:tildenabla}
For $F\in C^1(D_{\Hrad^1}(\mathcal T_{\omega_*},\delta),\R)$ written in the form $F(u)=f(Q,\omega,\lambda,Q(\widetilde r),\widetilde r)$, we define
\begin{align}\label{difrho}
\partial_\rho F = \left.\partial_\rho\right|_{\rho=Q(\widetilde r)}f (Q,\omega,\lambda,\rho,\widetilde r),
\end{align}
and
\begin{align}\label{diftilder}
\widetilde \nabla_{\widetilde r} F=\left.\nabla_{\widetilde r}\right|_{\rho=Q(\widetilde r)}f(Q,\omega,\lambda,\rho,\widetilde r).
\end{align}
In particular, we have
\begin{align}
\nabla_{\widetilde r}F = \widetilde \nabla_{\widetilde r} F + \partial_\rho F \widetilde r.
\end{align}
\end{definition}

We set $\mathbf R=S^0_{0,N+1}+\sum_{k=2}^7\mathbf R_k$, where r.h.s.\ contains  the terms   in \eqref{eq:mmain1}.
In the coordinate system given in Theorem \ref{thm:normal}, NLS \eqref{1} is reduced to
\begin{equation}\label{normeq}
\begin{aligned}
\dot \omega &= A(\omega)^{-1}\partial_\lambda E_{f,N} + \mathbf R_\omega,\\
 \dot \lambda &=- A(\omega)^{-1}\partial_{\omega}E_{f,N} + \mathbf R_\lambda,\\
\im \partial_t \widetilde r &= \(\mathcal H_{\omega_*}+\partial_{\rho}E_{N}P(\omega_*)\sigma_3\)\widetilde r + P(\omega_*)\sigma_3g(|\widetilde r|^2)\widetilde r + \mathbf R_{\widetilde r},\\& \text{with initial data $\omega_0=\omega(u_0)$, $\lambda_0=\lambda(u_0)$ and $\widetilde r_0=\widetilde r(u_0)$,}
\end{aligned}
\end{equation}
where we have ignored the motion of $\vartheta$ and
\begin{equation}
\begin{aligned}\label{Rsigmalambdar}
\mathbf R_\omega=A(\omega)^{-1}\partial_\lambda \mathbf R,\quad \mathbf R_\lambda=-A(\omega)^{-1}\partial_\omega \mathbf R,\quad \mathbf R_{\widetilde r}= P(\omega_*)\sigma_3 \widetilde \nabla_{\widetilde r}\mathbf R.
\end{aligned}
\end{equation}
Let us think of  $\widetilde r$ as a perturbation and let us set $(\omega_f, \lambda_f)$ to be the solution of
\begin{equation}\label{finiteeq}
\begin{aligned}
\dot \omega_f &= A(\omega_f)^{-1}\partial_\lambda E_{f,N}(Q,\omega_f,\lambda_f,0),\\
 \dot \lambda_f &=- A(\omega_f)^{-1}\partial_{\omega}E_{f,N}(Q,\omega_f,\lambda_f,0), \text{ with $\left.(\omega_f(t),\lambda_f(t))\right|_{t=0}=(\omega_0,\lambda_0)$.}
\end{aligned}
\end{equation}
We introduce the ``kinetic energy" and the ``potential energy"  as
\begin{align*}
K_Q(\omega,\lambda):=\frac{1}{2} A (\omega ) \lambda ^2     +\sum  _{m=2}^{N}   \lambda ^m     e _{m,0}^{(N)}(Q, \omega , 0) \text{ and } V_{Q}(\omega)=d(\omega)-\omega Q.
\end{align*}
Then, we have $K_Q(\omega,\lambda)+V_Q(\omega)=E_{f,N}(Q,\omega,\lambda,0)$.
Notice that in a neighborhood of $\mathcal T_{\omega_*}$, $K_Q(\omega,\lambda)\sim 2^{-1}A(\omega_*) \lambda^2$, which is the reason we call  $K_Q$ the ``kinetic energy".
For  $\omega_+$ a value to be defined momentarily and omitting $N$    we  denote
\begin{align}\label{def:fenergy2}
E_Q(\omega,\lambda):=E_{f,N}(Q,\omega,\lambda,0)-V_Q(\omega_+)=K_Q(\omega,\lambda)
+V_Q(\omega)-V_Q(\omega_+).
\end{align}
Since $\partial_\omega V_{Q}(\omega)=q(\omega)-Q$, we see that if $Q<q(\omega_*)$  then $V_{Q}(\omega)$ becomes monotonic increasing, which implies that $\omega_f$ will fall to left monotonically until it will be ejected out of the region where we can use the coordinate (this was essentially what Comech--Pelinovsky \cite{CP03CPAM} proved).
If $0<Q-q(\omega_*)\ll 1$, the equation $\partial_{\omega}V_{Q}=q(\omega)-Q=0$ will have two solutions $\omega_{\pm}=\omega_{Q,\pm}$ s.t.\   $\omega_-<\omega_*<\omega_+$.
Here, $\omega_-$ is a local maximum   and $\omega_+$   a local minimum of $V_{Q}$ (see Lemma \ref{lem:findimEcrt} for details).
Therefore, if $0<Q-q(\omega_*)\ll1$, $V_Q$ works as a trapping potential around $\omega=\omega_+$.

If $(\omega_f,\lambda_f)$ is trapped in the potential well $V_Q$, we have the oscillation in \eqref{finiteeq}.

\begin{proposition}\label{finitedimosc}
Fix $N\geq 2$.
There exists $\epsilon_0>0$ s.t.\ for $Q$ satisfying $0<Q-q(\omega_*)<\epsilon_0^2$ and $c\in (0,1)$,
if we take $(\omega_0,\lambda_0)$ s.t.\
\begin{align}\label{finiteenergybound}
\omega_0>\omega_- \text{ and }0< E_{Q}(\omega_0,\lambda_0)<c \(V_{Q}(\omega_-)-V_{Q}(\omega_+)\),
\end{align}
then     there exists $T_f>0$ s.t.\ the solution $(\omega_f(t),\lambda_f(t))$
 of \eqref{finiteeq} is $T_f$-periodic with $T_f\sim (Q-q(\omega_*))^{-1/4}$.
\end{proposition}

\begin{remark}
Condition \eqref{finiteenergybound} is given for $\omega_f$ to stay in the potential well.
For   energy trapping, $c=1$  would work but we need for $c$ to be a fraction of 1    to prove a uniform bound on the period $T_f$.
Indeed, if we start from $(\omega_\epsilon,0)$ with $\omega_\epsilon>\omega_+$ s.t.\ $E_{f,N}(Q,\omega_\epsilon,0,\rho_0)=V_{Q,\rho_0}(\omega_-)-\epsilon$, then $T_f$ diverges to $+\infty$ as $\epsilon \to +0$.
\end{remark}

The proof of Proposition \ref{finitedimosc} is in Section \ref{sec:findim}.
Here we remark  that, since \eqref{finiteeq} is just a two dimensional Hamilton system, it is clear that $(\omega_f(t),\lambda_f(t))$ is periodic.
So, what we only need to prove is the estimate of the period.

We now go back to the full system \eqref{normeq}.
The first observation is that since $(\omega,\lambda)\in \R^2$, the constraint of the finite part energy $E_Q=C$ gives us a $1$ dimensional torus (a circle).
Therefore, the position of $(\omega,\lambda)$ is determined by $ E_Q(\omega,\lambda)$ modulo the phase of the circle.
We will show that $E_Q$ is almost conserved for a long time. For the proof
of the following three results  see Sect.
\ref{sec:proofs fs}.

\begin{theorem}\label{thm:dynamics}
For  $N\geq 2$, $c\in (0,1)$ and $C_0>0$  there exists $\epsilon_0>0$ s.t.\  for   $Q\in (q(\omega_*),q(\omega_*)+\epsilon ^2_0)$, if $(\omega_0,\lambda_0,\widetilde r_0)$ satisfies \eqref{finiteenergybound} and $\|\widetilde r_0\|_{H^1}\leq C_0 (Q-q(\omega_*))^{3/4}$ ,   for $T:= (Q-q(\omega_*))^{-\frac32(N-1)+\frac12}$  then
 for the solution of  \eqref{normeq} we have
\begin{align}\label{rbound}
\sup_{t\in [0,T]}\|\widetilde r(t)\|_{H^1}\lesssim \ (Q-q(\omega_*))^{3/4},
\end{align}
and
\begin{align}\label{eq:almostcons}
\sup_{t\in [0,T]}\left|  E_{Q}(\omega(t),\lambda(t))- E_{Q}(\omega_0,\lambda_0)\right|\lesssim  (Q-q(\omega_*))^{2}.
\end{align}

\end{theorem}

\begin{remark}\label{rem:trapener}
By elementary computation, we see $V_{Q}(\omega_-)-V_{Q}(\omega_+)\sim (Q-q(\omega_*))^{3/2}$ (see \eqref{potentialhigh} below).
Thus, for $Q-q(\omega_*) >0$ small enough, the   $(\omega(t),\lambda(t))$
{Theorem} \ref{thm:dynamics}
are still in the potential well of $V_Q$ and satisfy \eqref{finiteenergybound} with $c<1$ replaced by a slightly larger $c'<1$.
In particular, since
\begin{align*}
E_Q(\omega(t),\lambda(t))\leq E_Q(\omega_0,\lambda_0)+C(Q-q(\omega_*))^2< c(1+C(Q-q(\omega_*))^{1/2})(Q-q(\omega_*))^{3/2},
\end{align*}
taking $0<Q-q(\omega_*) < \(\frac{1-c}{2cC}\)^2$, we have $E_Q(\omega(t),\lambda(t))<\frac{1+c}{2}(V_Q(\omega_-)-V_Q(\omega_+))$.
\end{remark}


As a direct Corollary of Theorem \ref{thm:dynamics}, we see  that $(\omega(t),\lambda(t))$ stays near the $1$ dimensional torus $\{(\omega,\lambda)\in \R^2\ |\  E_Q(\omega,\lambda)= E_Q(\omega_0,\lambda_0),\ |\lambda|+|\omega-\omega_*|\ll1, \omega>\omega_-\}$.
\begin{corollary}\label{corollary:longorb}
Let $$O_{\omega_0,\lambda_0}:=\{(\omega,\lambda)\in \R^2\ |\ E_Q(\omega,\lambda)=  E_Q(\omega_0,\lambda_0),\ |\lambda|+|\omega-\omega_*|\ll1\}.$$ Assume $E_Q(\omega_0,\lambda_0)\sim (Q-q(\omega_*))^{3/2}$.
Then under the assumption of Theorem \ref{thm:dynamics} for the   solution of  \eqref{normeq}
we have
\begin{align*}
\sup_{t\in[0,T]}\inf_{(\omega,\lambda)\in O_{\omega_0,\lambda_0}} \((Q-q(\omega_*))^{-3/2}|\omega(t)-\omega|+(Q-q(\omega_*))^{-5/4}|\lambda(t)-\lambda|\)\lesssim 1,
\end{align*}
where $T=(Q-q(\omega_*))^{-\frac32(N-1)+\frac12}$.
\end{corollary}

One may view Corollary \ref{corollary:longorb} as a long (but not infinite) time orbital stability result.
As with the usual orbital stability results on ground states, the control of the phase for $(\omega,\lambda)$ near $O_{\omega_0,\lambda_0}$ is more subtle.
However, we can show that for fixed $n$, the difference of $(\omega(t),\lambda(t))$ between $(\omega_f(t),\lambda_f(t))$ is small for $t\in [0,nT_f]$, where $T_f$ is given in Proposition \ref{finitedimosc}.
\begin{theorem}\label{thm:shadow}
Let $N\geq 2$, $n\in \N$, $c\in (0,1)$ and $C_0>0$.
Then, there exist $\epsilon_n>0$ s.t.\ for $Q\in (q(\omega_*),q(\omega_*)+\epsilon_n^2)$, if $(\omega_0,\lambda_0,\widetilde r_0)$ satisfies \eqref{finiteenergybound} and $\|\widetilde r_0\|_{H^1}\leq C_0 (Q-q(\omega_*))^{3/4}$, then
comparing the solutions of    \eqref{normeq} and of \eqref{finiteeq}
we have
\begin{align}\label{diff:est}
\sup_{t\in [0,nT_f]}\((Q-q(\omega_*))^{-3/2}|\omega(t)-\omega_f(t)|+(Q-q(\omega_*))^{-5/4}|\lambda(t)-\lambda_f(t)|\)\lesssim_n 1,
\end{align}
where $T_f$ is the period of $(\omega_f,\lambda_f)$ given in Proposition \ref{finitedimosc}.
\end{theorem}

Theorem \ref{thm:shadow} works for much shorter time $nT_f\sim (Q-q(\omega_*))^{-1/4}$ than the time given by Theorem \ref{thm:dynamics}, which is $T=(Q-q(\omega_*))^{-\frac32(N-1)+\frac12}$.
This is quite natural since the time evolution of the linear operator obtained linearizing
system \eqref{finiteeq} at the solution $(\omega_f,\lambda_f)$ appears to be growing linearly.
Hence, expressing $(\omega,\lambda)$ as a perturbation of $(\omega_f,\lambda_f)$,
their  difference tends to grow in a way similar to the "secular modes" discussed in p.473 \cite{Weinstein85}.
However, a combination of Theorems \ref{thm:dynamics} and  \ref{thm:shadow} shows that $(\omega,\lambda)$   oscillates   for $0<t< T=(Q-q(\omega_*))^{-\frac32(N-1)+\frac12} $ (and not just $0<t< nT_f$) with an oscillation  similar to that of  $(\omega_f,\lambda_f)$.
This is because, fixing $n=2$ (say), we can apply Theorem \ref{thm:shadow} for any time in $t\in [0,T]$.
So, taking $t_m$, $m=1,2,\cdots$ the time when $\omega(t_m)$ is (say) at a local maximum, then applying Theorem \ref{thm:shadow}  we see that $\|(t_{m+1}-t_m)-T_f^{(m)}\|\ll T_f^{(m)} $, with  $T_f^{(m)}$ the period of the solution $(\omega_{f,m}(t),\lambda_{f,m}(t))$ of \eqref{finiteeq} with $(\omega_{f,m}(t_m),\lambda_{f,m}(t_m))=(\omega(t_m),\lambda(t_m))$,
and in   each time interval $(t_{m}, t_{m}+2 T_f^{(m)} )$ the quantity $|\omega(t)-\omega_{f,m}(t)|+|\lambda(t)-\lambda_{f,m}(t)|$ satisfies \eqref{diff:est}.

\noindent Therefore, by the main Theorems \ref{thm:normal}, \ref{thm:dynamics} and \ref{thm:shadow}, we have proved the long time oscillation of $\omega(t)$ as informally stated in Theorem \ref{main:informal}.
\begin{remark}\label{rem:group}
Notice that at the point $( \omega _+, 0)$ the linearization of
\eqref{finiteeq}  is, for  $a(\omega)=-A(\omega)^{-1}q'(\omega)$,   \begin{equation*}
   B_0 = \begin{pmatrix}
0&    1 \\
 a (\omega _{+} ) &   0
\end{pmatrix}   \text{ with group }  e ^{tB_0}=  \begin{pmatrix}
\cos  \left ( t\sqrt{-a (\omega _{+} )} \right ) &
\frac{\sin  \left ( t\sqrt{-a (\omega _{+} )} \right )}{\sqrt{-a  (\omega _{+} )} }\\
 -\sqrt{-a (\omega _{+} )}    \sin  \left ( t\sqrt{-a (\omega _{+} )} \right )     &  \cos  \left ( t\sqrt{-a (\omega _{+} )} \right )
\end{pmatrix}    ,
 \end{equation*}
 with norm $ \| e ^{tB_0} \| _{L^\infty (\R )}\sim (-a  (\omega _{+} )) ^{-\frac{1}{2}}\sim (Q- q(\omega _*)) ^{-\frac{1}{2}}$.  In Lemma \ref{lem:finiteest} we consider periodic orbits $(\omega_f(t),\lambda_f(t))$ for  \eqref{finiteeq}  where $|\omega_f(t)-\omega_*|\lesssim (Q- q(\omega _*)) ^{\frac{1}{2}}$.
 It is clear now that it is not obvious how to bound the norm of the group of the  linearization
 of  \eqref{finiteeq} at $(\omega_f(t),\lambda_f(t))$ in terms of the norm   $ \| e ^{tB_0} \| _{L^\infty (\R )}$ and that  here it is not possible to use the continuity argument   in
  \cite[Proposition 3.3]{MW10DCDS}.

\end{remark}

\section{Normal form argument}\label{sec:normal}

In this section we prove Theorem \ref{thm:normal}.
We will start   summarizing the properties of the symbols $S^k_{i,j}$ and the dependent variable $\mu$.
Then we will construct a transform $\mathfrak F_0$ satisfying \eqref{error1} and \eqref{Darb}.
After that, we will successively construct canonical change of coordinates ($\phi_N^* \Omega_0=\Omega_0$) so that we can erase the interaction term to achieve the expansion \eqref{eq:mmain1}.

\subsection{Properties of symbols $\mathcal S^k_{i,j}$}\label{subsec:sandmu}

In this subsection we summarize the properties of $S^k_{i,j}$.

\begin{lemma}\label{lem:pdsymbol}
We have $P_d(\omega)P (\omega_*)= S^2_{1,0}$, $q'(\omega)=S^0_{1,0}$ and $q(\omega)-Q+Q(\widetilde r)=S^0_{2,0}$.

\end{lemma}

\begin{proof}
$P_d(\omega)P (\omega_*)=(P_d(\omega)-P_d(\omega_*))P (\omega_*)
=(\omega-\omega_*)\partial_{\omega}P_d(\omega_*+\theta(\omega-\omega_*))P (\omega_*)$.
The 2nd and 3rd claim follow  from $q'(\omega_*)=0$ and
$
q(\omega)-Q=q(\omega)-q(\omega_*)+q(\omega_*)-Q
$.
\end{proof}

We next consider the expansion of the dependent variable $\mu$.

\begin{lemma}\label{lem:muexpand}
We have
\begin{align}
\mu=-A(\omega)^{-1}\(q(\omega)-Q+Q(\widetilde r)\)+S^0_{0,2}.\label{muexpand}
\end{align}
and
\begin{align}\label{mudsymbol}
\partial_Q\mu=S^0_{0,0},\quad \partial_\omega \mu=-A^{-1}(\omega)q'(\omega)+S^0_{0,1},\quad \partial_\lambda \mu =S^0_{0,1},\quad \partial_\rho \mu=S^0_{0,0}\text{ and }\widetilde \nabla_{\widetilde r}\mu=S^1_{0,1}.
\end{align}
Recall that $\widetilde \nabla_{\widetilde r}$ is defined in \eqref{diftilder} and we have $\nabla_{\widetilde r} \mu=S^1_{0,1}+S^0_{0,0}\widetilde r$.
\end{lemma}

\begin{remark}
In formula \eqref{mudsymbol} the term $\nabla_{\widetilde r}\mu$ is defined by $\<\nabla_{\widetilde r}\mu, \sigma_1\widetilde v\>=d_{\widetilde r}F(\widetilde r)\widetilde v$ where $P(\omega)\widetilde v=\widetilde v$ and $F(\widetilde r)=\mu(Q,\omega,\lambda,Q(\widetilde r),\widetilde r)$ for fixed $Q,\omega$ and $\lambda$.
\end{remark}

\begin{proof}
By \eqref{def:mu}, we have
\begin{align*}
Q(\tilde \phi_\omega)+\<\widetilde \phi_\omega,\sigma_1(\lambda\Psi_3(\omega)+\mu\Psi_4(\omega)-P_d(\omega)\widetilde r\>+\<\widetilde \phi_\omega,\sigma_1 \widetilde r\>+S^0_{0,2}+Q(\widetilde r)-Q=0.
\end{align*}
Therefore, since
\begin{align*}
\<\widetilde \phi_\omega, \sigma_1(\lambda\Psi_3 +P(\omega) \widetilde r)\>=-\Omega(\Psi_1,\lambda \Psi_3+P \widetilde r)=0\text{ and }\<\tilde \phi_\omega,\Psi_4(\omega)\>=-\Omega(\Psi_1(\omega),\Psi_4(\omega))=A(\omega),
\end{align*}
we have \eqref{muexpand}.
Next, set $\underline u:=\widetilde{\phi}_\omega + \lambda  \Psi_3(\omega)+ \mu \Psi_4(\omega)+P (\omega)\widetilde{r}$.
Then differentiating \eqref{def:mu} (with $\rho=Q(\widetilde r)$) with respect to $Q$, $\omega$, $\lambda$ and $\widetilde r$, we have
\begin{align*}
\<\underline u, \sigma_1\Psi_4\>\partial_Q \mu &=1  \text{ where from \eqref{Psi12} and \eqref{50} we have } \<\underline u, \sigma_1\Psi_4\>=A+S^{0}_{0,1},\\
\<\underline u, \sigma_1\Psi_4\>\partial_\omega \mu &=-\<\underline u, \sigma_1(\Psi_2+\lambda\partial_\omega\Psi_3+\mu\partial_\omega \Psi_4 +\partial_\omega P \widetilde r)\>=-q'(\omega)+S^0_{0,1},\\
\<\underline u, \sigma_1\Psi_4\>\partial_\lambda \mu &=-\<\underline u, \sigma_1\Psi_3\>=S^0_{0,1},\\
\<\underline u, \sigma_1\Psi_4\>\<\nabla_{\widetilde r} \mu,\sigma_1 \cdot\> &=\<\underline u, \sigma_1 P \cdot\>=\<S^1_{0,1}+\widetilde r, \sigma_1 \cdot\>.
\end{align*}
Therefore, we have the conclusion.
\end{proof}

We next consider the relation of the indexes $i,j$ of $S^k_{i,j}$ and the differentiation with respect to $Q,\omega,\lambda$ and $\widetilde r$.
\begin{lemma}\label{difsymbol}
We set
\begin{equation}\nonumber
  I(i)= \left\{\begin{matrix}
  i  \text{  for $i=0,1$ }  ,   \\
i-2   \text{  for $i\geq 2$. }
\end{matrix}\right.
\end{equation}
Let $i,j,k\in \Z_{\geq 0}$.
We have
\begin{align*}
&\partial_Q S^k_{i,j}=S^k_{I(i),j}+S^k_{i,j-1},\quad \partial_\rho S^k_{i,j}=S^k_{I(i),j}+S^k_{i,j-1},\quad \partial_\omega S^k_{i,j}=S^k_{i-1,j}+S^k_{i+1,j-1},\quad \partial_\lambda S^k_{i,j}=S^k_{i,j-1},
\end{align*}
and
\begin{align*}
\widetilde\nabla_{\widetilde r} S^0_{i,j}&=S^1_{i,j-1},\\ d_{\widetilde r} S^1_{i,j}&=S^2_{i,j-1}+\<S^1_{0,1},\sigma_1\cdot\>S^1_{i,j-1}+\(S^1_{I(i),j}+ S^1_{i,j-1}\)\<\widetilde r,\sigma_1\cdot\>.
\end{align*}
Here, for $i,j=0$, the indexes $i-1,j-1$ are always understood as $0$.
\end{lemma}

\begin{proof}
Let $S^0_{i,j}=\hat f(Q,\omega,\lambda,\mu(Q,\omega,\lambda,Q(\widetilde r),\widetilde r),Q(\widetilde r),\widetilde r)$.
Then, by \eqref{mudsymbol}
\begin{align*}
\partial_Q S^0_{i,j}=\partial_Q \hat f + \partial_\mu \hat f \partial_Q \mu=S^0_{I(i),j}+S^0_{i,j-1}S^0_{0,0}.
\end{align*}
Here, notice that $I(i)$ appears because of the regularity of $\hat f$.
We obtain the others by similar manner.
\end{proof}

We next consider a $C^\infty$-diffeomorphism $\mathfrak F\in C^\infty(D_{\Hrad^1}(\mathcal T_{\omega_*},\delta),\Hrad^1(\R^3,\widetilde \C))$ s.t.
\begin{align}\label{Ferror}&
\mathfrak F^* Q=Q,\ \mathfrak F^*\omega=\omega+S^0_{0,l},\ \mathfrak F^* \lambda=\lambda+S^0_{1,l},\ \mathfrak F^* \widetilde r = e^{\im S^0_{0,l} \sigma_3}(\widetilde r + S^1_{\sigma(l),l})\\& \label{eq:sigma(i)}\text{for some $l\geq 1$ and  }  \sigma(l)= \left\{\begin{matrix}
  1  \text{  for $l= 1,2$ }  ,   \\
0   \text{  for $l\geq 3$. }
\end{matrix}\right.
\end{align}
We note that all the transforms we use in this paper satisfy \eqref{Ferror} for some $l\geq 1$.
In particular, $\mathfrak  F_0$ in Proposition \ref{prop:3}, $\phi_1$ in Lemma \ref{lem:nform1} and $\phi_2$ in Lemma \ref{lem:nform2} satisfy \eqref{Ferror} with $l=1$ and $\phi_{N+1}$ appearing in Lemma \ref{lem:ind} satisfies \eqref{Ferror} with $l=N$.
It is obvious that we have
\begin{align}\label{Sijtrans}
\mathfrak  F^* S^k_{i,j}=S^k_{i,j}.
\end{align}
We study the expansion of $\mathfrak  F^* Q(\widetilde r)$ and $\mathfrak  F^* \mu$.

\begin{lemma}\label{lem:muQtrans}
Let $l\geq 1$, $\delta>0$ and let $\mathfrak  F\in C^\infty(D_{\Hrad^1}(\mathcal T_{\omega_*},\delta),\Hrad^1(\R^3,\widetilde \C))$ be a $C^\infty$-diffeomorphism satisfying \eqref{Ferror}.
Then, for the $\sigma(l)$  in \eqref{eq:sigma(i)} we have
\begin{align*}
\mathfrak  F^* Q(\widetilde r)=Q(\widetilde r)+S^0_{\sigma(l),l+1},\quad \mathfrak  F^* \mu =\mu + S^0_{1,l}.
\end{align*}
\end{lemma}

\begin{proof}
First, since  $ e^{\im S^0_{0,l} \sigma_3}\sigma _1e^{\im S^0_{0,l} \sigma_3}=\sigma _1$ by $\sigma_3 \sigma _1=-\sigma_1 \sigma _3$, we have
\begin{align*}
\mathfrak  F^* Q(\widetilde r)=\<\widetilde r+ S^1_{\sigma(l),l},\sigma_1\(\widetilde r+ S^1_{\sigma(l),l}\)\>=Q(\widetilde r)+S^0_{\sigma(l),l+1}+S^0_{2\sigma(l),2l}=Q(\widetilde r)+S^0_{\sigma(l),l+1}.
\end{align*}
Next, from Lemma \ref{lem:muexpand}, we have
\begin{align*}
\mathfrak  F^*\mu&=\mu\(Q,\omega+S^0_{0,l},\lambda+S^0_{1,l},Q(\widetilde r)+S^0_{\sigma(l),l+1},\widetilde r+(S^0_{0,l}\sigma_3+S^0_{0,2l})\widetilde r+S^1_{\sigma(l),l}\)\\&
=\mu(Q,\omega,\lambda,Q(\widetilde r),\widetilde r)+S^0_{1,0}S^0_{0,l}+S^0_{0,1}S^0_{1,l}+S^0_{0,0}S^0_{\sigma(l),l+1}+S^0_{0,l+1}\\&=S^1_{1,l}.
\end{align*}
Therefore, we have the conclusion.
\end{proof}

\subsection{Darboux coordinates}
\label{sec:Darboux}

As usual in the proof of Darboux theorem, we start from studying the 1--form  $\Gamma(\widetilde u)=\frac12 \Omega(\widetilde u,\cdot)$ introduced in Sect.\  \ref{subsec:Hamstruc}.
\begin{lemma}\label{lem:gamma}
We have
\begin{equation}\label{eq:gamma2} \begin{aligned}  \Gamma (\widetilde{u})=\Gamma _0(\widetilde{u})
 +\Omega(S^0_{0,1}\widetilde{r}+S^1_{1,1}, d\widetilde{r}) +S^0_{1,1} d \omega +S^0_{0,2}d \lambda + (\text{exact})_1
  \end{aligned}
\end{equation}
where $(\text{exact})_1
:=d \left [ \Omega(\widetilde{\phi}_\omega,P  \widetilde{r})+  \Omega(\widetilde{\phi}_\omega,\Psi_3) \lambda + B \lambda \mu  \right ]$ and $\Gamma_0$ is defined in \eqref{eq:gamma1}.

\end{lemma}
\proof  It is a simple linear algebra argument. Recalling the definition of $ \underline{u}$ in \eqref{eq:vfieldcomp}, using the functions $(\vartheta,\omega,\lambda,\mu, {\widetilde{r}})$ we have
\begin{align*}
2\Gamma ( \widetilde{u} )=2Q d \vartheta + \Omega (\underline{u},\partial_\omega \underline{u}) d \omega + \Omega(\underline{u}, \Psi_3)d \lambda  + \Omega(\underline{u}, \Psi_4)d \mu + \Omega(\underline{u}, P  d\widetilde{r}).
\end{align*}
Since $\partial_\omega \underline{u} = \Psi_2+ \lambda  \partial_\omega \Psi_3+ \mu \partial_\omega \Psi_4 -\partial_\omega P_d \widetilde{r}$, we have
\begin{align*}
\Omega(\underline{u}, \partial_\omega\underline{u})=&\lambda \Omega(\widetilde{\phi}_\omega,\partial_\omega\Psi_3)- \Omega(\widetilde{\phi}_\omega,\partial_\omega P_d \widetilde{r})-A \lambda + S_{0,2}^{0}
\end{align*}
and
\begin{align*}
\Omega(\underline{u}, \Psi_3)&=\Omega(\widetilde{\phi}_\omega,\Psi_3)-\mu \Omega(\Psi_3,\Psi_4)=\Omega(\widetilde{\phi}_\omega,\Psi_3)-B \mu ,\\
\Omega(\underline{u}, \Psi_4)&=\lambda \Omega(\Psi_3,\Psi_4)=B \lambda ,\\
\Omega(\underline{u}, P  d\widetilde{r})&=\Omega(\widetilde{\phi}_\omega,P  d\widetilde{r})+ \Omega(P  \widetilde{r},d\widetilde{r})=
d[ \Omega(\widetilde{\phi}_\omega, P  \widetilde{r})]+\Omega(\widetilde{\phi}_\omega,\partial_\omega P_d \widetilde{r})d \omega + \Omega(P\widetilde{r},d\widetilde{r}).
\end{align*}
Here, recall $A=A(\omega)$ and $B=B(\omega)$ are given in Proposition \ref{prop:1} and \eqref{70} respectively.
Therefore, we have
\begin{align*}
2\Gamma (\underline{u})=&2Q d\vartheta + \lambda\Omega  (\widetilde{\phi}_\omega,   \partial_\omega \Psi_3 )d \omega-\lambda A(\omega) d \omega+ S^0_{0,2}d \omega\\&
+
\Omega(\widetilde{\phi}_\omega,\Psi_3) d \lambda-\mu B(\omega)d \lambda+ \lambda B(\omega)d \mu + d \left [ \Omega(\widetilde{\phi}_\omega, P \widetilde{r}) \right ]+ \Omega(P \widetilde{r},d\widetilde{r}) .
\end{align*}
We   use the following identities:
\begin{equation*}  \begin{aligned}   \lambda\Omega(\widetilde{\phi}_\omega,\partial_\omega \Psi_3)  d \omega+\Omega(\widetilde{\phi}_\omega,\Psi_3) d \lambda &=d \(\Omega(\widetilde{\phi}_\omega,\Psi_3)\lambda \) - A\lambda   d \omega ,
\\
 -B \mu  d \lambda  &=-
 d\(B \lambda \mu  \) +B\lambda  d \mu +S^0_{0,2} d \omega .
\end{aligned}
\end{equation*}
Substituting we obtain by Lemma \ref{lem:pdsymbol},
\begin{align*}
2\Gamma( \underline{u}  )=&2Q d\vartheta -2 A\lambda  d \omega+2 B\lambda  d\mu + \Omega(\widetilde{r},d\widetilde{r})
\\& + \Omega(S^1_{1,1}, d\widetilde{r}) +S^0_{0,2} d \omega
+d\left [ \Omega(\widetilde{\phi}_\omega,P  \widetilde{r})+ \Omega(\widetilde{\phi}_\omega,\Psi_3) \lambda - B\lambda \mu \right ] .\end{align*}
From \eqref{mudsymbol}, we have
\begin{align*}
B\lambda d\mu &=B\lambda \(\partial_Q \mu dQ+S^0_{1,0}d \omega + S^0_{0,1}d \lambda + S^0_{0,0}\< \widetilde{r},\sigma_1 d\widetilde{r} \>+\<S^1_{0,1}, \sigma_1 d\widetilde{r}\>\)\\&=
B\lambda \partial_Q \mu dQ+S^0_{1,1}d \omega + S^0_{0,2}d \lambda + S^0_{0,1}\<\widetilde{r},\sigma_1 d\widetilde{r}\>+\<S^1_{0,2},\sigma_1 d\widetilde{r}\>
\end{align*}
Therefore, we have  the following, which implies \eqref{eq:gamma2}:
\begin{align}
2\Gamma (\underline{u})=&2Q d\vartheta -2A \lambda  d \omega + \Omega(\widetilde{r},\sigma_1 d\widetilde{r}) +2B\lambda  \partial_Q \mu dQ \label{gamma3}\\&
+ \<S^0_{0,1}\widetilde{r}+S^1_{1,1}, \sigma_1 d\widetilde{r}\> +S^0_{1,1} d \omega +S^0_{0,2}d \lambda
+d( \Omega(\widetilde{\phi}_\omega,P  \widetilde{r})+ \Omega(\widetilde{\phi}_\omega,\Psi_3) \lambda - B \lambda \mu ).\nonumber
\end{align}
\qed

We now state the Darboux theorem, which corresponds to the first half of Theorem \ref{thm:normal}.
Recall $\mathcal T_{\omega_*}=\{e^{\im \theta \sigma_3}\widetilde \phi_{\omega_*}\ |\ \theta\in\R\}$.

\begin{proposition}[Darboux theorem]\label{prop:3}
There exists $\delta_0>0$ s.t.
there exists a $C^\infty$--diffeomorphism $\mathfrak F_0 \in C^\infty(D_{\Hrad^1}(\mathcal T_{\omega_*},\delta_0),\Hrad^1(\R^3,\widetilde \C))$ onto a neighborhood of $\mathcal T_{\omega_*}$ in $\Hrad^1$ s.t.\
\begin{align}\label{darerror}
\mathfrak F_0^*\omega=\omega+S^0_{0,2},\ \mathfrak F_0^* \lambda=\lambda+S^0_{1,1},\ \mathfrak F_0^*\widetilde r =e^{\im S^{0}_{0,1}\sigma_3}\(\widetilde r + S^1_{1,1}\)\text{ and }\mathfrak F_0^* Q=Q
\end{align}
 and
\begin{align*}
\mathfrak F_0^* \Omega = \Omega_0,
\end{align*}
Furthermore, we have $\mathfrak F_0^* \mu =\mu +S^0_{1,1}$  and $\mathfrak F_0^* Q(\widetilde{r}) =Q(\widetilde{r}) +S^0_{1,2}$.
\end{proposition}

\begin{remark}
The coordinates $(\vartheta, Q,\omega,\lambda,\widetilde r)=(\vartheta(u), Q(u),\omega(u),\lambda(u),\widetilde r(u))$ in \eqref{darerror} are given by $(\vartheta(u), Q(u),\omega(u),\lambda(u),\widetilde r(u))=\mathcal  F_2^{-1}(u)$, where $\mathcal  F_2$ is the modulation coordinate given in Proposition \ref{prop:32}.
We make no care about $\mathfrak F_0^* \vartheta-\vartheta$, even though we can compute  it by looking in to the proof.
\end{remark}

For the proof of Proposition \ref{prop:3}, we consider the following lemma.

\begin{lemma}\label{lem:dar1}
There exists $\delta>0$ s.t.\ there exists $\mathcal X\in C^\infty(D_{\R}(0,2)\times D_{\Hrad^1}(0,\delta),\Hrad^1)$ s.t.
\begin{align}\label{dar1est}
\mathcal X(s)_Q=0, \quad \mathcal X(s)_\omega=S^0_{0,2}, \quad \mathcal X(s)_\lambda=S^0_{1,1},\quad \mathcal X(s)_{\widetilde r}=S^0_{0,1}\im \sigma_3 \widetilde r + S^1_{1,1}
\end{align}
and if  $\Omega_s:=\Omega_0+s(\Omega-\Omega_0)$, we have $i_{\mathcal X(s)}\Omega_s= -\widetilde \Gamma$, where
$\widetilde \Gamma = \Gamma-\Gamma_0 -(exact)_1$ and $i_X \Omega _s=\Omega _s(X,\cdot)$.
\end{lemma}

\begin{proof}
We will directly solve the equation
\begin{align}\label{Xs1}
\Omega_0(\mathcal X(s),\cdot)+s\(\Omega(\mathcal X(s),\cdot)- \Omega_0(\mathcal X(s),\cdot)\)=-\widetilde \Gamma.
\end{align}
By \eqref{eq:vfieldcomp}, it suffices to determine $\mathcal X(s)_\tau$ for $\tau=\vartheta,Q,\omega,\lambda,\widetilde r$.
By \eqref{gamma3} in the proof of Lemma \ref{lem:gamma}, we have
\begin{align}\label{Xs2}
\widetilde \Gamma=S^0_{1,1}d\omega +S^0_{0,2}d \lambda +\<S^0_{0,1}\widetilde r +S^1_{1,1},\sigma_1 d\widetilde r\>.
\end{align}
For the first term of the r.h.s.\ of \eqref{Xs1}, we have
\begin{align}
\Omega_0(\mathcal X(s),\cdot)=&\mathcal X(s)_Qd\vartheta-\mathcal X(s)_\vartheta dQ+A(\omega)\mathcal X(s)_\omega d \lambda-A(\omega)\mathcal X(s)_\lambda d\omega+\<\im \sigma_3 \mathcal X(s)_{\widetilde r},\sigma_1 d\widetilde r\>\label{Xs3}\\&+ \(S^0_{0,0}\mathcal X(s)_\lambda + S^0_{0,1}\mathcal X(s)_\omega + \<S^0_{0,1}\widetilde{r}+S^1_{0,1},\sigma_1\mathcal X(s)_{\widetilde r} \>\)  dQ\nonumber\\&+\mathcal X(s)_Q\(S^0_{0,0}d \lambda + S^0_{0,1}d \omega + \<S^0_{0,1}\widetilde{r}+S^1_{0,1},\sigma_1 d\widetilde{r} \>\).\nonumber
\end{align}
By \eqref{Xs2} and Lemma \ref{difsymbol}, we have
\begin{align}
d \left[S^0_{1,1} d \omega\right]&= S^0_{1,0}dQ\wedge d\omega+S^0_{1,0} d\omega \wedge d \lambda +d\omega\wedge \<S^1_{1,0}+S^0_{1,0}\widetilde r,\sigma_1 d\widetilde r\>\label{Xs4}\\
d\left[S^0_{0,2}d \lambda\right]& =S^0_{0,1}dQ\wedge d \lambda+S^0_{1,1}d\omega \wedge d \lambda+d \lambda\wedge \<S^1_{1,1}+S^0_{0,1}\widetilde r,\sigma_1 d\widetilde r\>,\label{Xs5}
\end{align}
and
\begin{align}
&d\<S^0_{0,1}\widetilde{r}+S^1_{1,1}, \sigma_1 d\widetilde{r}\>=dQ\wedge \<S^0_{0,0}\widetilde r+S^1_{1,0},\sigma_1 d\widetilde r\>+d\omega\wedge \<S^0_{1,0}\widetilde r+S^1_{1,0},\sigma_1 d\widetilde r\>+\label{Xs6}\\&
d \lambda\wedge \<S^0_{0,0}\widetilde r+S^1_{1,0},\sigma_1 d\widetilde r\>+\<\widetilde r, \sigma_1 d\widetilde r\>\wedge \<S^1_{0,0},\sigma_1 d\widetilde r\>+\<S^1_{0,1}, \sigma_1 d\widetilde r\>\wedge \<S^1_{1,0},\sigma_1 d\widetilde r\>+\<S^2_{1,0}d \widetilde r, \sigma_1 d\widetilde r\>.\nonumber
\end{align}
Here, notice that terms such as $\<S^0_{0,1}d\widetilde r,\sigma_1 d\widetilde r\>$ and $S^0_{0,0}\<\widetilde r, \sigma_1 d\widetilde r\>\wedge \<\widetilde r, \sigma_1 d\widetilde r\>$ are symmetric and do not appear when we take the exterior derivative.
By \eqref{Xs2}, \eqref{Xs4}--\eqref{Xs6} and $\Omega-\Omega_0=d \widetilde \Gamma$, we have
\begin{align}
&\Omega(\mathcal X(s),\cdot)-\Omega_0(\mathcal X(s),\cdot)=\(S^0_{1,0}\mathcal X(s)_\omega + S^0_{1,0}\mathcal X(s)_\lambda +\<S^1_{1,0}+S^0_{0,0}\widetilde r, \sigma_1 \mathcal X(s)_{\widetilde r}\>\)dQ\label{Xs7}\\&+\(S^0_{1,0}\mathcal X(s)_Q + S^0_{1,0}\mathcal X(s)_\lambda +\<S^1_{1,0}+S^0_{1,0}\widetilde r, \sigma_1 \mathcal X(s)_{\widetilde r}\>\)d\omega
\nonumber\\
&+\(S^0_{1,0}\mathcal X(s)_Q + S^0_{1,0}\mathcal X(s)_\omega +\<S^1_{1,0}+S^0_{0,0}\widetilde r, \sigma_1 \mathcal X(s)_{\widetilde r}\>\)d\lambda
\nonumber\\
&+\mathcal X(s)_Q\<S^1_{0,0}+S^0_{1,0}\widetilde r,\sigma_1 d\widetilde r\> + \mathcal X(s)_\omega\<S^1_{1,0}+S^0_{1,0}\widetilde r,\sigma_1 d\widetilde r\>+\mathcal X(s)_\lambda \<S^1_{1,0}+S^0_{0,0}\widetilde r,\sigma_1 d\widetilde r\>
\nonumber\\
&
+\<S^2_{1,0}\mathcal X(s)_{\widetilde r},\sigma_1 d\widetilde r\>+\<S^1_{0,0},\sigma_1\mathcal X(s)_{\widetilde r}\>\<\widetilde r, \sigma_1 d\widetilde r\>+\<\widetilde r,\sigma_1 \mathcal X(s)_{\widetilde r}\>\<S^1_{0,0},\sigma_1 d\widetilde r\>,\nonumber
\end{align}
where we have used $\<S^1_{i,j} \sigma_1 \mathcal X(s)_{\widetilde r}\>S^1_{j,i}=S^2_{1,0} \mathcal X(s)_{\widetilde r}$ for $(i,j)=(1,0)$ or $(0,1)$.

We substitute \eqref{Xs2}, \eqref{Xs3} and \eqref{Xs7} into \eqref{Xs1} and compare the coefficients of $d\vartheta$, $dQ$, $d \omega$, $d \lambda$ and $d\widetilde r$.
First, looking at the coefficients of $d\vartheta$, we have
\begin{align*}
\mathcal X(s)_Q=0.
\end{align*}
Next, the coefficients of  $d\omega$, $d \lambda$, $d\widetilde r$ have to satisfy
\begin{align}
&-A(\omega)\mathcal X(s)_\lambda +s\( S^0_{1,0}\mathcal X(s)_\lambda +\<S^1_{1,0}+S^0_{1,0}\widetilde r, \sigma_1 \mathcal X(s)_{\widetilde r}\>\)=S^0_{1,1},\label{Xs10}\\
&A(\omega)\mathcal X(s)_\omega+s\(S^0_{1,0}\mathcal X(s)_Q + S^0_{1,0}\mathcal X(s)_\omega +\<S^1_{1,0}+S^0_{0,0}\widetilde r, \sigma_1 \mathcal X(s)_{\widetilde r}\>\)=S^0_{0,2},\label{Xs11}\\
&\im \sigma_3 \mathcal X(s)_{\widetilde r}+s\(
\mathcal X(s)_\omega \(S^1_{1,0}+S^0_{1,0}\widetilde r\)
+\mathcal X(s)_\lambda \(S^1_{1,0}+S^0_{0,0}\widetilde r\)
+S^2_{1,0}\mathcal X(s)_{\widetilde r}\)\label{Xs12}\\&
+s\(
\<S^1_{0,0},\sigma_1\mathcal X(s)_{\widetilde r}\>\widetilde r+\<\widetilde r,\sigma_1 \mathcal X(s)_{\widetilde r}\>S^1_{0,0}
\)=S^0_{0,1} \widetilde r +S^1_{1,1}.\nonumber
\end{align}
We further decompose the third equation by setting $\mathcal X(s)_{\widetilde r}=x(s)\im \sigma_3 \widetilde r + \widetilde {\mathcal X}(s)_{\widetilde r}$.
Then, we have
\begin{align}
&-x(s)+s\(
\mathcal X(s)_\omega S^0_{1,0}
+\mathcal X(s)_\lambda S^0_{0,0}
+
\<S^1_{0,0},\sigma_1\mathcal X(s)_{\widetilde r}\>
\)=S^0_{0,1} ,\label{Xs15}\\
&\im \sigma_3 \widetilde{\mathcal X}(s)_{\widetilde r}+s\(
\mathcal X(s)_\omega S^1_{1,0}
+\mathcal X(s)_\lambda S^1_{1,0}
+S^2_{1,0}\mathcal X(s)_{\widetilde r}+
\<\widetilde r,\sigma_1 \mathcal X(s)_{\widetilde r}\>S^1_{0,0}
\)=S^1_{1,1}.\label{Xs16}
\end{align}
We can solve the system \eqref{Xs10}, \eqref{Xs11}, \eqref{Xs15} and \eqref{Xs16} by Neumann series and obtain
\begin{align*}
\mathcal X(s)_\omega=S^0_{0,2}, \quad \mathcal X(s)_\lambda=S^0_{1,1},\quad \mathcal X(s)_{\widetilde r}=S^0_{0,1}\im \sigma_3 \widetilde r + S^1_{1,1}.
\end{align*}

Finally, we define $\mathcal X(s)_\omega$ by looking at the coefficients of $dQ$.
\begin{align*}
&-\mathcal X(s)_\vartheta+\(S^0_{0,0}\mathcal X(s)_\lambda + S^0_{0,1}\mathcal X(s)_\omega + \<S^0_{0,1}\widetilde{r}+S^1_{0,1},\sigma_1\mathcal X(s)_{\widetilde r} \>\) \\&+s \(S^0_{1,0}\mathcal X(s)_\omega + S^0_{1,0}\mathcal X(s)_\lambda +\<S^1_{1,0}+S^0_{0,1}\widetilde r, \sigma_1 \mathcal X(s)_{\widetilde r}\>\)=0.
\end{align*}
Therefore, we have constructed the desired vector field $\mathcal X(s)$
\end{proof}

We now prove Proposition \ref{prop:3}.
\begin{proof}[Proof of Proposition \ref{prop:3}]
The desired change of coordinate $\mathfrak F_0$ will given by the transform $\left.\mathcal Y(s)\right|_{s=1}$, where $\mathcal Y(s)$ is the solution of
\begin{align}\label{prdar1}
\frac{d}{ds}\mathcal Y(s)= \mathcal X(s)(\mathcal Y(s)),\quad \mathcal Y(0)=\mathrm{Id},
\end{align}
see \cite{Cuccagna12Rend}.
The existence of such $\mathcal Y$ follows from standard argument.
If we set $$\mathcal Y^*(s)\omega=\omega+y_\omega(s),\quad \mathcal Y^*(s)\lambda=\lambda+y_\lambda(s)\text{ and }\mathcal Y^*(s)\widetilde r=e^{\widetilde y_{\widetilde r}(s)\im \sigma_3}(\widetilde r+ y_{\widetilde r}(s)),$$
\eqref{prdar1} can be written as
\begin{align*}
\frac{d}{ds}y_\omega(s)&=S^0_{0,2}\\
\frac{d}{ds}y_\lambda(s)&=S^0_{1,1}\\
\frac{d}{ds}\(e^{\widetilde y_{\widetilde r}(s)\im \sigma_3}(\widetilde r+ y_{\widetilde r}(s))\)&=S^0_{0,1}\im \sigma_3 e^{\widetilde y_{\widetilde r}(s)\im \sigma_3}(\widetilde r+ y_{\widetilde r}(s))+S^1_{1,1}.
\end{align*}
Thus, we have
\begin{align*}
y_\omega=y_\omega(1)=S^0_{0,2},\quad y_\lambda=y_\lambda(1)=S^0_{1,1},\quad y_{\widetilde r}=y_{\widetilde r}(1)=S^1_{1,1}\text{ and }\widetilde y_{\widetilde r}=\widetilde y_{\widetilde r}(1)=S^0_{0,1},
\end{align*}
 by elementary argument.
Further, since $\mathcal X(s)_Q=0$, we have $\mathcal Y^*Q=Q$.

Finally, the last statement of Proposition \ref{prop:3} follows from Lemma \ref{lem:muQtrans}.
\end{proof}

%
%

%
%
%
%
%
%
%
%
%

\subsection{Energy expansion}
\label{subsec:en exp1}

%
%

We have the following expansion.

\begin{lemma}
  \label{lem:expansion}  In the  coordinates  $( \vartheta , Q ,\omega ,\lambda ,\widetilde{r}) $ given in Proposition \ref{prop:32}  we have the expansion
\begin{align}  \label{eq:expan1}
E =& d (\omega ) -\omega Q + \frac{\lambda ^2}{2}  A (\omega )   +  \frac{\mu ^2}{2} B (\omega )  +(\omega-\omega_*)Q(\widetilde r)+\frac 1 2\<H_{\omega_*}  \widetilde{r}, \sigma _1 \widetilde{r}\>+E_P( \widetilde{r})\\& +S_{1,2}^{0} +\mathbf{R}_2+\mathbf{R}_7,\nonumber \end{align}
   where $\mathbf R_k$ are terms given in Definition \ref{def:Rk}

  \end{lemma}

\begin{proof}
Set $F(s,t)=S_\omega(s\Psi+t\widetilde r)$, where $S_\omega=E+\omega Q$ and $\Psi=\widetilde\phi_\omega+\lambda\Psi_3(\omega)+\mu \Psi_4(\omega)-P_d\widetilde r=S^1_{0,0}$.
Then, by Taylor expansion, we have
\begin{align}\label{Taylor}
F(1,1)=F(0,1)+F(1,0)+\sum_{j=1}^2\int_0^1 \frac{1}{j!} \partial_s\partial_t^j F(s,0)\,ds+\int_0^1\int_0^1\frac{(1-t)^2}{2}\partial_s\partial_t^3 F(s,t)\,dtds.
\end{align}
Thus, we have
\begin{align}\label{eq:expand1}
S_\omega(\Psi+ \widetilde r) =& S_\omega(\Psi)+S_\omega(\widetilde r)+\sum_{j=1}^2\int_0^1 \frac{1}{j!}\partial_s\left.\partial_t^j\right|_{t=0}S_\omega(s\Psi+t \widetilde r)\,ds\\&+ \int_0^1\int_0^1\frac{(1-t)^2}{2}\<\nabla^4E_P(s\Psi+t\widetilde r)(\Psi,\widetilde r,\widetilde r),\sigma_1\widetilde r\>\,dtds.\nonumber
\end{align}
Notice that we can replace $S_\omega$ by $E_P$ in the second line of \eqref{eq:expand1} because the $S_\omega-E_P$ consist  of quadratic terms.

We expand all terms  of \eqref{eq:expand1}.
For the first, we have
\begin{align*}
S_\omega(\Psi)&=S_\omega(\widetilde \phi_\omega)+\frac12\<H_{\omega}\(\lambda\Psi_3+\mu \Psi_4-P_d\widetilde r\),\sigma_1\(\lambda\Psi_3+\mu \Psi_4-P_d\widetilde r\)\>+S^0_{0,3}\\&=
d(\omega)+\frac12 \Omega(-\im \mathcal H_\omega(\lambda\Psi_3+\mu\Psi_4),\lambda\Psi_3+\mu\Psi_4)+S^0_{1,2}\\&=
d(\omega)+\frac12 A(\omega)\lambda^2+\frac{1}{2}B(\omega)\mu^2+S^0_{1,2},
\end{align*}
where we have used  $P_d(\omega)\widetilde r=S^1_{1,1}$ in the second line and  \eqref{46} and $a(\omega)=S^0_{1,0}$ in the third line.
For the third term with $j=1$ in the r.h.s.\ of \eqref{eq:expand1}, we have
\begin{align*}
\int_0^1 \partial_s\left.\partial_t\right|_{t=0}S_\omega(s\Psi+t \widetilde r)\,ds&=\int_0^1 \partial_s\<\nabla S_\omega(s\Psi),\sigma_1 P\widetilde r\>\,ds
=\<\nabla S_\omega(\Psi),\sigma_1 P\widetilde r\>\\&=\<H_\omega (\lambda\Psi_3+\mu \Psi_4-P_d\widetilde r),\sigma_1 P \tilde r\>+S^0_{0,3}=S^0_{0,3},
\end{align*}
where we have used $\nabla S_\omega(0)=0$ in the second equality and  orthogonality in the fourth equality.
For the remaining terms of the 1st line of \eqref{eq:expand1}, we have
\begin{align*}
&S_\omega(\widetilde r)+\int_0^1\frac{1}{2} \partial_s\left.\partial_t^2\right|_{t=0}S_\omega(s\Psi+t \widetilde r)\,ds=S_\omega(\widetilde r)+\frac{1}{2}\int_0^1 \partial_s\<\nabla^2 S_\omega(s\Psi) \widetilde r,\sigma_1  \widetilde r\>\,ds\\&
=S_\omega(\widetilde r)-\frac12\<\nabla^2 S_\omega (0)\widetilde r,\sigma_1 \widetilde r\>+\frac{1}{2}\<H_\omega \widetilde r,\sigma_1 \widetilde r\>\\&\quad+\frac12\int_0^1\<\nabla^3E_P(\tilde \phi_\omega+t(\lambda\Psi_3+\mu\Psi_4-P_d\widetilde r))(\lambda\Psi_3+\mu\Psi_4-P_d\widetilde r,\widetilde r),\sigma_1 \widetilde r \>\,dt\\&
=E_P(\widetilde r)+\frac{1}{2}\<H_{\omega_*} \widetilde r,\sigma_1 \widetilde r\>+(\omega-\omega_*)Q(\widetilde r)+\mathbf{R}_2.
\end{align*}
Here, since
\begin{align}
\nabla^3E_P(u)(v_1,v_2)=&2g'(|u|^2)\(\<u,\sigma_1 v_1\>_{\C^2}v_2+\<u,\sigma_1 v_2\>_{\C^2}v_1+\<v_1,\sigma_1 v_2\>_{\C^2}u\)\label{E3}\\&+4g''(|u|^2)\<u,\sigma_1 v_1\>_{\C^2}\<u,\sigma_1 v_2\>_{\C^2}u,\nonumber
\end{align}
with  $\tilde \phi_\omega+t(\lambda\Psi_3+\mu\Psi_4-P_d\widetilde r)=S^1_{0,0}$ and $\lambda\Psi_3+\mu\Psi_4-P_d\widetilde r=S^1_{0,1}$, we have
\begin{align*}
&\frac12\int_0^1\<\nabla^3E_P(\tilde \phi_\omega+t(\lambda\Psi_3+\mu\Psi_4-P_d\widetilde r))(\lambda\Psi_3+\mu\Psi_4-P_d\widetilde r,\widetilde r),\sigma_1 \widetilde r \>\,dt\\&= \<s^1_{0,1}\widetilde r,\sigma_1\widetilde r\>+\int_{\R^3} \<S^1_{0,0},\sigma_1 \widetilde r\>_{\C^2}\<S^1_{0,1},\sigma_1 \widetilde r\>_{\C^2}\,dx=\mathbf{R}_2.
\end{align*}
Finally, the 2nd line of \eqref{eq:expand1} is $\mathbf R_7$ by   definition.

\end{proof}

We next study the expansion of $\mathfrak F_0^* E$.
However, before that we study what happens for $\mathfrak F^* \mathbf R_k$ as a preliminary.

\begin{lemma}\label{lem:Rexpand}
Let $l\geq 1$, $\delta>0$ and let $\mathfrak F\in C^\infty(D_{\Hrad^1}(\mathcal T_{\omega_*},\delta),\Hrad^1(\R^3,\widetilde \C))$ satisfy \eqref{Ferror}--\eqref{eq:sigma(i)}.
Then,            we have
\begin{align}
\mathfrak F^*\(\sum_{k=2}^7\mathbf R_k\)&=\sum_{k=2}^7\mathbf R_k +S^0_{\sigma(l)+1,l+1}+S^0_{0,l+2}.\label{Rexpand}
\end{align}

\end{lemma}

For the   elementary but long proof, see  Sect.\ \ref{app:Rexpand} in the Appendix.

We next study the expansion of $\mathfrak F^* E_P$.

\begin{lemma}\label{lem:transEp}
Let $l\geq 1$, $\delta>0$ and let $\mathfrak F\in C^\infty(D_{\Hrad^1}(\mathcal T_{\omega_*},\delta),\Hrad^1(\R^3,\widetilde \C))$ satisfy \eqref{Ferror}--\eqref{eq:sigma(i)}.
Then, we have
\begin{align}
\mathfrak F^*E_P(\widetilde r)&=E_P(\widetilde r)+\mathbf R_2 +\mathbf R_7+S^0_{4\sigma(l),4l}.\label{Epexpand}
\end{align}

\end{lemma}

\begin{proof}
By  Taylor expansion as in \eqref{Taylor} we have
\begin{align*}
\mathfrak F_0^* E_P(\widetilde r)&=E_P(\widetilde r+S^1_{\sigma(l),l})=E_P(\widetilde r)+E_P(S^1_{\sigma(l),l})+\sum_{j=1,2}\int_0^1\partial_s\left.\partial_t^j\right|_{t=0} E_P(t\widetilde r+s S^1_{\sigma(l),l})\,ds + \mathbf R_7.
\end{align*}
Further, we have $E_P(S^1_{\sigma(l),l})=\int_{\R^3} G(s^1_{2\sigma(l),2l})\,dx=S^0_{4\sigma(l),4l}$, and by \eqref{36} we have
\begin{align*}
\int_0^1\partial_s\left.\partial_t\right|_{t=0} E_P(t\widetilde r+s S^1_{1,1})\,ds=\<\nabla^2 E_P(sS^1_{\sigma(l),l})S^1_{\sigma(l),l},\sigma_1 \widetilde r\>=S^0_{4\sigma(l),4l+1}
\end{align*}
and by \eqref{E3} we have
\begin{align*}
\int_0^1\partial_s\left.\partial_t^2\right|_{t=0} E_P(t\widetilde r+s S^1_{1,1})\,ds=\int_0^1 \<\nabla^3 E_P(s S^1_{\sigma(l),l})(S^1_{\sigma(l),l},\widetilde r),\sigma_1\widetilde r\>\,ds=\mathbf R_2.
\end{align*}
Therefore, we obtain \eqref{Epexpand}.
\end{proof}

We next consider the expansion in the Darboux coordinates.
\begin{lemma}
  \label{lem:back}
Let  $\mathfrak F_0$ be the transformation  in Proposition \ref{prop:3}.
Then $ \mathfrak F_0^* E $ has the expansion
\begin{align}  \label{eq:darexpand1}
\mathfrak F_0^* E =& d (\omega ) -\omega Q + \frac{\lambda ^2}{2}  A (\omega )   +  \frac{\mu ^2}{2} B (\omega )  +(\omega-\omega_*)Q(\widetilde r)+\frac 1 2\<H_{\omega_*}  \widetilde{r}, \sigma _1 \widetilde{r}\>+E_P( \widetilde{r})\nonumber\\& +S_{1,2}^{0} +\sum_{j=2}^7\mathbf{R}_j
\end{align}
where $\mathbf R_j$are terms in the form given in Definition \ref{def:Rk}.
\end{lemma}

%
%
%
%
%
%
%
%

\begin{proof} We consider the expansion of $E$ in \eqref{eq:expan1} and we apply
the pull back $\mathfrak F_0^*$  to each of the terms in the r.h.s.
First, if we set $ \omega '=\mathfrak F_0^* \omega = \omega + y _{\omega}$ where $y _{\omega}=S^{0} _{0,2}$, we have  by $d '(\omega  )=q(\omega  )$
\begin{align}
  d (\omega ') -\omega 'Q &=   d (\omega  ) -\omega  Q  +  (q(\omega  )-Q ) y _{\omega} +O(y _{\omega} ^2)= d (\omega  ) -\omega  Q +S^{0} _{1,0}S^{0} _{0,2}+O\left ( (S^{0} _{0,2})^2\right )\nonumber  \\&= d (\omega  ) -\omega  Q +S^{0} _{1,2}.\label{darexp1}
\end{align}
Next,
\begin{align}
&\mathfrak F_0^*\(\frac{\lambda ^2}{2}   A (\omega )  +  \frac{\mu ^2}{2} B (\omega )  +(\omega-\omega_*)Q(\widetilde r)\)\nonumber\\&=\frac{(\lambda+S^0_{1,1})^2}{2}  A (\omega+S^0_{0,2} )   +  \frac{\(\mu+S^0_{1,2}\) ^2}{2} B (\omega+S^0_{0,2} ) +(\omega-\omega_*+S^0_{0,2})\(Q(\widetilde r)+S^0_{1,2}\)\nonumber\\&
=\frac{\lambda ^2}{2}  A (\omega )   +  \frac{\mu ^2}{2} B (\omega )  +(\omega-\omega_*)Q(\widetilde r)+S^0_{1,2}.\label{darexp2}
\end{align}
For $\mathfrak F_0^*\(\frac12\<H_\omega \widetilde r, \sigma_1 \widetilde r\>\)$, recall by \eqref{36} that  $H_{\omega_*}=-\Delta + \omega_* + g(\phi_{\omega_*}^2)+g'(\phi_{\omega_*}^2)\phi_{\omega_*}^2+g'(\phi_{\omega_*}^2)\phi_{\omega_*}^2\sigma_1$.
Thus,
\begin{align}
&\mathfrak F_0^*\(\frac 1 2\<H_{\omega_*}  \widetilde{r}, \sigma _1 \widetilde{r}\>\)=\frac12 \< H_{\omega_*} e^{\widetilde y_{\widetilde r}\im \sigma_3}\(\widetilde r + S^1_{1,1}\),\sigma_1e^{\widetilde y_{\widetilde r}\im \sigma_3}\(\widetilde r+ S^1_{1,1}\)\>\nonumber\\&\quad+\frac12\<\(e^{-\widetilde y_{\widetilde r}\im \sigma_3}H_{\omega_*}e^{\widetilde y_{\widetilde r}\im \sigma_3}-H_{\omega_*}\)\(\widetilde r+S^1_{1,1}\), \sigma_1 \(\widetilde r+S^1_{1,1}\)\>\nonumber\\&=
\frac12 \< H_{\omega_*} \widetilde r ,\sigma_1 \widetilde r \>+S^0_{1,2}
+\<g'(\phi_{\omega_*}^2)\phi_{\omega_*}^2\sigma_1
\(e^{2\widetilde y_{\widetilde r}\im \sigma_3}-1\)\widetilde r,\sigma_1 \widetilde r\>.\label{darexp3}
\end{align}
Now, by $y_{\widetilde r}=S^0_{0,1} $ and $e^{2\widetilde y_{\widetilde r}\im \sigma_3}-1=S^0_{0,1}\im \sigma_3+S^0_{0,2} $, we have $\sigma_1\(e^{2\widetilde y_{\widetilde r}\im \sigma_3}-1\)=S^0_{0,1}\sigma_2+S^0_{0,2}\sigma_1$.
Thus, by
\begin{align*}
\sigma_1=\frac{1}{2}\( e \cdot {^te} + f\cdot {^tf}\),\quad \sigma_2=\frac12 \( e\cdot {^tf} - f\cdot {^te}\),\quad e=\begin{pmatrix}1 \\ 1\end{pmatrix} ,\ f=\begin{pmatrix}\im \\ -\im\end{pmatrix},
\end{align*}
where ${^te}$, ${^tf}$ are transpose of $e$, $f$, we have
\begin{align*}
g'(\phi_{\omega_*}^2)\phi_{\omega_*}^2\sigma_1
\(e^{2\widetilde y_{\widetilde r}\im \sigma_3}-1\)=g'(\phi_{\omega_*}^2)\phi_{\omega_*}^2\sum_{x,y=e,f} \(S^0_{0,1}x\)\cdot {^t \(S^0_{0,0}y\) } =S^1_{0,1}\cdot {^t\(S^1_{0,0}\)},
\end{align*}
where we have omitted the sum in the r.h.s.
Now, since $\<u,v\>_{\C^2}={^tu}\cdot v$, we have
\begin{align*}
&\<g'(\phi_{\omega_*}^2)\phi_{\omega_*}^2\sigma_1
\(e^{2\widetilde y_{\widetilde r}\im \sigma_3}-1\)\widetilde r,\sigma_1 \widetilde r\>_{\C^2}={^t\( S^1_{0,1}\cdot {^t\(S^1_{0,0}\)}  \widetilde r\) } \cdot \(\sigma_1 \widetilde r\)=\({^t\(\widetilde r\) } \cdot S^1_{0,0}\)\({^t\(S^1_{0,1}\) } \cdot\(\sigma_1 \widetilde r\)\)\\&=
\<\sigma_1 S^1_{0,0},\sigma_1 \widetilde r\>_{\C^2}\<S^1_{0,1},\sigma_1 \widetilde r\>_{\C^2}.
\end{align*}
Since $\sigma_1 S^1_{0,0}=S^1_{0,0}$, we have
\begin{align}\label{darexp4}
\<g'(\phi_{\omega_*}^2)\phi_{\omega_*}^2\sigma_1
\(e^{2\widetilde y_{\widetilde r}\im \sigma_3}-1\)\widetilde r,\sigma_1 \widetilde r\>=\mathbf R_2.
\end{align}

Finally, from Lemmas \ref{Rexpand} and \ref{lem:transEp}, we have
\begin{align*}
\mathfrak F_0^*\(S^0_{1,2}+\mathbf R_2+\mathbf R_7 + E_P(\widetilde r)\)=S^0_{1,2}+\sum _{k=2}^7\mathbf R_k + E_P(\widetilde r),
\end{align*}
where, we note that $S^0_{1,2}$ and $\mathbf R_k$ in the l.h.s.\ and r.h.s.\ are different.

Therefore, we have the conclusion.
\end{proof}

\subsection{Canonical transformations}
\label{sec:normalf}

We consider the coordinates of Proposition \ref{prop:3}. In these coordinates
the energy $E$ has the expansion   of  formula \eqref{eq:darexpand1}.
In the sequel we will consider real valued  functions, for $l \in \N$,  of the form
 \begin{equation}
\label{eq:chi01}\chi   =\sum _{m+n=l  +1} c_{m,n}(Q,\omega ,Q(\widetilde{r}))     \lambda^{m} \mu^{n} + \sum _{m+n=l } \lambda^{m} \mu^{n}
 \Omega(C_{m,n
}(Q,\omega ,Q(\widetilde{r}))
  , \widetilde{r} ).
\end{equation}
The fact  that $\chi$ is real valued is equivalent to
\begin{equation}
\label{eq:chi02}  c_{m,n} =\overline{c}_{m,n} \text{  and }     C_{m,n
} =\sigma _1 \overline{{C}}_{m,n}.
\end{equation}
We consider the flow $\phi ^{s}$ obtained from the hamiltonian vector field $X_{\chi}$ associated to $\Omega_0$.
Notice that for $\phi:=\phi^1$, we have $\phi ^*\Omega _0=\Omega _0$.  The first question is if $\phi $ exists.

\begin{lemma} \label{lem:ODE}  Fix an arbitrary $l\in \N$ and, for $\epsilon >0$,  set
\begin{equation} \label{eq:domain0}   \begin{aligned}   &   \U _{\epsilon  }  :=\mathcal F_2\(\R\times  D_{\R^4\times P(\omega_*)\Hrad^1}((q(\omega_*),\omega_*,0,0, 0) ,\epsilon ) \) \subset  \Hrad^1.
\end{aligned}    \end{equation}
For the case $l=1$, we assume that $c_{m,n}=S^0_{1,0}$ and $C_{m,n}=S^1_{1,0}$.
Then, there exists $\epsilon_0>0$ s.t.\ if $\epsilon\in (0,\epsilon_0)$, we have
 \begin{equation} \label{eq:reg1}\begin{aligned} &\phi  ^s \in C^\infty(D_\R(0,2)\times  \U _{\epsilon}
  ,   \Hrad^1(\R ^3, \widetilde{\C} )
 )
\end{aligned}   \end{equation}
and there     exist   $\epsilon _1>0$
such that
\begin{equation} \label{eq:main1}\begin{aligned} &  \phi  ^s( \U _{\epsilon _1})\subset   \U _{\epsilon } \text{ for all $|s|< 2$
.}
\end{aligned}   \end{equation}
For every $|s|< 2$  and for $(Q^s, \omega ^s, \lambda ^s, \widetilde{r}^s):=(\phi^{s,*}Q, \phi^{s,*}\omega ,\phi^{s,*} \lambda , \phi^{s,*}\widetilde{r})$  we have $Q^s=Q$ and
\begin{equation} \label{eq:main1symb}\begin{aligned} &  \omega ^s =\omega +S ^{0}_{0,l},\quad
   \lambda ^s =\lambda +S ^{0}_{1,l}, \quad
  \widetilde{r} ^s = e^{\im \sigma _3S ^{0}_{0,l }}\widetilde{r}+S ^{1}_{0,l}.
\end{aligned}   \end{equation}	

	\end{lemma}

\begin{proof}
By \eqref{eq:HamVector}  we have
\begin{equation} \nonumber \begin{aligned}
(X _{\chi})_{\lambda}&= -A(\omega)^{-1}\partial_\omega \chi = S ^{0}_{1,l} \\
(X _{\chi})_{\omega}&= A(\omega)^{-1}\partial_\lambda \chi = S ^{0}_{\delta_{1l},l }
\\ (X _{\chi})_{\widetilde{r}}&=-\im \sigma_3 \nabla_{\widetilde{r}}\chi  = -\im \sigma_3 (\partial_{\rho}\chi +\partial_{\mu}\chi \partial_{\rho}\mu ) \ \widetilde{r} + \sum _{m+n=l } \lambda^{m} \mu^{n}
   C_{m,n
} +\partial_{\mu}\chi \nabla_{\widetilde r}\mu\\&=  S ^{0}_{\delta_{1l},l } \im \sigma _3  \widetilde{r}+S ^{1}_{\delta_{1l},l },
\end{aligned}   \end{equation}
where $\delta_{1l}=1$ for $l=1$ and $\delta_{1l}=0$ otherwise.
The proof of Lemma \ref{lem:ODE}  is based on
\begin{equation} \label{eq:odeest} \begin{aligned} &
\tau ^s= \tau +\int _0^s  \{\tau , \chi \} \circ \phi ^{s'}ds' \text{ for }\tau = \omega , \lambda,
\end{aligned}   \end{equation}
and
\begin{equation} \nonumber \begin{aligned} &
 e^{-\im \sigma _3\int _0^sS ^{0}_{0,l } \circ \phi ^{s'}ds'} \widetilde{r} ^s=\widetilde{r}+  \int _0^{s
 } e^{-\im \sigma _3 \int _0^{s ^{\prime }}S ^{0}_{0,l } \circ \phi ^{s ^{\prime\prime}}d{s ^{\prime\prime}}} S ^{1}_{0,l } \circ \phi ^{s'}ds'.
\end{aligned}   \end{equation}
In the case $l\ge 2$ the error term represented by the integral in the r.h.s.\ is of order $l$ and this leads to the desired result by elementary ODE analysis, see Lemma 3.8  \cite{Cuccagna12Rend}. Even for the case $l=1$, the error term is order $2$  by hypothesis.  The proof of \eqref{eq:main1symb} is similar to that in Lemma 3.8  \cite{Cuccagna12Rend}.
\end{proof}

We now exploit Lemma \ref{lem:ODE} in order to eliminate the term $  \frac{\mu ^2}{2} B (\omega ) $
from $E$ using   $\chi$  with $N=1$ defined  for some for some $\mathrm{c} =S^{0}_{0,0}$ by
 \begin{equation}
\label{eq:chi02}\chi   =     \mathrm{c}(Q,\omega,Q(\widetilde r) )  \ \lambda \mu   .
\end{equation}

\begin{lemma} \label{lem:nform1}
There exists $\mathrm{c}=S^0_{0,0}$ s.t.\ for the  function $\chi$ in  \eqref{eq:chi02}  and  for $\phi _1$ the map associated to $\chi$,
we have $\mathfrak F_1=\mathfrak F_0\circ \phi_1 \in C^\infty(D_{\Hrad^1}(\mathcal T_{\omega_*},\delta),\Hrad^1)$ for some $\delta>0$ satisfying \eqref{error1} and \eqref{Darb}.
Moreover, $\mathfrak F_1$ is a $C^\infty$-diffeomorphism onto the neighborhood of $\mathcal T_{\omega_*}$ and we have
\begin{align}  E_1&:= E \circ \mathfrak F_1\nonumber\\& =d (\omega ) -\omega Q     +(\omega-\omega_*)Q(\widetilde r)+ \frac{\lambda ^2}{2}  A (\omega )+\frac 1 2\<H_{\omega_*}  \widetilde{r}, \sigma _1 \widetilde{r}\>+E_P( \widetilde{r})+S_{1,2}^{0} +\sum_{j=2}^7\mathbf{R}_j,\label{eq:nform11}
\end{align}
where $\mathbf R_j$ are terms in the form given in Definition \ref{def:Rk}.
Moreover,
 \end{lemma}

\begin{remark}
This case is not covered by Lemma \ref{lem:ODE}.
However, by the proof, it will be clear that  $\phi^s$ satsifies \eqref{eq:reg1} and \eqref{eq:main1}.
\end{remark}

\begin{proof}
  We  pick $ \mathrm{c} =\beta(Q,\omega,Q(\widetilde r) )     \lambda \mu$ without initially specifying $\mathrm{c}$.
Let $\phi^s$ the flow obtained from the Hamilton vector field $X_\chi$.
We set $\tau^s=(\phi^s)^*\tau=\tau\circ \phi^s$ for $\tau=\omega,\lambda,\widetilde r$.
By Lemma \ref{lem:ODE}
\begin{align*}
\omega^s=\omega+S^0_{0,1},\ \lambda^s=\lambda+S^0_{1,1}.
\end{align*}
Further, since
\begin{align*}
(X_\chi)_{\widetilde r}=-\im \sigma_3(\partial_\rho \beta \mu \lambda+\beta \partial_\rho \mu \lambda) \widetilde r-\sigma_3\beta \lambda \nabla_{\widetilde r}\mu =-\im \sigma_3 S^0_{0,1}\widetilde r +S^1_{0,2},
\end{align*}
following the proof of Lemma \ref{lem:ODE}, we have
\begin{align*}
\widetilde r^s=e^{\im \sigma_3 S^0_{0,1}}\(\widetilde r+S^1_{0,2}\).
\end{align*}
Notice that the change of coordinate made by $\phi^1$ has almost the same property as the change of coordinate by $\mathfrak F_0$ given in Proposition \ref{prop:3}.
The only difference is that  $(\phi^1)^*\omega=\omega+S^0_{0,1}$ whereas the change of coordinates by Darboux theorem we had $\mathfrak F_0^* \omega =\omega+S^0_{0,2}$.

Thus, we expand $\omega^s$ more carefully.
Since $(X_\chi)_\omega=A^{-1}(\omega)\mathrm{c} (\mu+\lambda\partial_\lambda \mu)$, and by
\eqref{mudsymbol} and \eqref{Sijtrans}     we have
\begin{align*}
\omega^s&=\omega+\int_0^s A^{-1}(\omega^\tau)\mathrm{c}(Q,\omega^\tau,Q(\widetilde r^\tau))\(\mu^\tau+\lambda^\tau \partial_\lambda \mu^\tau\)\,d\tau\\&
=\omega + A^{-1}(\omega)\mathrm{c}(Q,\omega,Q(\widetilde r))\int_0^s \mu(Q,\omega^\tau,\lambda,Q(\widetilde r),\widetilde r)\,d\tau + S^0_{0,2}\\&=
\omega + A^{-1}(\omega)\mathrm{c}\mu s +A^{-1}(\omega)\mathrm{c} \partial_\omega \mu(Q,\omega,0,Q(\widetilde r),0) \int_0^s (\omega^\tau-\omega)\,d\tau + S^0_{0,2},
\end{align*}
where we have used the Taylor expansion in the third line.
Notice that $S^0_{0,2}$ in the second and third line are different.
We set $D(Q,\omega,Q(\widetilde r)):=A^{-1}(\omega)\partial_\omega \mu(Q,\omega,0,Q(\widetilde r),0)=S^0_{1,0}$.
Then, we have
\begin{align*}
\omega^1 = \omega+ A^{-1}\mathrm{c} \mu  \(\frac{e^{D\mathrm{c} }-1}{D\mathrm{c}}\)+S^0_{0,2}=\omega+ A^{-1}\mathrm{c} \mu \sum_{n=1}^\infty \frac{(D\mathrm{c})^{n-1}}{n!}+S^0_{0,2}.
\end{align*}
We now compute $E\circ \mathfrak F_0\circ \phi^1$.
First,
\begin{align*}
d(\omega^1)-\omega^1 Q +(\omega^1-\omega_*)Q(\widetilde r^1)=d(\omega)-\omega Q +(\omega-\omega_*)Q(\widetilde r)+	 \(q(\omega)-Q+Q(\widetilde r)\)(\omega^1-\omega)+S^0_{1,2},
\end{align*}
where we have used $d'(\omega)=q(\omega)$ and $q'(\omega)=S^0_{1,0}$.
Now, since $q(\omega)-Q+Q(\widetilde r)=-A(\omega)\mu+S^0_{0,2}$ by Lemma \ref{lem:muexpand}, we have
\begin{align*}
d(\omega^1)-\omega^1 Q +(\omega^1-\omega_*)Q(\widetilde r^1)=d(\omega)-\omega Q +(\omega-\omega_*)Q(\widetilde r)-\mathrm{c} \mu^2\sum_{n=1}^\infty \frac{(D\mathrm{c})^{n-1}}{n!}+S^0_{1,2}.
\end{align*}
For the expansion of other parts,
from Lemmas \ref{Rexpand} and \ref{lem:transEp}, we have
\begin{align*}
\mathfrak \phi_1^*\(S^0_{1,2}+\sum_{k=2}^7\mathbf R_k+ E_P(\widetilde r)\)=S^0_{1,2}+\sum_{k=2}^7\mathbf R_k + E_P(\widetilde r).
\end{align*}
However, we remark that $S^0_{1,2}$, $\mathbf R_k$ and $S^0_{0,0}$ and $\tilde v$ in the l.h.s.\ and r.h.s.\ are not the same.

\noindent Thus, if we set $\mathrm{c}$ to be the solution of
\begin{align*}
-\mathrm{c} \sum_{n=1}^\infty \frac{(D\mathrm{c})^{n-1}}{n!}+\frac{1}{2}B(\omega)=0,
\end{align*}
we are able to erase the term $\frac{\mu^2}{2}B(\omega)$ from the energy expansion.
Recalling $D=D(Q,\omega,Q(\widetilde r))=S^0_{1,0}$, we can easily solve the above equation by Neumann series.
Therefore, we have the conclusion.
\end{proof}

Before, proceeding to the next step, we remark the following.
\begin{lemma}\label{lem:lininverse}
For any $s\geq 0$,
$\left.-\im \mathcal H_{\omega_*}\right|_{\Sigma^s_c}$ is invertible.
That is, $(\left.-\im \mathcal H_{\omega_*}\right|_{\Sigma^s_c})^{-1}:\Sigma^s_c\to \Sigma^s_c$ exists.
\end{lemma}

\begin{proof}
For $s=0$, the lemma is a consequence of the fact that $\left.-\im \mathcal H_{\omega_*}\right|_{P(\omega_*)\Lrad}$ has no $0$ spectrum.
For $s>0$, for $v\in \Sigma^s_c \subset P(\omega_*)\Lrad$, we can find $u\in P(\omega_*)\Lrad$ s.t.\ $-\im \mathcal H_{\omega_*}u=v$.
Then, by elliptic regularity, we can show $u\in \Sigma^s_c$ (actually, we gain two more derivatives).
\end{proof}


We will prove Theorem \ref{thm:normal} by induction.
We first prove the base case $N=2$.

\begin{lemma}[Theorem \ref{thm:normal} for $N=2$]\label{lem:nform2}
There exists $\chi$ of the form given in \eqref{eq:chi01} with $l=2$ s.t.\ for $\phi _2$ the map associated to $\chi$ we have
$\mathfrak F_2=\mathfrak F_1\circ \phi_2 \in C^\infty(D_{\Hrad^1}(\mathcal T_{\omega_*},\delta),\Hrad^1)$ for some $\delta>0$ satisfying \eqref{error1} and \eqref{Darb}.
Moreover, $\mathfrak F_2$ is a $C^\infty$-diffeomorphism onto the neighborhood of $\mathcal T_{\omega_*}$ and $ E_2:= E_1 \circ \phi_2$ satisfies \eqref{eq:fenergy} with $N=2$.
\end{lemma}

\begin{proof}
We expand $S^0_{1,2}$ in \eqref{eq:nform11}.
Then, we have
\begin{align}
S^0_{1,2}=\sum_{m+n=2}e_{m,n}(Q,\omega,Q(\widetilde r)) \lambda^m \mu^n + \sum_{m+n=1}\lambda^m\mu^n \Omega( E_{m,n}(Q,\omega, Q(\widetilde r)), \widetilde r) +\mathbf R_2+S^0_{0,3},\label{expandS02}
\end{align}
where $e_{m,n}(Q,\omega,Q(\widetilde r))=S^0_{1,0}$ and $E_{m,n}(Q,\omega, Q(\widetilde r))=S^1_{1,0}$.
Further, notice that we can assume $P(\omega_*)E_{m,n}=E_{m,n}$.
We want to erase the terms
\begin{align}\label{nonresonant21}
e_{1,1}\lambda\mu+e_{0,2}\mu^2+\lambda \Omega( E_{1,0}, \widetilde r)+\mu \Omega( E_{0,1}, \widetilde r).
\end{align}
Set
\begin{align*}
\chi=&c_{2,0}\lambda^2+c_{1,1} \lambda \mu+\lambda \Omega (G_{1,0}, \widetilde r)+\mu \Omega( G_{0,1}, \widetilde r),
\end{align*}
where $c_{m,n}=c_{m,n}(Q,\omega, Q(\widetilde r))$ and $G_{m,n}=G_{m,n}(Q,\omega, Q(\widetilde r))$ are to be determined.
We can assume $P(\omega_*)G_{m,n}=G_{m,n}$.
In the end, we will have $c_{m,n}=S^0_{1,0}$ and $G_{m,n}=S^1_{1,0}$.

We define symbols $\widetilde S^k_{i,j}$ similar  $S^k_{i,j}$ but
which are also a function of of $c_{m,n}$ and $G_{m,n}$.
Specifically, we write $ f=f(Q,\omega,\lambda,\widetilde r)=\widetilde S^k_{i,j}$ if there exists $\hat f=\hat f(Q,\omega,\lambda,\mu,\rho,\widetilde r,c_{m,n},G_{m,n})$ with
\begin{align*}
\|\widehat{f}\|_{X ^n_s}\lesssim_s &(|Q-q(\omega_*)|^{1/2}+|\omega-\omega_*|+|\lambda|+|\mu|+\rho^{1/2}+\|\widetilde{r}\|_{\Sigma_{-s}}+|c_{m,n}|+\|G_{m,n}\|_{\Sigma_{-s}})^i\\&\times (|\lambda|+|\mu|+\|\widetilde{r}\|_{\Sigma_{-s}})^j
\end{align*}
and
\begin{align*}
f(Q,\omega,\lambda,\widetilde r)=\hat f(Q,\omega,\lambda,\mu(Q,\omega,\lambda,Q(\widetilde r),\widetilde r),Q(\widetilde r),\widetilde r, c_{m,n}(Q,Q(\widetilde r),\widetilde r), G_{m,n}(Q,Q(\widetilde r),\widetilde r)).
\end{align*}

We now compute the pullbacks by $\phi_2$.
First, since
\begin{align*}
(X_\chi)_\lambda=&-A(\omega)^{-1}\(\partial_\omega \mu c_{1,1}\lambda+\partial_\omega \mu\<\im \sigma_3 G_{0,1},\sigma_1 \widetilde r\>\)+S^0_{0,2}
=\widetilde S^0_{2,1}+S^0_{0,2},
\end{align*}
we have
\begin{align*}
\lambda^s:=\lambda\circ \phi_2^s=\lambda+\widetilde S^0_{2,1}+S^0_{0,2}.
\end{align*}
Similarly,
\begin{align*}
(X_\chi)_\omega&=A(\omega)^{-1}\(2c_{2,0} \lambda+c_{1,1}\mu +\Omega( G_{1,0}, \widetilde r)\)+S^0_{0,2}=\widetilde S^0_{1,1}+S^0_{0,2},\\
(X_\chi)_{\widetilde r}&=-\im \sigma_3\(c_{1,1}\lambda \partial_\rho \widetilde r + \lambda \im \sigma_3 G_{1,0}+\mu \im \sigma_3 G_{0,1}\)+S^0_{0,2}=\widetilde S^0_{1,1} \im \sigma_3 \widetilde r + \lambda G_{1,0}+\mu G_{0,1}+S^1_{2,0}
\end{align*}
Thus, we have
\begin{align*}
\omega^s&:=\omega\circ \phi_2^s=\omega+\widetilde S^0_{1,1}+S^0_{0,2},\\
\widetilde r^s&:=\widetilde r\circ \phi_2^s =e^{ \widetilde S^0_{1,1}\im \sigma_3}\(\widetilde r +\widetilde S^1_{1,1}+S^1_{0,2}\)
\end{align*}
Furthermore, we have
\begin{align*}
Q(\widetilde r^s)&=Q(\widetilde r)+\widetilde S^0_{1,2}+S^0_{1,3},\\
\mu^s&:=\mu\circ \phi_2^s=\mu+\partial_\omega \mu (\omega_1-\omega)+S^0_{0,2}=\mu+\widetilde S^0_{2,1}+S^0_{0,2}.
\end{align*}
Using the above, we can rewrite the expansion of $\omega^1$ and $\widetilde r^1$ as
\begin{align*}
\omega^1&=\omega+A(\omega)^{-1}\(2c_{2,0}\int_0^1 \lambda^s\,ds+c_{1,1}\int_0^1 \mu^s\,ds +\Omega (G_{1,0}, \int_0^1 \widetilde r^s\,ds)\)+S^0_{0,2}\\&=\omega+A(\omega)^{-1}\(2 c_{2,0} \lambda + c_{1,1}\mu +\Omega( G_{1,0},  \widetilde r)\)+\widetilde S^0_{2,1}+S^0_{0,2} \ ;\\
\widetilde r^1&=e^{\im \sigma_1 \widetilde S^0_{1,1}}\(\widetilde r +(\lambda+\widetilde S^0_{2,1})G_{1,0}+(\mu+\widetilde S^0_{2,1})G_{0,1} +S^1_{0,2}\).
\end{align*}
We now substitute this into the energy expansion \eqref{eq:nform11}.
 One of the contributions is
\begin{align*}
&d(\omega^1)-\omega^1 Q + (\omega^1-\omega_*) Q(\widetilde r^1)=d(\omega)-\omega Q + (\omega-\omega_*) Q(\widetilde r)+(q'(\omega)-Q+Q(\widetilde r))(\omega^1-\omega)\\&\quad+\widetilde S^0_{2,2}+S^0_{0,3}\\&
=d(\omega)-\omega Q + (\omega-\omega_*) Q(\widetilde r)-2 c_{2,0} \lambda \mu- c_{1,1}\mu^2 -\mu\Omega( G_{1,0}, \widetilde r)+\widetilde S^0_{2,2}+S^0_{0,3}.
\end{align*}
The next term will not give a significant contribution:
\begin{align*}
\frac{1}{2}A(\omega) (\lambda^1)^2=\frac{1}{2}A(\omega) \lambda^2+\widetilde S^0_{2,2}+S^0_{0,3}.
\end{align*}
For the main term in $\widetilde r$, we have
\begin{align*}
\frac12 \<H_{\omega_*}\widetilde r^1, \sigma_1 \widetilde r^1\>=&\frac12 \<H_{\omega_*}\widetilde r, \sigma_1 \widetilde r\>+(\lambda+\widetilde S^0_{2,1})\Omega(-\im\mathcal H_{\omega_*}G_{1,0}, \widetilde r)+(\mu+\widetilde S^0_{2,1})\Omega(-\im \mathcal H_{\omega_*}G_{0,1}, \widetilde r)\\&+\widetilde S^0_{2,2}+S^0_{0,3}+\mathbf R_2.
\end{align*}
There will be no significant contributions from the remaining terms.
Thus, we have
\begin{align*}
&E^{(2)}=d (\omega ) -\omega Q     +(\omega-\omega_*)Q(\widetilde r)+ \frac{\lambda ^2}{2}  A (\omega )+\frac 1 2\<H_{\omega_*}  \widetilde{r}, \sigma _1 \widetilde{r}\>+E_P( \widetilde{r})\\&
\quad-2 c_{2,0} \lambda \mu- c_{1,1}\mu^2 -\mu \Omega( G_{1,0},\widetilde r)+(\lambda+\widetilde S^0_{2,1})\Omega(-\im \mathcal H_{\omega_*}G_{1,0}, \widetilde r)+(\mu+\widetilde S^0_{2,1})\Omega(-\im H_{\omega_*}G_{0,1}, \widetilde r )\\&\quad+\sum_{m+n=2}e_{m,n} \lambda^m \mu^n + \sum_{m+n=1}\lambda^m\mu^n \Omega(E_{m,n},  \widetilde r) +\widetilde S^0_{2,2}+S^0_{0,3} +\sum_{j=2}^7\mathbf{R}_j
\end{align*}
Therefore, to eliminate the terms \eqref{nonresonant21}, it suffices to solve the system
\begin{align*}
-2 c_{2,0}+\widetilde S^0_{2,0}+e_{1,1}=0, \text{ where }e_{1,1}=e_{1,1}(Q,\omega , \rho )\\
-c_{1,1}+\widetilde S^0_{2,0}+e_{0,2}=0,\text{ where }e_{0,2}=e_{0,2}(Q,\omega , \rho )\\
G_{1,0}+(-\im \mathcal H_{\omega_*}) ^{-1}\widetilde S^1_{2,0}+(-\im \mathcal H_{\omega_*}) ^{-1}E_{1,0}=0,\text{ where }E_{1,0}=E_{1,0}(Q,\omega , \rho )\\
- G_{0,1}- (-\im \mathcal H_{\omega_*}) ^{-1}G_{1,0}+(-\im \mathcal H_{\omega_*}) ^{-1}\widetilde S^1_{2,0}+(-\im \mathcal H_{\omega_*}) ^{-1}E_{0,1}=0\text{ where }E_{0,1}=E_{0,1}(Q,\omega , \rho ),
\end{align*}
where we exploited the invertibility of $\left . -\im \mathcal H_{\omega_*}\right |_{\Sigma ^s_c}$
and where the $\widetilde S^i_{2,0}$ are $C^\infty$ functions in the independent variables
$(c_{2,0},c_{1,1}, G_{1,0} , G_{0,1}, Q,\omega , \rho )$. Now, at
\begin{equation*}
   (c_{2,0},c_{1,1}, G_{1,0} , G_{0,1}, Q,\omega , \rho ) = (0,0, 0 , 0, q(\omega _*),\omega _*, 0 )
\end{equation*}
all the terms in the above system are 0, because of $e_{m,n} =S^0_{1,0}$ and $E_{m,n} =S^1_{1,0}$ and of the
properties of the $\widetilde S^i_{2,0}$. The functions
$c_{m,n}(Q,\omega , \rho )= S^0_{1,0}$ and $G_{m,n}(Q,\omega , \rho )= S^1_{1,0}$
remain determined by Implicit Function Theorem.

\end{proof}

We now consider the inductive step for $N>2$.

\begin{lemma}\label{lem:ind}
Suppose that for $N\geq 2$, there exists a $C^\infty$-diffeomorphism $$\mathfrak F_N\in C^\infty(D_{\Hrad^1}(\mathcal T_{\omega_*},\delta),\Hrad^1),$$ for some $\delta>0$ satisfying \eqref{error1}, \eqref{Darb} and \eqref{eq:mmain1}.
Then, there exists $\chi$ of the form given in \eqref{eq:chi01} with $l=N+1$ s.t.\ for $\phi_{N+1}$ the canonical map associated to $\chi$, we have $\mathfrak F_{N+1}=\mathfrak F_N\circ \phi_{N+1}\in C^\infty(D_{\Hrad^1}(\mathcal T_{\omega_*}),\Hrad^1)$ satisfying \eqref{error1}--\eqref{Darb}.
Moreover, $\mathfrak F_{N+1}$ is a $C^\infty$-diffeomorphism onto the neighborhood of $\mathcal T_{\omega_*}$ and $E_{N+1}:=E_{N}\circ \phi_{N+1}$ satisfies \eqref{eq:mmain1} with $N$ replaced by $N+1$.
\end{lemma}

\begin{proof}
We expand the $S^0_{0,N+1}$ term in $E_N$ as
\begin{align*}
S^0_{N+1}=\sum_{m+n=N+1} e_{m,n}(Q,\omega,Q(\widetilde r)) \lambda^m \mu^n + \sum_{m+n=N} \lambda^m \mu^n \Omega(E_{m,n}(Q,\omega,Q(\widetilde r)),\widetilde r)+\mathbf R_2+ S^0_{0,N+2}.
\end{align*}
Our aim is to erase the first two terms except the term $e_{N+1,0} \lambda^{N+1}$, which is absorbed into $E _{f,N+1}$.
We consider a transformation $\phi_{N+1}$ like in Lemma \ref{lem:ODE} with $l=N+1$.
By \eqref{eq:HamVector}, \eqref{mudsymbol}  and Lemma \ref{lem:ODE}, we have
\begin{align*}
\omega^1=&\omega+A^{-1}
\(\sum_{m+n=N+1}m c_{m,n} \lambda^{m-1}\mu^n +\sum_{m+n=N}m \lambda^{m-1}\mu^n \Omega(C_{m,n},\widetilde r)\)+S^0_{0,N+1},\\
\lambda^1=&\lambda+ A^{-2}q'(\omega)\(\sum_{m+n=N+1}n c_{m,n} \lambda^m \mu^{n-1}+\sum_{m+n=N}n \lambda^m \mu^{n-1} \Omega(C_{m,n},\widetilde r)\)+S^0_{0,N+1},\\
\widetilde r=&e^{\im \sigma_3 S^0_{0,N}}\(\widetilde r +\sum_{m+n=N}\lambda^m \mu^n C_{m,n}+S^0_{0,N+1}\).
\end{align*}
Thus, we obtain, for $A_2= A +  e _{2,0}^{(N)}=A+S^{0}_{1,0}$,
\begin{align*}
E_{N+1}=&d(\omega)-(\omega-\omega_*)Q(\widetilde r)+\frac12 \lambda^2 A (\omega)    + \sum  _{m=2}^{N}   \lambda ^m     e _{m,0}^{(N)} +e_{N+1,0} \lambda^{N+1}   +\frac12 \Omega(-\im\mathcal H_{\omega_*}\widetilde r, \widetilde r)\\&
+\sum_{m+n=N+1}m c_{m,n} \lambda^{m-1}\mu^{n+1}+\sum_{m+n=N}m \lambda^{m-1}\mu^{n+1} \Omega(C_{m,n},\widetilde r)\\&+  A_2(\omega)A^{-2}(\omega)q'(\omega)\(\sum_{m+n=N+1}nc_{m,n}\lambda^{m+1}\mu^{n-1}
+\sum_{m+n=N}n \lambda^{m+1}\mu^{n-1}\Omega(C_{m,n},\widetilde r)\)\\&
+\sum_{m+n=N}\lambda^m \mu^n\Omega(-\im \mathcal H_{\omega_*}C_{m,n},\widetilde r)
+\sum_{m+n=N+1}e_{m,n} \lambda^m \mu^m +\sum_{m+n=N}\lambda^m \mu^m \Omega(E_{m,n},\widetilde r)\\&
+E_P(\widetilde r)+S^0_{0,N+2}+\sum_{k=2}^7 \mathbf R_k.
\end{align*}
Thus, it suffices choose $c_{m,n}$ and $C_{m,n}$ s.t.
\begin{align}
&(m+1)c_{m+1,n-1}+  A_2  A^{-2}  q' (n+1)c_{m-1,n+1}+e_{m,n}=0,\quad 0\leq m\leq N,\label{homeq1}\\&
-\im \mathcal H_{\omega_*} C_{m,n}+(m+1)C_{m+1,n-1}+ A_2  A^{-2}  q' (n+1)C_{m-1,n+1}+E_{m,n}=0,\quad 0\leq m\leq N.\label{homeq2}
\end{align}
Here, we are setting $c_{-1,N+2}=0$ and $C_{N+1,-1}=C_{-1,N+1}=0$.
One can easily solve \eqref{homeq1}   by Neumann series.
For the second equation, we express it in a matrix form
\begin{align*}
\(\mathbf H_N+A_2A^{-2} q' \mathbf A_N\)\mathbf C_N =-\mathbf E_N,
\end{align*}
where
\begin{align*}
\mathbf C_N={^t(C_{0,N},C_{1,N-1},\cdots,C_{N,0})},\quad \mathbf E_N={^t(E_{0,N},E_{1,N-1},\cdots,E_{N,0})}
\end{align*}
and
\begin{align*}
\mathbf H_N:= -\im \mathcal H_{\omega_*} I_{N} +\mathbf B_N,
\end{align*}
and $I_N$ is a $(N+1)\times (N+1)$ unit matrix, $\mathbf A_N$, $\mathbf B_N$ are $(N+1)\times (N+1)$ matrices with
\begin{align*}
(\mathbf A_N)_{i,j}= (N+2-i)\delta_{i-1,j},\quad (\mathbf B_N)_{i,j}=i \delta_{i+1,j}.
\end{align*}
Now, $\mathbf H_N$ is invertible in $\Sigma^s$ for arbitrary $s$ because $\mathbf H_N$ is an upper triangular matrix and $-\im \mathcal H_{\omega_*}$ is invertible by Lemma \ref{lem:lininverse}.
Next, for each $s\geq 0$, taking $\delta_s$   small s.t.\ $\|  A_2A^{-2} q' \mathbf H_N^{-1}\mathbf A_N\|_{\mathcal L(\(\Sigma^s)^{N+1}\)}<1$, we can invert $\mathbf H_N + A_2A^{-2} q'\mathbf A_N$.
\end{proof}

By Lemmas \ref{lem:nform2} and \ref{lem:ind}, we have completed the proof of Theorem \ref{thm:normal}.
\begin{remark}\label{rem:FGR}   Prop. \ref{prop:1} tells us that the nonzero eigenvalues of $-\im \mathcal{H}_{\omega}$ are  given by $\pm \im \lambda$ with $\lambda (\omega ) =\pm \sqrt{ A(\omega) ^{-1}q'(\omega )}$. For any given $N$ here    we take $\omega$ sufficiently close to $\omega _*$
 s.t.\ $N \lambda (\omega )\ll \omega  $ when $\omega   > \omega _*$, so that we do not see the dissipation by radiation phenomena detected in \cite{Bambusi13CMPas}, \cite{BP95}--\cite{BS03AIHPN},
 \cite{Cuccagna03RMP}--\cite{CPV05CPAM},  \cite{SW99IM} and other works referenced in these papers.  Taking a really large number of iterates $N$, so that  $N \lambda ( \omega )\sim \omega   $,
 would lead to a $A_2(\omega )A^{-2}(\omega ) q' (\omega ) \mathbf H_N^{-1}\mathbf A_N$ not small, obviously disrupting the above proof. Also,
 by the large number of coordinate changes, the size of   the coefficients of the polynomials
 \eqref{eq:chi01} could be not simple to control. Furthermore,   $\omega$ is oscillating
 around $\omega  _*$ and so it is natural to ask which $\omega$ should be
 the ones used for the iterates.   It would be natural to choose values of
 $\omega$ close to  its limiting value, assuming   Conjecture 1.1  \cite{MRS10JNS} true.
 But, to the best of our knowledge, there is  no information on  this limit value,
 save for the fact that it should be strictly larger that $\omega _*$. All this suggests that, as we mentioned in  Sect.\ref{sec:introduction},  the framework in this paper is not the correct one to prove the conjecture and that a better choice of modulation coordinates is required.
\end{remark}

\section{Long time oscillation}
\label{sec:contspec}
In this section, we consider the dynamics of \eqref{normeq}, which is equivalent (modulo $\vartheta$) to NLS \eqref{1} in the coordinate system given by Theorem \ref{thm:normal}.
We will consider only the situation $0<Q-q(\omega_*)\ll 1$.
Therefore, we set
\begin{align}\label{defep}
\epsilon:=(Q-q(\omega_*))^{1/2},
\end{align}
and use $\epsilon$ to measure the smallness.

\subsection{Dynamics of finite dimensional system}\label{sec:findim}
Fix $N\geq 2$.
We start from studying the finite dimensional system \eqref{finiteeq}:
\begin{equation}\tag{\ref{finiteeq}}
\begin{aligned}
\dot \omega_f &= A(\omega_f)^{-1}\partial_\lambda E_Q(\omega_f,\lambda_f ),\\
 \dot \lambda_f &=- A(\omega_f)^{-1}\partial_{\omega}E_Q(\omega_f,\lambda_f ),
\end{aligned}
\end{equation}
where $E_Q(\omega,\lambda)=E_{f,N}(Q, \omega,\lambda , 0 )$, with $E_{f,N}$ defined in \eqref{eq:fenergy}.
As explained in Sect.\ \ref{subsec:main}, $V_Q$ can be thought as the potential energy and the remaining terms, which sum up  approximately to $2^{-1}A(\omega_*)\lambda^2$, can be thought as the kinetic energy.

We first study the critical points of $V_Q$.
The following is elementary.
\begin{lemma}\label{lem:findimEcrt}
Let $0<Q-q(\omega_*)\ll1$.
Then, near $\omega _*$ the potential $V_Q$ has two critical points, $\omega _\pm $  with \begin{equation}\label{eq:findimEcrt}  \begin{aligned}   &      \omega _\pm  =\omega _*\pm \epsilon \sqrt{\frac{2}{q ^{\prime\prime}(\omega _*)}}  +O(\epsilon^2)  , \end{aligned}\end{equation}
  with $\omega _-$ a local maximum and $ \omega _+$ a local minimum of $V_Q$,  with $\epsilon$ defined in \eqref{defep} .
\end{lemma}

\begin{proof}
Since (recall  $d(\omega ):=E(\widetilde{\phi}_\omega ) +\omega q(\omega )$
and so $d'(\omega )=q(\omega ),  $ since \eqref{3} implies $\partial _{\omega} E(\widetilde{\phi}_\omega ) + \omega q'(\omega )=0$)
$$V_Q'(\omega)=q(\omega)-Q=q(\omega_*)+\frac12 q''(\omega_*)(\omega-\omega_*)^2+O((\omega-\omega_*)^3)-Q,$$
we only have to find the solutions of
\begin{align*}
\omega-\omega_* =\pm\sqrt{\frac{2}{q''(\omega_*)}\(\epsilon^2+O((\omega-\omega_*)^3)\)}.
\end{align*}
We can solve this for example by fixed point argument and obtain the conclusion
using $\text{sign} V_Q ^{\prime\prime}(\omega_\pm )=\mp$.
\end{proof}

Since $E_{Q}$ is the Hamiltonian of \eqref{finiteeq}, it is a constant of the motion.
Notice also that $(\omega_+,0) $ is an isolated local minimum  $(\omega_-,0) $ is a
 saddle for
$(\omega , \lambda )\to E_{Q}( \omega ,\lambda  )$.  Thus,
  if $E_{Q}( \omega_f,\lambda_f )<V_Q(\omega_-)$ and $\omega_f(0)>\omega_-$, then the solution $(\omega_f(t),\lambda_f(t))$ will be trapped in this potential well.

\begin{lemma}\label{lem:finiteest}
Let $0<Q-q(\omega_*)\ll1$ and
\begin{align}\label{trap}
\omega_f(0)>\omega_-,\quad E_{Q}( \omega_f(0),\lambda_f(0) )<V_Q(\omega_-).
\end{align}
Then, for all $t\in \R$ we have
\begin{align}\label{varsigmalambdabound}
|\omega_f(t)-\omega_*|\lesssim \epsilon,\quad |\lambda_f(t)|\lesssim \epsilon^{3/2},
\end{align}
where $(\omega_f(t),\lambda_f(t))$ are solutions of \eqref{finiteeq}.
Moreover, there exists $T_f\ge  0$ depending on $Q$ and $E_{f,N}(Q,\omega(0),\lambda_f(0),0)$ s.t.\ $(\omega_f(t),\lambda_f(t))$ is a  periodic functions of $t$ with period $T_f$.
\end{lemma}

\begin{proof}
By the energy conservation and the trapping condition, it is obvious that the solutions are periodic with some period $T_f$.
Notice that if $(\omega_f(0),\lambda_f(0))=(\omega_+,0)$, then the solution is stationary.

\noindent To prove \eqref{varsigmalambdabound}, we estimate $V_Q(\omega_-)-V_Q(\omega_+)$.
Since
\begin{align*}
V_Q(\omega)=V_Q(\omega_*)-(Q-q(\omega_*)) (\omega-\omega_*)+\frac16 q''(\omega_*)(\omega-\omega_*)^3 + O((\omega-\omega_*)^4),
\end{align*}
 using \eqref{defep} and  \eqref{eq:findimEcrt}   and elementary computations we get
\begin{align}
V_Q(\omega_-)-V_Q(\omega_+)&=(Q-q(\omega_*))(\omega_+-\omega_-)+\frac16 q''(\omega_*)\((\omega_--\omega_*)^3
-(\omega_+-\omega_*)^3\)+O(\epsilon^4)\nonumber\\&=
\frac43 \epsilon^3\sqrt{\frac{2}{q''(\omega_*)}}+O(\epsilon^4).\label{potentialhigh}
\end{align}
By a   similar computation, we can show $\omega_{++}-\omega_*\sim \epsilon$, for $\omega_{++}>\omega_+$   the solution of $V_Q(\omega_{++})=V_Q(\omega_-)$.
Since $\omega(t)\in (\omega_-,\omega_{++})$ for all $t$, we obtain the first estimate of \eqref{varsigmalambdabound}.
Turning to the second estimate, since the kinetic energy is $ \frac{1}{2} A (\omega ) \lambda ^2     +\sum  _{m=2}^{N}   \lambda ^m     e _{m,0}^{(N)}(Q, \omega , 0))\sim \lambda^2$, then $E_{Q}<V_Q(\omega_-)$ implies
\begin{align*}
\lambda^2\lesssim  V_Q(\omega_-) -V_Q(\omega_+)\lesssim \epsilon^3.
\end{align*}
\end{proof}

\begin{remark}\label{rem:trap}
The estimates in \eqref{varsigmalambdabound} simply come  from the trapping condition  \eqref{trap}.
Therefore, even if $(\omega,\lambda)$ is not the solution of \eqref{finiteeq}, we can use the bound $\eqref{varsigmalambdabound}$ if they satisfy \eqref{trap}.
\end{remark}

From \eqref{varsigmalambdabound}, it is natural to rescale the problem using \eqref{defep}.
In particular, we set
\begin{align}\label{rescale:f}
\omega_f = \omega_+ + \epsilon \zeta_f,\quad \lambda_f = \epsilon^{3/2} \kappa_f,\quad t=\epsilon^{-1/2}\tau.
\end{align}
Then, setting\begin{equation}\label{rescale:fenrgy}
\begin{aligned}
&\widetilde E_Q(\zeta_f,\kappa_f):=\epsilon^{-3} E_Q(\omega_f,\lambda_f)\\& =\sqrt{\frac{q''(\omega_*)}{2}}(1+O(\epsilon)) \zeta_f^2+\frac16 q''(\omega_*) \zeta_f^3+\frac12 A(\omega_*)\kappa_f^2(1+O(\epsilon)+O(\epsilon  \kappa_f) +O(\epsilon  \zeta_f) ),
\end{aligned}
\end{equation}
the system \eqref{finiteeq} can be written as
\begin{equation}\label{rescale:feq}
\begin{aligned}
\dot \zeta_f &= A(\omega_+ +\epsilon \zeta_f)^{-1}\partial_\kappa \widetilde E_Q(\zeta_f,\kappa_f),\\
\dot \kappa_f &= -A(\omega_+ +\epsilon \zeta_f)^{-1}\partial_\zeta \widetilde E_Q(\zeta_f,\kappa_f),
\end{aligned}
\end{equation}
where $\dot\zeta_f=\frac{d}{d\tau} \zeta_f$ and $\dot\kappa_f=\frac{d}{d\tau} \kappa_f$

We now prove Proposition \ref{finitedimosc}.
Because of \eqref{lem:finiteest}, it suffices to estimate $T_f$.
Notice that the condition \eqref{finiteenergybound} is excluding the case $(\omega_f(t),\lambda_f(t))\equiv (\omega_+,0)$.

\begin{proof}[Proof of Proposition \ref{finitedimosc}]
   First, the rescaled equation can be written as
\begin{equation}\label{rescale:feq2}
\begin{aligned}
\dot \zeta_f &= \kappa_f (1+O(\epsilon)) ,\\
\dot \kappa_f &= - c_0 (1+O(\epsilon))\zeta_f+ O(\zeta_f^2)+O(\epsilon  ) \kappa_f ^2  \text{ with }c_0
=
A(\omega_*)^{-1} \sqrt{2q''(\omega_*)}  .
\end{aligned}
\end{equation}
If $0<\widetilde E_Q(\zeta_f,\kappa_f)\ll1$, then   $|\zeta_f|+|\kappa_f|\ll1$.
Thus, the $\zeta_f^2$ term in \eqref{rescale:feq2} is negligible. Elementary plane phase analysis shows that
$\zeta_f$ varies monotonically  from  its minimum $\zeta_{min}<0$ to its maximum $\zeta_{max}>0$ and then comes back.
Let us see that it takes $\tau\sim 1$ time  to vary from $\zeta_{min} $ to 0, and
since for the other parts of the motion  the time frame  is similar,
we can conclude that the rescaled  period $\tau _f$ is approximately 1.
We have
\begin{equation}\nonumber
\begin{aligned} \frac{d\kappa}{dt}= \frac{d\kappa}{d\zeta}\frac{d\zeta}{dt} = \frac{1}{2}\frac{d \kappa ^2}{d\zeta}
(1+O(\epsilon ) ) =- c_0 (1+O(\epsilon))\zeta + O(\zeta ^2)+O(\epsilon  ) \kappa  ^2  .
\end{aligned}
\end{equation}
An elementary integrating factor argument leads to
\begin{equation}\nonumber
\begin{aligned}    \frac{d \kappa ^2}{d\zeta}
  = -2 c_0 (1+O(\epsilon)+O(\zeta ))\zeta
\end{aligned}
\end{equation}
 and  since, as we mentioned above,  here we assume $|\zeta |\ll 1$, we have
\begin{equation}\nonumber
\begin{aligned}    & \kappa (\zeta )\approx  \sqrt{     c_0 (\zeta _{min}^2  -\zeta ^2 ) }
\end{aligned}
\end{equation}
and so
\begin{equation}\nonumber
\begin{aligned}    & \tau =   \int _{\zeta _{min}}^{0}  \frac{d\zeta}{\kappa (\zeta )}\approx 1.
\end{aligned}
\end{equation}
 Therefore, the period $\tau_f$ in the rescaled time $\tau$ satisfies $\tau_f\sim 1$, which implies $T_f\sim \epsilon^{-1/2}$.


\end{proof}

\subsection{Dispersive  Estimates}

In this subsection, we consider  several estimates based on dispersion  for the proof of Theorem \ref{thm:dynamics}.
We use the following standard Strichartz spaces.
\begin{definition}
For $s\geq 0$ and $I=[t_1,t_2]$, we set
\begin{align*}
\stz^s(I):=L^\infty H^s(I)\cap L^2 W^{s,6}(I),\quad \stz^{*,s}(I):=L^1 H^s(I) + L^2 W^{s,6/5}(I),
\end{align*}
and
\begin{align*}
\|\widetilde u\|_{\stz^s(I)}&:=\|\widetilde u\|_{L^\infty_t(I, H^s)} + \|\widetilde u\|_{L^2_t(I, W^{s,6})},\\
\|\widetilde u\|_{\stz^{*,s}(I)}
&:=\inf_{\tilde u=\tilde v+\tilde w}\(\|\widetilde v\|_{L^1_t(I, H^s)} + \|\widetilde w\|_{L^2_t(I, W^{s,6/5})}\).
\end{align*}
\end{definition}

\begin{definition}
Let $I\subset \R$ be an interval and $F\in C(\R;\R)$.
For $t_1,t \in I$ and $\widetilde r_0\in P(\omega_*)H^1(\R^3,\widetilde \C)$, we denote by $U_F(t_1,t)\widetilde r_0\in P(\omega_*)H^1(\R^3,\widetilde \C)$   the solution of
\begin{align}
\im\partial_t \widetilde r =(\mathcal H_{\omega_*}+F(t)P(\omega_*)\sigma_3)\widetilde r,\quad \widetilde r(t_1)=\widetilde r_0.
\end{align}
\end{definition}

\begin{lemma}\label{lem:stz}
Let $I\subset \R$ and $t_0\in I$.
There exists $\delta>0$ s.t.\ for $I\subset \R$ and $t_0\in I$, if $\|F\|_{L^\infty(I)}<\delta$, then we have
\begin{align}
\|U_F(\cdot,t_0)\widetilde r\|_{\stzo}\lesssim \|\widetilde r\|_{H^1},\label{homstz}\\
\|\int_{t_0}^\cdot U_F(t,s)\widetilde f(s)\,ds\|_{\stzo}\lesssim \|\widetilde f\|_{\stzso}\label{inhomstz}
\end{align}
\end{lemma}

Before the proof, we recall the following result from \cite{Cuccagna03RMP}, which was inspired by
an earlier statement in \cite{BP95}, see also \cite{BS03AIHPN}.
\begin{lemma} \label{lem:03RMP} Consider the strong limits
\begin{equation} \label{eq:03RMP1} \begin{aligned} &
W = s- \lim _{t\to +\infty } e ^{\im t \mathcal{H}_{\omega _*}}  e^{-\im \sigma _3(-\Delta +\omega _*  )t }  \text{  and } Z = s- \lim _{t\to +\infty }  e^{ \im \sigma _3(-\Delta +\omega _*  )t }e ^{-\im t \mathcal{H}_{\omega _*}} P (\omega _*)
\end{aligned}\end{equation}
and set
\begin{equation} \label{eq:03RMP2} \begin{aligned} &
 P_+ =W   \begin{pmatrix}
\uno _\C & 0   \\
0 &    0
\end{pmatrix}
Z  \,  , \,
P_- =W  \begin{pmatrix}
0 & 0   \\
0 &    \uno _\C
\end{pmatrix}
 Z    \,  .
\end{aligned}\end{equation}
Then, $W$ initially defined on $C_0^\infty(\R^3,\widetilde \C)$ can be extended to an isomorphism $W:L^p(\R^3,\widetilde \C)\to P(\omega_*) L^p(\R^3,\widetilde \C)$ for all $p\in [1,\infty]$ with $W^{-1}=Z$.
Furthermore, for any $l$  we have
\begin{equation} \label{eq:03RMP3} \begin{aligned} &
 \|   P (\omega _*) \sigma _3-
(P_+-P_- )   \| _ { \Sigma ^{-l}\to \Sigma ^{ l}} \le c_l <\infty .
\end{aligned}\end{equation}
\end{lemma}

\begin{proof}[Proof of Lemma \ref{lem:stz}]
Set $\widetilde U_F(t,t_0)\widetilde r_0$ to be the solution of
\begin{align*}
\im \widetilde r_t = \mathcal H_{\omega_*}\widetilde r + F(t)(P_+-P_-)\widetilde r,\quad \widetilde r(t_0)=\widetilde r_0.
\end{align*}
Then, since
\begin{align*}
\mathcal H_{\omega_*}P(\omega_*)=\mathcal H_{\omega_*}WZ=W \sigma_3(-\Delta+\omega_*)Z,
\end{align*}
we have
\begin{align*}
\mathcal H_{\omega_*}+F(t)\(P_+-P_-\)=W\sigma_3(-\Delta+\omega_*+F(t))Z.
\end{align*}
By the $L^p$ boundedness of $W$ and $Z$, we have \eqref{homstz} and \eqref{inhomstz} for $\widetilde U_F$.

Now, since $U_F(t,t_0) \widetilde r_0$ is the solution of
\begin{align*}
\im r_t =(\mathcal H_{\omega_*}+F(t)(P_+-P_-))\widetilde r + F(t)\(P(\omega_*)\sigma_3 -(P_+-P_-)\)\widetilde r,\quad \widetilde r(t_0)=\widetilde r_0,
\end{align*}
we have
\begin{align*}
\widetilde r(t) = \widetilde U_F(t,t_0)\widetilde r_0 -\im \int_{t_0}^t \widetilde U_F(t,s)F(s)\(P(\omega_*)\sigma_3 -(P_+-P_-)\)\widetilde r(s)\,ds,
\end{align*}
where $r(t)=U_F(t,t_0)\widetilde r_0$.
Thus,
\begin{align*}
\|\widetilde r\|_{\stz^s}&\lesssim \|\widetilde r_0\|_{H^s}+\delta\|\(P(\omega_*)\sigma_3 -(P_+-P_-)\)\widetilde r\|_{L^2 W^{s,6/5}}\\&\lesssim \|\widetilde r_0\|_{H^s}+\delta \|\widetilde r\|_{\stz^s}.
\end{align*}
Therefore, for sufficiently small $\delta$, we have \eqref{homstz}.

Next, since $-\im \int_{t_0}^t U_F(t,s)f(s)\,ds$ is the solution of
\begin{align*}
\im r_t =(\mathcal H_{\omega_*}+F(t)(P_+-P_-))\widetilde r + F(t)\(P(\omega_*)\sigma_3 -(P_+-P_-)\)\widetilde r+f,\quad \widetilde r(t_0)=0,
\end{align*}
we have
\begin{align*}
\widetilde r(t)=-\im\int_{t_0}^t\widetilde U_F(t,s)F(s)\(P(\omega_*)\sigma_3-(P_+-P_-)\)\widetilde r(s)\,ds-\im\int_{t_0}^t\widetilde U_F(t,s)f(s)\,ds,
\end{align*}
where $\widetilde r(t)=-\int_{t_0}^t U_F(t,s)f(s)\,ds$.
Thus, as before
\begin{align*}
\|\widetilde r\|_{\stz^s}\lesssim \delta\|\widetilde r\|_{\stz^s}+\|\widetilde f\|_{\stz^{*,s}}.
\end{align*}
Therefore, we have \eqref{inhomstz} for $\delta>0$ sufficiently small.
%
%
%
\end{proof}

The following estimate is standard.

\begin{lemma}\label{lem:nonlinearity}
We have
\begin{align*}
\|g(|\widetilde r|^2)\widetilde r\|_{\stzso}\lesssim \|\widetilde r\|_{\stzo}^{3}+ \|\widetilde r\|_{\stzo}^{\max(3,2p+1)}.
\end{align*}
\end{lemma}

\begin{proof}
Since $g(0)=0$, $|g^{(n)}(s)|\lesssim s^{p-n}$ for $s\geq 1$ and some $p<2$, we have
\begin{align*}
|g(|\widetilde r|^2)\widetilde r|\lesssim |\widetilde r|^3+ |\widetilde r|^{\max(3,2p+1)},\quad |\nabla_x \(g(|\widetilde r|^2)\widetilde r\)|\lesssim \(|\widetilde r|^2+|\widetilde r|^{\max(2,2p)}\)|\nabla_x \widetilde r|.
\end{align*}
Therefore,
\begin{align*}
\|g(|\widetilde r|^2)\widetilde r\|_{\stzso} &\lesssim \| |\widetilde r|^{3}\|_{L^2L^{6/5}}+\||\widetilde r|^{2}\nabla_x \widetilde r\|_{L^2L^{6/5}}+\| |\widetilde r|^{\max(3,2p+1)}\|_{L^2L^{6/5}}+\||\widetilde r|^{\max(2,2p)}\nabla_x \widetilde r\|_{L^2L^{6/5}}\\&
\lesssim  \|\widetilde r\|_{L^\infty L^{3}}^{2}\|\widetilde r\|_{L^2 W^{1,6}}+\|\widetilde r\|_{L^\infty L^{\max(3,3p)}}^{\max(2,2p)}\|\widetilde r\|_{L^2 W^{1,6}}\lesssim \|\widetilde r\|_{\stzo}^{3}+\|\widetilde r\|_{\stzo}^{\max(3,2p+1)},
\end{align*}
where we have used the embedding $L^3\hookrightarrow H^1$ and $L^{\max(3,3p)}\hookrightarrow H^1$.
\end{proof}

In the following, we consider the situation when $(\omega,\lambda)$ is trapped in the potential well (i.e. $(\omega(t),\lambda(t))$ satisfies \eqref{trap} for all $t\in I$ for some $I\subset \R$).
By Lemma \ref{lem:finiteest}, this condition guarantees $\|\omega-\omega_*\|_{L^\infty(I)}\lesssim \epsilon$ and $\|\lambda\|_{L^\infty}\lesssim \epsilon^{3/2}$.

\begin{lemma}\label{lem:estmu}
Let $I\subset \R$ be an interval and assume $\|\omega-\omega_*\|_{L^\infty(I)}\lesssim \epsilon$ and $\|\lambda\|_{L^\infty}+\|\widetilde r\|_{L^\infty H^1(I)}\lesssim \epsilon^{3/2}$.
Then,  the function $\mu$ in \eqref{mudsymbol} satisfies
\begin{align*}
\|\mu\|_{L^\infty(I)}\lesssim \epsilon^2.
\end{align*}
\end{lemma}

\begin{proof}
Immediate from Lemmas \ref{lem:pdsymbol} and \ref{lem:muexpand}.
\end{proof}

We estimate $\mathbf R=S^0_{0,N+1}+\sum_{k=2}^7\mathbf R_k$.

\begin{lemma}\label{lem:boundfatR}
Let $T>0$ and set $I=[0,T]$ and assume $\|\omega-\omega_*\|_{L^\infty(I)}\lesssim \epsilon$ and $\|\lambda\|_{L^\infty}+\|\widetilde r\|_{L^\infty H^1(I)}\lesssim \epsilon^{3/2}$.
Then, we have
\begin{align*}
\|\mathbf R\|_{L^\infty(I)}\lesssim \epsilon^{3(N+1)/2}+\epsilon\|\widetilde r\|_{L^6}^2.
\end{align*}

\end{lemma}

\begin{proof}
It suffices to estimate each term in the sum defining $\mathbf R$.

First, by Lemma \ref{lem:estmu}, we have $|\mu|+|\lambda|+\|\widetilde r\|_{\Sigma^{0}}\lesssim \epsilon^{3/2}$.
Thus,
\begin{align}\label{boundfatR1}
|S^0_{0,N+1}|\lesssim \epsilon^{3(N+1)/2}.
\end{align}
Next, we estimate $\mathbf R_2$.
The   $\<S^2_{1,0}\widetilde r,\widetilde r\>$ term in $\mathbf R_2$ can be bounded by $ \epsilon \|\widetilde r\|_{\Sigma^{-2}}^2\lesssim \epsilon \|\widetilde r\|_{L^6}^2$.
Next, recalling Definition \ref{def:beta}, for $s\geq 0$ we have
\begin{align}\label{betabound}
\beta_n(s)\lesssim
\begin{cases}
s+s^{\max(1,p)} & n=0,\\
1+s^{\max(0,p-1)} & n\geq1.
\end{cases}
\end{align}
Thus, the terms in $\mathbf R_2$ of the form $\int \beta_n(|S^1_{0,0}+\widetilde r|^2)s^1_{1,0}|\widetilde r|^2\,dx$ ($n\geq -1$) can be bounded by
\begin{align}
\left|\int_{\R^3} \beta_n(|S^1_{0,0}+\widetilde r|^2)s^1_{1,0}|\widetilde r|^2\,dx\right|&\lesssim \epsilon\int_{\R^3}\<x\>^{-4} (|\widetilde r|^2+|\widetilde r|^{\max(4,2+2p)})\nonumber\\&\lesssim  \epsilon \(\|\widetilde r\|_{L^6}^2+\|\widetilde r\|_{L^6}^{\max(4,2+2p)}\)\lesssim \epsilon \|\widetilde r\|_{L^6}^2.\label{boundfatR2}
\end{align}
The other term in \eqref{r2} satisifies the same estimate.

\noindent For $\mathbf R_k$ ($k=3,4,5$), by \eqref{r345} and \eqref{betabound}, for $n\geq k-3$, we have
\begin{align}
|\int_{\R^3}\beta_n(|S^1_{0,0}+\widetilde r|^2)\<S^1_{0,0},\sigma_1\widetilde r\>^i|\widetilde r|^{2j}\,dx|&\lesssim\int_{\R^3}\<x\>^{-2}\(|\widetilde r|^k + |\widetilde r|^{\max(k,k+2(p-1))}\)\,dx\nonumber\\&
\lesssim \|\widetilde r\|_{L^6}^k+\|\widetilde r\|_{L^6}^{\max(k,k+2(p-1))}\lesssim \epsilon \|\widetilde r\|_{L^6}^2,\label{boundfatR3}
\end{align}
for $k=3,4$ and
\begin{align}
|\int_{\R^3}\beta_n(|S^1_{0,0}+\widetilde r|^2)\<S^1_{0,0},\sigma_1\widetilde r\>^i|\widetilde r|^{2j}\,dx|&\lesssim\int_{\R^3}\<x\>^{-2}|\widetilde r|^5\,dx\lesssim \epsilon \|\widetilde r\|_{L^6}^2,\label{boundfatR32}
\end{align}
where we  used $\|\widetilde r\|_{L^6}\lesssim \|\widetilde r\|_{H^1}\lesssim \epsilon^{3/2}\leq \epsilon$.

\noindent For $\mathbf R_6$, since $g'''(s)\lesssim s^{p-3}$ for $s\geq 1$, we have
\begin{align}
&\left|\int_0^1\int_0^1(1-t)^2\int_{\R^3}g'''(|v(t,s)|^2)\<v(t,s),\sigma_1 S^1_{0,0}\>_{\C^2}\<v(t,s),\sigma_1 S^1_{1,1}\>_{\C^2}\<v(t,s),\sigma_1 \widetilde r\>_{\C^2}^{2}\,dxdsdt\right| \nonumber\\&
\lesssim \int_0^1\int_0^1\int_{|v(t,s)|\leq 1} |S^1_{1,1}||\widetilde r|^{2}\,dxdsdt+\int_0^1\int_0^1\int_{|v(t,s)|> 1} |v(t,s)|^{2p-3} |S^1_{1,1}||\widetilde r|^{2}\,dxdsdt\nonumber\\&
\lesssim \epsilon^{5/2}\|\widetilde r\|_{L^6}^{2}\le  \epsilon \|\widetilde r\|_{L^6}^{2}.\label{boundfatR4}
\end{align}
Finally, for $\mathbf R_7$, by \eqref{R7} and \eqref{nabla4E}, we have
\begin{align}
|\mathbf R_7|\lesssim \epsilon \|\widetilde r\|_{L^6}^2.\label{boundfatR5}
\end{align}
Combining \eqref{boundfatR1}, \eqref{boundfatR2}, \eqref{boundfatR3}, \eqref{boundfatR4} and \eqref{boundfatR5}, we have the conclusion.
\end{proof}

We also have a similar estimate for $\mathbf R_\omega$ and $\mathbf R_\lambda$ (see \eqref{Rsigmalambdar} for the definition).
\begin{lemma}\label{lem:RsigmaRlambdabound}
Let $T>0$ and set $I=[0,T]$ and assume $\|\omega-\omega_*\|_{L^\infty(I)}\lesssim \epsilon$ and $\|\lambda\|_{L^\infty}+\|\widetilde r\|_{L^\infty H^1(I)}\lesssim \epsilon^{3/2}$.
Then, we have
\begin{align*}
\|\mathbf R_\omega\|_{L^\infty(I)}+\|\mathbf R_\lambda\|_{L^\infty(I)}\lesssim \epsilon^{3N/2}+ \|\widetilde r\|_{L^6}^2.
\end{align*}
\end{lemma}

\begin{proof}
The proof is similar to the proof of Lemma \ref{lem:boundfatR}.
Therefore, we omit it.
\end{proof}

We now estimate the $W^{s,6/5}$ norm for $s=0,1$ of $\mathbf R_{\widetilde r}$, see \eqref{Rsigmalambdar}.
\begin{lemma}\label{lem:errorstzbound}
Let $s=0,1$.
Assume $|\omega-\omega_*|\lesssim \epsilon$ and $|\lambda|+\|\widetilde r\|_{H^1(I)}\lesssim \epsilon^{3/2}$.
Then, we have
\begin{align*}
\|\mathbf R_{\widetilde r}(t)\|_{ W^{s,6/5}}\lesssim_\sigma \epsilon^{3N/2}+\epsilon \|\widetilde r(t)\|_{W^{s,6}},\quad \forall t\in I
\end{align*}
\end{lemma}

The proof of Lemma \ref{lem:errorstzbound} is elementary.
However, it will be long simply because there are many terms in $\mathbf R_{\widetilde r}$.
Therefore, we give the proof in Sect.\ \ref{sec:errorstzbound} in the appendix of this paper.

%
%

\begin{lemma}\label{lem:estddtQ}
Let $T>0$ and set $I=[0,T]$ and assume $\|\omega-\omega_*\|_{L^\infty(I)}\lesssim \epsilon$ and $\|\lambda\|_{L^\infty}+\|\widetilde r\|_{L^\infty H^1(I)}\lesssim \epsilon^{3/2}$.
We have
\begin{align*}
\left|\frac{d}{dt}Q(\widetilde r(t))\right|\lesssim \|\widetilde r\|_{L^6}^2+ \epsilon^{ \frac{3}{2}  (N+1)}.
\end{align*}
\end{lemma}

\begin{proof}
We directly compute the derivative.
\begin{align*}
\frac{d}{dt}Q(\widetilde r )&=\<\partial_t \widetilde r, \sigma_1 \widetilde r\>=-\Omega(\im \sigma_3 \widetilde r,\widetilde r)\\&
=-\Omega\((H_{\omega_*}+\partial_\rho E_{N})\widetilde r+ g(|\widetilde r|^2)\widetilde r +\mathbf R_{\widetilde r},\widetilde r\)\\&
=-\Omega(g'(\phi_{\omega_*}^2)\phi_{\omega_*}^2 \sigma_1 \widetilde r,\widetilde r)-\Omega(\mathbf R_{\widetilde r},\widetilde r).
\end{align*}
Therefore, by Lemma \ref{lem:errorstzbound}, we have the following, which yields the proof:
\begin{align*}
\left| \frac{d}{dt}Q(\widetilde r)\right|\lesssim \|\widetilde r\|_{L^6}^2+ \epsilon^{\frac32 N}\|\widetilde r\|_{L^6}.
\end{align*}
\end{proof}

\subsection{Long time almost conservation of $E_Q$ and the shadowing for several periods}
\label{sec:proofs fs}
We now prove Theorem \ref{thm:dynamics}.

\begin{proof}[Proof of Theorem \ref{thm:dynamics}]
Take $C_1>0$ sufficiently large (chosen later) and take $\epsilon_0>0$ so that we have $C_1 \epsilon\ll 1$.
We will fix $\epsilon\in (0,\epsilon_0)$.

\noindent We first show that for $0<\widetilde T\leq  \epsilon^{-\frac32(N-1) +1}=:T$, if we have
\begin{align}\label{eq:almostcons2}
\sup_{t\in [0,\widetilde T]}\left|  E_{Q}(\omega(t),\lambda(t))-  E_{Q}(\omega_0,\lambda_0)\right|\leq C_1 \epsilon^4,
\end{align}
then we have for  $C=C(C_0)$, where $C_0$ the constant in the claim of Theorem \ref{thm:dynamics},
\begin{align}\label{rbound2}
\|\widetilde r\|_{\stz^1(0,\widetilde T)}\leq C \epsilon^{3/2}
\end{align}
 and
\begin{align}\label{eq:almostcons3}
\sup_{t\in [0,\widetilde T]}\left|  E_{Q}(\omega(t),\lambda(t))-  E_{Q}(\omega_0,\lambda_0)\right|\leq \frac12 C_1 \epsilon^{4}.
\end{align}
Since $\stz^1\hookrightarrow L^\infty H^1$, by continuity argument this will give us the conclusion.

\noindent First, by \eqref{eq:almostcons2} and the assumption \eqref{finiteenergybound}, we see that $(\omega(t),\lambda(t))$ satisfies \eqref{trap} for all $t\in [0,\widetilde T]$.
Therefore, by Lemma \ref{lem:findimEcrt}, we have \eqref{varsigmalambdabound} where the implicit constants do not depend on $C_1$.
Therefore, by Lemmas \ref{lem:stz}, \ref{lem:nonlinearity} and \ref{lem:errorstzbound}, we have
\begin{align*}
\|\widetilde r\|_{\stz^1(0,t)}\lesssim \|\widetilde r_0\|_{H^1}+\|\widetilde r\|_{\stz^1(0,t)}^{4}+\|\widetilde r\|_{\stz^1(0,t)}^{\max(3,2p+1)}+t^{1/2}\epsilon^{3N/2}+\epsilon\|\widetilde r\|_{\stz^1(0,t)},
\end{align*}
for any $t\in [0,\widetilde T]$.
Therefore, from an easy continuity argument, we have
\begin{align*}
\|\widetilde r\|_{\stz^1(0,\widetilde T)}\lesssim \epsilon^{3/2},
\end{align*}
where the implicit constant is independent of $C_1$
 (but depends on $C_0$).

Next, we prove \eqref{eq:almostcons3}.
We claim that
\begin{align}\label{eq:almostcons4}
\left|E_{f,N}(Q,\omega,\lambda,Q(\widetilde r))-E_Q(\omega,\lambda)\right|=\left|\int_0^1 \partial_\rho E_{f,N}(Q,\omega,\lambda,tQ(\widetilde r))\,dt\  Q(\widetilde r)\right|\lesssim \epsilon^4.
\end{align}
The equality follows from the definition $E_Q(\omega,\lambda) = E_{f,N}(Q,\omega,\lambda,0) $, is a  consequence of
 the assumption $|\omega -\omega _*|\lesssim \epsilon$ and $|\lambda|+\|\widetilde r\|_{H^1}\lesssim \epsilon^{3/2}$ and   $|\partial_\rho E_{f,N} |\lesssim \epsilon$ (the latter follows from the previous inequalities and from $\partial_\rho E_{f,N}
 =\omega -\omega _*+O(\lambda ^2)$,  see \eqref{eq:fenergy}).

\noindent In view of \eqref{eq:almostcons4}, it is enough to prove \eqref{eq:almostcons3}
with $E_Q(\omega,\lambda)$ replaced by $E_{f,N}(Q,\omega,\lambda,Q(\widetilde r))$.

 \noindent From
\begin{align*}
\frac{d}{dt}E_{f,N}(Q,\omega(t),\lambda(t),Q(\widetilde r(t)))&=\partial_{\omega}E_{f,N}\dot \omega+\partial_\lambda E_{f,N}\dot \lambda +\partial_\rho E_{f,N}\frac{d}{dt}Q(\widetilde r)\\& =\partial_{\omega}E_{f,N}\mathbf R_{\omega}+\partial_\lambda E_{f,N}\mathbf R_\lambda + \partial_\rho E_{f,N} \frac{d}{dt}Q(\widetilde r),
\end{align*}
by Lemmas \ref{lem:RsigmaRlambdabound}, \ref{lem:errorstzbound} and \ref{lem:estddtQ}  we have
\begin{align*}
\left|\frac{d}{dt}E_{f,N}(Q,\omega(t),\lambda(t),Q(\widetilde r(t)))\right|\lesssim \epsilon^{3/2}\(\epsilon^{\frac32 N}+\|\widetilde r(t)\|_{L^{6}}^2\)+\epsilon\(\|\widetilde r\|_{L^6}^2+ \epsilon^{ \frac{3N+1}{2}}\),
\end{align*}
where we have used $|\partial_\omega E_{f,N}|\lesssim \epsilon^2$, $|\partial_\lambda E_{f,N}|\lesssim \epsilon^{3/2}$.
Thus, we have
\begin{align*}
&\left|E_{f,N}(Q,\omega(\widetilde T),\lambda(\widetilde T),Q(\widetilde r(\widetilde T)))-E_{f,N}(Q,\omega_0,\lambda_0,Q(\widetilde r_0)\right|\lesssim \int_0^{\widetilde T}\left|\frac{d}{dt}E_{f,N}(Q,\omega(t),\lambda(t),Q(\widetilde r(t)))\right|\,dt\\&
\lesssim \epsilon\|\widetilde r\|_{\stz^1(0,\widetilde T)}+\widetilde T \epsilon^{\frac32(N+1)}\lesssim \epsilon^4,
\end{align*}
where the implicit constant does not depend on $C_1$.
Therefore, we have the conclusion.
\end{proof}

Corollary \ref{corollary:longorb} is an easy consequence of Theorem \ref{thm:dynamics}.
\begin{proof}[Proof of Corollary \ref{corollary:longorb}]
By Theorem \ref{thm:dynamics}, we have
\begin{align}\label{pf:longorb0}
\sup _{t\in [0,T]}|E_Q(\omega (t),\lambda (t))-E_Q(\omega_0,\lambda_0)|\lesssim \epsilon^4  .
\end{align}
By Remark  \ref{rem:trapener},    for $c\in (0,1)$ the constant in
\eqref{finiteenergybound}  we have
\begin{equation}\label{pf:longorb1}
    \text{$E_Q(\omega(t),\lambda(t))<\frac{1+c}{2}(V_Q(\omega_-)-V_Q(\omega_+))\sim \epsilon^{3 }$ for all  $t\in [0,T]$.}
\end{equation}
By Remark  \ref{rem:trap} we know that
\begin{align}\nonumber
\sup _{t\in [0,T]}\( |\omega (t)-\omega_*| ^3 +  |\lambda (t)|^2\) \lesssim \epsilon^{3 }   .
\end{align}
This implies that, setting  $\omega=\omega(t)$, $\lambda=\lambda(t)$,
it suffices to prove that there exists $\alpha\in \R$ with $|\alpha|\lesssim \epsilon$ s.t.\
\begin{align}\label{aadjust}
E_Q(\omega_+ + (1+\alpha)(\omega-\omega_+),(1+\alpha)\lambda)=E_Q(\omega_0,\lambda_0).
\end{align}
By the Taylor expansion
\begin{align*}
E_Q(\omega,\lambda)=\(\sqrt{\frac{q''(\omega_*)}{2}}\epsilon+\frac16 q''(\omega_*)(\omega-\omega_+)\)(\omega-\omega_+)^2+\frac12 A(\omega_*)\lambda^2+O(\epsilon^4)
\end{align*}
we have
\begin{align}
&E_Q(\omega_+ + (1+\alpha)(\omega-\omega_+),(1+\alpha)\lambda)=E_Q(\omega,\lambda)\nonumber\\&\quad+\alpha\(\(\sqrt{2q''(\omega_*)}\epsilon+\frac12 q''(\omega_*)(\omega-\omega_+)\)(\omega-\omega_+)^2+\frac12 A(\omega_*)\lambda^2\)+O(\alpha \epsilon^4+\alpha^2 \epsilon^3).\label{EQalpha}
\end{align}
Notice now that \eqref{pf:longorb0}--\eqref{pf:longorb1} and the hypothesis
$ E_Q(\omega_0,\lambda_0)\sim \epsilon ^{3 }$, imply  $ E_Q(\omega ,\lambda )\sim \epsilon ^{3 }$ and that we have $\omega \in (\omega_+-c'(\omega_+-\omega_-),\omega_{++})$ for some $c'\in (0,1)$,
uniformly bounded away from 0 for all $t\in [0,T]$,
 where $\omega_{++}>\omega_+$ is the solution of $V_Q(\omega_{++})=V_Q(\omega_-)$
 with $\omega_{++}>\omega_{+ }$.
Thus by \eqref{eq:findimEcrt} we have
\begin{align*}
\omega-\omega_+>-c'(\omega_+-\omega_-)=
-2c'\sqrt{\frac{2}{q''(\omega_*)}}\epsilon+O(\epsilon^2).
\end{align*}
Therefore, we have
\begin{align}\label{nondeglow}
\(\sqrt{2q''(\omega_*)}\epsilon+\frac12 q''(\omega_*)(\omega-\omega_+)\)>(1-c')\sqrt{2q''(\omega_*)}\epsilon+O(\epsilon^2).
\end{align}
By \eqref{EQalpha} and \eqref{nondeglow}, it is clear that we can take $|\alpha| \lesssim \epsilon$ such that the 2nd line in  \eqref{EQalpha} is 0
and
 \eqref{aadjust} holds.
Therefore, we have the conclusion.
\end{proof}

We finally prove Theorem \ref{thm:shadow}.

\begin{proof}[Proof of Theorem \ref{thm:shadow}]
To estimate the difference, we use the scaling
\begin{align}\label{rescale:inf}
\omega = \omega_+ + \epsilon \zeta,\quad \lambda = \epsilon^{3/2} \kappa,\quad t=\epsilon^{-1/2}\tau.
\end{align}
For fixed $n\in \N$, $nT_f\sim_n \epsilon^{-1/2}$ so it suffices to show that for fixed $\tau_0>0$,
\begin{align}\label{rescale:est1}
\sup_{\tau\in [0,\tau_0]}\(|\zeta(\tau)-\zeta_f(\tau)|+|\lambda(\tau)-\lambda_f(\tau)|\)\lesssim_{\tau_0}\epsilon.
\end{align}

By the estimate \eqref{rbound}, we have
\begin{equation}\label{rescale:eq1}
\begin{aligned}
\dot \zeta &= \kappa+O(\epsilon),\\
\dot \kappa &= -A(\omega_*)^{-1}\(\sqrt{2q''(\omega_*)}\zeta+\frac12 q''(\omega_*)\zeta^2\)+O(\epsilon),
\end{aligned}
\end{equation}
where $\dot \zeta=\frac{d}{d\tau} \zeta$ and $\dot \kappa=\frac{d}{d\tau} \kappa$.
Notice that this system looks completely the same as \eqref{rescale:eq1} although the $O(\epsilon)$ are of course different.
Subtracting \eqref{rescale:feq} from \eqref{rescale:eq1}, we have
\begin{equation}\label{rescale:eq1}
\begin{aligned}
\frac{d}{d\tau}\( \zeta-\zeta_f\) &= \kappa-\kappa_f+O(\epsilon),\\
\frac{d}{d\tau}\( \kappa-\kappa_f\) &= -A(\omega_*)^{-1}\(\sqrt{2q''(\omega_*)}\(\zeta-\zeta_f\)+\frac12 q''(\omega_*)\(\zeta+\zeta_f\)\(\zeta-\zeta_f\)\)+O(\epsilon),
\end{aligned}
\end{equation}
Therefore, setting $f(\tau):=|\zeta(\tau)-\zeta_f(\tau)|+|\kappa(\tau)-\kappa_f(\tau)|$, we have
\begin{align}\label{gronwall:1}
f(t)\lesssim \epsilon t + \int_0^t f(s)\,ds,
\end{align}
where we have used the bound $|\zeta(\tau)|+|\zeta_f(\tau)|\lesssim 1$ for $\tau\in [0,\tau_0]$.
Therefore, by Gronwall inequality, we have
\begin{align}
\sup_{\tau\in[0,\tau_0]}|f(\tau)|\lesssim_{\tau_0} \epsilon.
\end{align}
Since this is \eqref{rescale:est1}, we have the conclusion.
\end{proof}

\appendix
\section{Appendix}
\subsection{Proof of  Proposition \ref{prop:1}}\label{subsec:prop1}

In this section, we give the proof of Proposition \ref{prop:1}.
We note that  Proposition \ref{prop:1} was proved by Comech--Pelinovsky \cite{CP03CPAM}.

First, we set
\begin{align}\label{37}
\Pi_R:=\frac  1 2 \(1+\sigma_1\),\quad \Pi_I:=\frac 1 2 \(1 -\sigma_1\).
\end{align}
The operators $\Pi_R$ and $\Pi_I$ satisfies $\Pi_R \widetilde u=\widetilde{\Re u}$ and $\Pi_I \widetilde u=\widetilde{\Im u}$.
So, it correspond taking real and imaginary parts respectively.
Further, notice that $\Pi_R$ and $\Pi_I$ are orthogonal projections with respect to the innerproduct $\<\cdot,\cdot\>$ and satisfy $\Pi_R+\Pi_I=1$ and commute with $\pi_0$.
Moreover, we have $[H_\omega,\Pi_X]=0$ for $X=R,I$.
We define
\begin{align*}
L_{+,\omega}:=H_\omega \Pi_R,\quad L_{-,\omega}:=H_\omega \Pi_I.
\end{align*}
Then, we have
\begin{align}\label{38.3}
L_{+,\omega} =-\Delta + \omega +g(\phi_\omega^2)+2g'(\phi_\omega^2),\quad
L_{-,\omega} =-\Delta + \omega +g(\phi_\omega^2).
\end{align}
and
\begin{align}\label{38}
H_\omega=L_{+,\omega} \Pi_R + L_{-,\omega}\Pi_I.\end{align}

%
%
%
%
%

Notice that $\im \sigma_3 \widetilde{\phi}_\omega$ is purely imaginary valued and $\partial_\omega \widetilde{\phi}_\omega$ is real valued.
This structure is important when we study the generalized kernel. The following lemma is elementary.
\begin{lemma}\label{lem:6}
We have
\begin{itemize}
\item
$\Pi_R \im \sigma_3 \widetilde{\phi}_\omega=0$ and $\Pi_I \partial_\omega \widetilde{\phi}_\omega=0$.
\item
$\mathcal H_\omega \Pi_R=\Pi_I \mathcal H_\omega$ and $\mathcal H_\omega \Pi_I=\Pi_R\mathcal H_\omega$.
\end{itemize}
\end{lemma}

Following \cite{CP03CPAM}, we introduce the Riesz projection $P_d(\omega)$ for $\omega\in \mathcal O$.

\begin{definition}\label{def:projPd}
For     a fixed sufficiently small  $0<r \ll 1$ we define
in $\Lrad(\R ^3, \C ^2)$ the operator
\begin{align}\label{eq:projPd}
P_d(\omega):=-\frac{1}{2\pi \im}\oint_\gamma ( \mathcal H_\omega -z)^{-1} \,dz, \text{ where $\gamma=\{ r  e^{\im \theta}\ |\ \theta\in [0,2\pi)\}$.}
\end{align}
\end{definition}

\begin{remark}
By a simple change of coordinate we can express $P_d(\omega)$ as
\begin{align}\label{eq:projPd1}
P_d(\omega) =-\frac{1}{2\pi \im}\oint_\gamma (-\im \mathcal H_\omega -z)^{-1} \,dz, \text{ with $\gamma $ as above.}
\end{align}

\end{remark}

The properties of $P_d$ (which correspond to 4.\ of Proposition \ref{prop:1}) are summarized as follows.

\begin{lemma}\label{lem:7}
For $s\in \R$, there exists $\delta_s>0$ s.t.\ $P_d \in \cap_{s}C^\infty(D_{\R}(\omega_*,\delta_s), \mathcal L(\Sigma^{s}(\R^3,\C^2)))$.
Moreover, we have
\begin{itemize}
\item
$P_d(\omega)^2=P_d(\omega)$.
\item
$\Omega(P_d(\omega)U,V)=\Omega(U,P_d(\omega)V)$.
\item
$[P_d(\omega),\mathcal H_\omega]=0$ and $[P_d(\omega),\Pi_A]=[P_d(\omega),\pi _0]=0$ for $A=R,I$.
\end{itemize}
\end{lemma}

\begin{proof}
We first expand $\mathcal H_{\omega}=\mathcal H_{\omega_*} +  h_\omega$.
From \eqref{39} and Lemma \ref{lem:4.3}, for arbitrary $s\in\R$, we have $\|h_\omega\|_{\Sigma^s\to \Sigma^s}\to 0$ as $\omega\to \omega_*$.
We can rewrite \eqref{eq:projPd} as
\begin{align*}
P_d(\omega)&=-\frac{1}{2\pi \im}\oint_{\gamma}(\mathcal H_{\omega_*}-z)^{-1}\(1+(\mathcal H_{\omega_*}-z)^{-1} h_\omega\)^{-1}\,dz\\&
=-\frac{1}{2\pi \im}\oint_{\gamma}(\mathcal H_{\omega_*}-z)^{-1}\(1+\widetilde h_\omega(z)\)\,dz\\&
=P_d(\omega_*)-\frac{1}{2\pi \im}\oint_{\gamma}(\mathcal H_{\omega_*}-z)^{-1}\widetilde h_\omega(z) \,dz,
\end{align*}
where $\widetilde h_\omega(z)$ are determined by the relation $$1+\widetilde h_\omega(z)=\(1+(\mathcal H_{\omega_*}-z)^{-1}h_\omega\)^{-1}.$$
Therefore, we see $P_d \in \cap_{s}C^\infty(D_{\R}(\omega_*,\delta_s), \mathcal L(\Sigma^{s}(\R^3,\C^2)))$ for some $\delta_s>0$.

$P_d(\omega)^2=P_d(\omega)$ is standard, see for example section 6 of \cite{HSBook}.
The fact that $P_d$ is symmetric w.r.t.\ $\Omega$ follows from simple change of variables.
 \begin{align*}&
\Omega(P_d(\omega)U,V) =\Omega\(-\frac{1}{2\pi \im}\oint_\gamma (\mathcal H_\omega-z)^{-1}\,dz U,V\) =\Omega\(-\frac{r}{2\pi }\int_0^{2\pi} (\mathcal H_\omega-re^{\im \theta})^{-1} e^{\im \theta}\,d\theta U,V\)\\&=-\frac{r}{2\pi }\int_0^{2\pi}\Omega\( (\mathcal H_\omega-re^{\im \theta})^{-1} e^{\im \theta}\, U,V\)d\theta
= \frac{r}{2\pi }\int_0^{2\pi}\Omega\( U,( \mathcal H_\omega +re^{-\im \theta})^{-1} e^{-\im \theta}\, V\)d\theta\\&
=\Omega\( U,-\frac{r}{2\pi }\int_0^{2\pi}(\mathcal H_\omega-re^{ \im \theta})^{-1} e^{ \im \theta}\,d\theta V\) =\Omega(U,P_d(\omega)V).
\end{align*} Similarly, we have the following, which implies $[P_d(\omega),\Pi_A]=0$ for $A=R,I$:
\begin{align*}
\sigma_1 P_d(\omega) \sigma_1 &=-\frac{1}{2\pi \im}\sigma_1\oint_\gamma (\mathcal H_\omega -z)^{-1} \,dz \sigma_1  = -\frac{1}{2\pi \im}\oint_\gamma (\sigma_1\mathcal H_\omega \sigma_1 -z)^{-1} \,dz
\\&=  \frac{1}{2\pi \im}\oint_\gamma ( \mathcal H_\omega+ z)^{-1} \,dz= -\frac{1}{2\pi \im}\oint_\gamma (\mathcal H_\omega  -z)^{-1} \,dz=P_d(\omega).
\end{align*}
Finally, $[P_d(\omega), \mathcal H_\omega]=0$ follows from the expression \eqref{eq:projPd} and $[P_d(\omega),\pi_0]$ follows from the expression \eqref{eq:projPd1} and $[-\im \mathcal H_\omega, \pi_0]=0$.
%
%
\end{proof}

\begin{remark}
In Lemma \ref{lem:7}, we only proved $P_d \in \cap_{s}C^\infty(D_{\R}(\omega_*,\delta_s), \mathcal L(\Sigma^{s}(\R^3,\C^2)))$ and not $P_d \in \cap_{s\geq0}C^\infty(D_{\R}(\omega_*,\delta_s), \mathcal L(\Sigma^{-s}(\R^3,\C^2),\Sigma^{s}(\R^3,\C^2)))$.
We will show the latter after we obtain another expression of $P_d$ based on $\Psi_j(\omega)$ which we will define below.
\end{remark}

We next study the structure of the generalized kernel of $-\im \mathcal H_{\omega_*}$.

\begin{lemma}\label{lem:varsigmastar}
There exists $\Psi_j(\omega_*) \in \mathcal S(\R^3,\widetilde \C^2)$ ($j=3,4$) s.t.\ $$-\im \mathcal H_{\omega_*} \Psi_j(\omega_*)=\Psi_{j-1}(\omega_*)\text{ and }\Pi_R\Psi_3(\omega_*)=\Pi_I \Psi_4(\omega_*)=0.$$
Here, $\Psi_2(\omega_*)$ is defined in \eqref{Psi12}.
Moreover, we have
\begin{align}\label{eq:Astarpos}
A(\omega_*)=\Omega(\Psi_2(\omega_*),\Psi_3(\omega_*))=-\Omega(\Psi_2(\omega_*),\Psi_3(\omega_*))>0.
\end{align}
\end{lemma}

\begin{proof}
First, it suffices to restrict everything in $\mathrm{Ran}P_d(\omega_*)$.
Set $\mathcal H_{\omega_*,d}:=\left.\mathcal H_{\omega_*} \right|_{\mathrm{Ran}P_d(\omega_*)}$.
Recall the relation $$\mathrm{Ran}\mathcal H_{\omega_*,d}=\overline{\mathrm{Ran}\mathcal H_{\omega_*,d}}= (\mathcal{N}(\mathcal H_{\omega_*}^*))^\perp=\(\mathrm{span}\{\sigma_1\widetilde{\phi}_{\omega_*}\}\)^\perp,$$ with orthogonality in terms   \eqref{eq:bilform} and we have used $\mathcal H_{\omega_*}^*=\sigma_1\sigma_3\mathcal H_{\omega_*} \sigma_3\sigma_1$ and Assumption \ref{a:1} in the last equality.
Now,
\begin{align*}
\<\sigma_1\widetilde \phi_{\omega_*},\Psi_2(\omega_*)\>=q'(\omega_*)=0,
\end{align*}
so we have $\Psi_2(\omega_*)\in \mathrm{Ran}\mathcal H_{\omega_*}$.
Therefore, there exists nontrivial $\hat \Psi_3(\omega_*)$ s.t.\ $\im \mathcal H_{\omega_*,d}\hat \Psi_3(\omega_*)=\Psi_3(\omega_*)$.
By standard elliptic regularity argument, we have $\hat\Psi_3\in \mathcal S(\R^3,\C^2)$.
Moreover, setting $\Psi_3(\omega_*)=\Pi_I (1-\pi_0) P_d(\omega_*)\hat \Psi$, we see that $\Psi_3(\omega_*)\in \mathcal S(\R^3,\widetilde \C)$ satisfies $-\im \mathcal H_{\omega_*}\Psi_3(\omega_*)=\Psi_2(\omega_*)$ and $\Pi_R \Psi_3(\omega_*)=\pi_0\Psi_3(\omega_*)=0$.

  In the case $q'({\omega  _*})=0$, we try also to find $ \hat{\Psi}_4({\omega  _*})$ s.t.\ $-\im\mathcal H_{{\omega  _*},d}  \hat{\Psi}_4({\omega  _*})=\Psi_3({\omega  _*})$.
As before, it suffices to show $\<\Psi_3({\omega  _*}),\sigma_1\widetilde{\phi}_{{\omega  _*}}\>=0$.
However, this follows from $$\<\Psi_3({\omega  _*}),\sigma_1\widetilde{\phi}_{{\omega  _*}}\>=\<\Pi_I \Psi_3({\omega  _*}),\sigma_1\widetilde{\phi}_{\omega  _*} \>=\<\Psi_3({\omega  _*}),\Pi_I \sigma_1\widetilde{\phi}_{\omega  _*}\>=0.$$
Again, by standard elliptic regularity argument, we have $\hat \Psi_4(\omega_*)\in \mathcal S(\R^3,\C^2)$.
Setting $\Psi_4(\omega_*) = P_d(\omega_*)(1-\pi_0)\Pi_R\hat\Psi_4(\omega_*) \in \mathcal S(\R^3,\widetilde \C)$, we have $-\im \mathcal H_{\omega_*} \Psi_4(\omega_*)=\Psi_3(\omega_*)$, $\pi_0\Psi_4(\omega_*)=\Pi_I \Psi_4(\omega_*)=0$.

Finally, we show \eqref{eq:Astarpos}.
\begin{align}
-\Omega(\Psi_1(\omega_*),{\Psi}_4({\omega  _*})) &=-\Omega(-\im\mathcal H_{{\omega  _*}} \Psi_2(\omega  _*), {\Psi}_4({\omega  _*}))= \Omega( \Psi_2(\omega_*), -\im\mathcal H_{{\omega  _*}}{\Psi}_4({\omega  _*}))\nonumber\\&=\Omega(\Psi_2(\omega  _*), \Psi_3({\omega  _*} )) =-\Omega(\im \mathcal H_{\omega  _*} \Psi_3({\omega  _*}),\Psi_3({\omega  _*}))
=\<H_{\omega  _*} \Psi_3({\omega  _*}),\sigma _1\Psi_3({\omega  _*})\>\nonumber\\&=\<L_{-,{\omega  _*}} \Psi_3({\omega  _*}),\sigma _1\Psi_3({\omega  _*})\>,\label{43}
\end{align}
where $L_{-,\omega_*}=H_{\omega_*}\Pi_I$.
We can express $L_{-,\omega_*}=-\Delta+\omega_*+g(\phi_{\omega_*}^2)$.
Obviously $\Psi_1(\omega_*)=\im \sigma_3 \widetilde \phi_{\omega_*}$ satisfies $L_{-,\omega_*}\Psi_1(\omega_*)$ and since $\widetilde \phi_{\omega_*}$ is positive valued, $L_{-,\omega_*}\geq 0$ and \eqref{43} will be $0$ if and only if $ u=c \Psi_1(\omega_*)$.
If the last quantity  in \eqref{43} were $0$ it would mean that $\Psi_3({\omega  _*} )=c \Psi_1(\omega_*)$.  This would imply $ 0= -\im \mathcal H_{\omega _*} \Psi_3(\omega _*)=\Psi_2(\omega_*)$, which is absurd.

We note that since $0\neq A(\omega_*)=\<\sigma_1 \widetilde \phi_{\omega_*},\Psi_4\>$, there exists no $\Psi_5$ s.t.\ $-\im \mathcal H_{\omega_*}\Psi_5=\Psi_4(\omega_*)$.
\end{proof}

\begin{proof}[Proof of Proposition \ref{prop:1}]
We set $\tilde \Psi_4(\omega):=P_d(\omega)\Psi_4(\omega_*)$.
Notice that by Lemma \ref{lem:7}, $\Pi_I\tilde\Psi_4(\omega)=0$ and $\tilde \Psi_4\in C^\infty(D_{\R}(\omega_*)(0,\delta_s),\Sigma^s)$.
Similarly, by Lemma \ref{lem:7} and by Lemma \ref{lem:varsigmastar} we have
 $\pi _0 \tilde \Psi_4(\omega) =P_d(\omega)\pi _0 \Psi_4(\omega_*)=0$.
Since $-\im \mathcal H_\omega$ is invariant in $\mathrm{Ran}P_d(\omega)$, we have $$(-\im \mathcal H_\omega)^2 \tilde \Psi_4(\omega)=b(\omega)\Psi_2(\omega)+a(\omega)\tilde \Psi_4(\omega),$$ with $a,b\in C^\infty(D_{\R}(0,\delta_0),\R)$, $b(\omega_*)=1$ and $a(\omega_*)=0$.
Thus, we set $\Psi_4(\omega)=b(\omega)^{-1}\tilde \Psi_4(\omega)\in C^\infty(D_{\R}(\omega_*)(0,\delta_s),\Sigma^s)$ making $\delta_s$ smaller to avoid the zero of $b$ if necessary.
Therefore, we have $(-\im \mathcal H_\omega)^2 \Psi_4(\omega)=\Psi_2(\omega)+a(\omega)\Psi_4(\omega)$.
We further set $$\Psi_3(\omega):=-\im \mathcal H_\omega \Psi_4(\omega)\in C^\infty(D_{\R}(\omega_*)(0,\delta_s),\Sigma^s),$$ with modifying $\delta_s$ to $\delta_{s+2}$.
Notice that we have $\Pi_R \Psi_3(\omega)=-\im \mathcal H_\omega \Pi_I \Psi_4(\omega)=0$  and $\pi _0 \Psi_3(\omega)=-\im \mathcal H_\omega \pi _0 \Psi_4(\omega)=0$.

By applying $\Omega(\Psi_1(\omega),\cdot)$ to the equality $-\im \mathcal H_\omega \Psi_3 =\Psi_2(\omega)+a(\omega)\Psi_4(\omega)$, we obtain
\begin{align*}
0=\Omega(\Psi_1(\omega),\Psi_2(\omega))+ a(\omega)\Omega(\Psi_1(\omega),\Psi_4(\omega)).
\end{align*}
By $\Omega(\Psi_1(\omega),\Psi_2(\omega))=-q'(\omega)$ and $\Omega(\Psi_1(\omega),\Psi_4(\omega))=-\Omega(\Psi_2(\omega),\Psi_3(\omega))=-A(\omega)$, we obtain 5.\ of

Finally, we can express $P_d(\omega)$ by
\begin{align*}
P_d(\omega)=A(\omega)^{-1}\( -\Omega(\cdot,\Psi_4(\omega))\Psi_1(\omega) + \Omega(\cdot,\Psi_3(\omega))\Psi_2(\omega)+\Omega(\Psi_2(\omega), \cdot)\Psi_3(\omega)-\Omega(\Psi_1(\omega),\cdot)\Psi_4(\omega)\)
\end{align*}
By this expression, it is clear that we have
$P_d \in \cap_{s}C^\infty(D_{\R}(\omega_*,\delta_s), \mathcal L(\Sigma^{-s}(\R^3,\C^2),\Sigma^{s}(\R^3,\C^2)))$ for some $\delta_s>0$.
\end{proof}

\subsection{Proof of Lemma \ref{lem:Rexpand}}
\label{app:Rexpand}
In this section, we prove Lemma \ref{lem:Rexpand}.
We expand each $\mathfrak F^* \mathbf R_k$.

\begin{lemma}
Let $l\geq 1$, $\delta>0$ and let $  \mathfrak{F}\in C^\infty(D_{\Hrad^1}(\mathcal T_{\omega_*},\delta),\Hrad^1(\R^3,\widetilde \C))$ satisfy \eqref{Ferror}--\eqref{eq:sigma(i)}.
Then, we have
\begin{align}
\mathfrak F^*\mathbf R_2&=\mathbf R_2 +\mathbf R_3+S^0_{\sigma(l)+1,l+1}+S^0_{0,l+2}.\label{R2expand}
\end{align}
\end{lemma}

\begin{proof}
We first consider the term $\<S^2_{1,0}\widetilde r,\sigma_1\widetilde r\>$.
By $\mathfrak F^* S^2_{1,0}=S^2_{1,0}$.
Moreover, from \eqref{Ferror}, expanding $e^{S^0_{0,l}\im \sigma_3}$, we have $\mathfrak F^*\widetilde r=(1+S^0_{0,l}\im \sigma_3 +S^0_{0,2l})\widetilde r+S^1_{\sigma(l),l}$.
Therefore, we have
\begin{align}\label{R21}
\mathfrak F^* \<S^2_{1,0}\widetilde r,\sigma_1 \widetilde r\>=\<S^2_{1,0}\widetilde r,\sigma_1\widetilde r\>+S_{0,l+2}^0+S_{\sigma(l)+1,l+1}^0.
\end{align}
Next, we consider the term $\int_{\R^3}\beta_n(|S^1_{0,0}+\widetilde r|^2)s^1_{1,0}|\widetilde r|^2\,dx$ in $\mathbf R_2$.
We have $|e^{S^0_{0,l}\im \sigma_3}\widetilde r|^2=|\widetilde r|^2$ and
\begin{align*}
|S^1_{0,0}+e^{S^0_{0,l}\im \sigma_3}(\widetilde r+S^1_{0,l})|^2=|S^1_{0,0}+\widetilde r+S^1_{0,l}|^2=|S^1_{0,0}+\widetilde r|^2,
\end{align*}
where we have used $e^{S^0_{0,l}\im \sigma_3}S^1_{0,0}=S^1_{0,0}$.
Thus, we have
\begin{align}
&\mathfrak F^*\int_{\R^3}\beta_n(|S^1_{0,0}+\widetilde r|^2)s^1_{1,0}|\widetilde r|^2\,dx=\int_{\R^3}\beta_n(|S^1_{0,0}+\widetilde r|^2)s^1_{1,0}|\widetilde r+S^1_{\sigma(l),l}|^2\,dx\label{r2expand11}
\\&=\mathbf R_2 +\int_{\R^3}\beta_n(|S^1_{0,0}+\widetilde r|^2)\<S^1_{\sigma(l)+1,l},\sigma_1\widetilde r\>_{\C^2}\,dx+\int_{\R^3}\beta_n(|S^1_{0,0}+\widetilde r|^2)s^1_{1+2\sigma(l),2l}\,dx.\nonumber
\end{align}
By \eqref{expandbeta}, we have
\begin{align}
&\int_{\R^3}\beta_n(|S^1_{0,0}+\widetilde r|^2)\<S^1_{\sigma(l)+1,l},\sigma_1\widetilde r\>_{\C^2}\,dx
=\int_{\R^3}\beta_n(|S^1_{0,0}|^2)\<S^1_{\sigma(l)+1,l},\sigma_1\widetilde r\>_{\C^2}\,dx\nonumber\\&+
\int_{\R^3}\beta_{n+1}(|S^1_{0,0}+\widetilde r|^2)\<S^1_{0,0},\sigma_1\widetilde r\>_{\C^2}\<S^1_{\sigma(l)+1,l},\sigma_1\widetilde r\>\,dx\nonumber\\&
+\int_{\R^3}\beta_{n+1}(|S^1_{0,0}+\widetilde r|^2)|\widetilde r|^2\<S^1_{\sigma(l)+1,l},\sigma_1\widetilde r\>_{\C^2}\,dx\nonumber\\&
=S^0_{\sigma(l)+1,l+1}+\mathbf R_2+\mathbf R_3. \label{r2epandpr1}
\end{align}
We remark that by using \eqref{expandbeta}, we are using the convention   in Definition \eqref{def:beta}.
Notice that here, if $n=-1$ (which means that $\beta_n=\beta_{-1}=\mathrm{constant}$) the terms in the third and fourth term do not appear.
On the other hand, if $n\geq 0$, then $n+1\geq 1= 3-2$ so it is compatible with the Definition of $\mathbf R_3$ in Definition \ref{def:Rk}.

Next, for the last term in the second line of \eqref{r2expand11},
\begin{align}
&\int_{\R^3}\beta_n(|S^1_{0,0}+\widetilde r|^2)s^1_{1+2\sigma(l),2l}\,dx=
\int_{\R^3}\beta_n(|S^1_{0,0}|^2)s^1_{1+2\sigma(l),2l}\,dx\nonumber\\&
\quad+\int_{\R^3}\beta_{n+1}(|S^1_{0,0}+\widetilde r|^2)\<S^1_{1+2\sigma(l),2l},\widetilde r\>_{\C^2}\,dx+\int_{\R^3}\beta_{n+1}(|S^1_{0,0}+\widetilde r|^2)|\widetilde r|^2s^1_{1+2\sigma(l),2l}\,dx\nonumber\\&=
S^0_{1+2\sigma(l),2l}+\int_{\R^3}\beta_{n+1}(|S^1_{0,0}|^2)\<S^1_{1+2\sigma(l),2l},\widetilde r\>_{\C^2}\,dx\nonumber\\&
\quad+\int_{\R^3}\beta_{n+2}(|S^1_{0,0}+\widetilde r|^2)\<S^1_{0,0},\sigma_1 \widetilde r\>_{\C^2}\<S^1_{1+2\sigma(l),2l},\widetilde r\>_{\C^2}\,dx\nonumber\\&
\quad+\int_{\R^3}\beta_{n+2}(|S^1_{0,0}+\widetilde r|^2)|\widetilde r|^2\<S^1_{1+2\sigma(l),2l},\widetilde r\>_{\C^2}\,dx+\mathbf R_2 \nonumber\\&
=S^0_{1+2\sigma(l),2l}+S^0_{1+2\sigma(l),2l+1}+\mathbf R_2+\mathbf R_3+\mathbf R_2\nonumber\\&=S^0_{1+2\sigma(l),2l}+\mathbf R_2+\mathbf R_3,\label{r2epandpr2}
\end{align}
where we have used \eqref{expandbeta} in the first and second equality.
Again, notice that $n+2\geq 1$ and this is compatible with the definition of $\mathbf R_3$.

Finally, the term $\int_{\R^3}\beta_n(|S^1_{0,0}+\widetilde r|^2)\<S^1_{1,0},\sigma_1\widetilde r\>_{\C^2}\<S^1_{0,0},\sigma_1\widetilde r\>_{\C^2}\,dx$ can be handled as above repeatedly using \eqref{expandbeta}.

Therefore, we have \eqref{R2expand}.
\end{proof}

\begin{lemma}
Let $k=3,4,5$.
Let $l\geq 1$, $\delta>0$ and let $  \mathfrak{F}\in C^\infty(D_{\Hrad^1}(\mathcal T_{\omega_*},\delta),\Hrad^1(\R^3,\widetilde \C))$ satisfy \eqref{Ferror}--\eqref{eq:sigma(i)}.
Then, we have
\begin{align}
\mathfrak F^*\mathbf R_k&=\sum_{j=2}^k\mathbf R_j+S^0_{(k-1)\sigma(l),(k-1)l+1}.\label{R34expand}
\end{align}

\end{lemma}

\begin{proof}
We consider the expansion of the term $\mathfrak F^* \int_{\R^3} \beta_n(|S^1_{0,0}+\widetilde r|^2)\<S^1_{0,0},\sigma_1\widetilde r\>^i |\widetilde r|^{2k}\,dx$.
Recall that $n\geq k-2$.
By $e^{S^0_{0,l}\im \sigma_3}S^1_{0,0}=S^1_{0,0}$, we have
\begin{align*}
&\mathfrak F^* \int_{\R^3} \beta_n(|S^1_{0,0}+\widetilde r|^2)\<S^1_{0,0},\sigma_1\widetilde r\>^i |\widetilde r|^{2k}\,dx\\&=\int_{\R^3} \beta_n(|S^1_{0,0}+\widetilde r|^2)\<S^1_{0,0},\sigma_1(\widetilde r+S^1_{\sigma(l),l})\>^i |\widetilde r +S^1_{\sigma(l),l}|^{2k}\,dx\\&
=\sum_{j=2}^k \mathbf R_j +\int_{\R^3}\beta_n(|S^1_{0,0}+\widetilde r|^2)\<S^1_{(k-1)\sigma(l),(k-1)l},\sigma_1 \widetilde r\>_{\C^2}\,dx+\int_{\R^3}\beta_n(|S^1_{0,0}+\widetilde r|^2)s^1_{k\sigma(l),kl}\,dx.
\end{align*}
Using \eqref{expandbeta} as for \eqref{r2epandpr1} and \eqref{r2epandpr2}, we have
\begin{align*}
\int_{\R^3}\beta_n(|S^1_{0,0}+\widetilde r|^2)\<S^1_{(k-1)\sigma(l),(k-1)l},\sigma_1 \widetilde r\>_{\C^2}\,dx=S^0_{(k-1)\sigma(l),(k-1)l+1}+\mathbf R_2+\mathbf R_3,
\end{align*}
and
\begin{align*}
\int_{\R^3}\beta_n(|S^1_{0,0}+\widetilde r|^2)s^1_{k\sigma(l),kl}\,dx=S^0_{k\sigma(l),kl}+\mathbf R_2+\mathbf R_3.
\end{align*}
Therefore, we have \eqref{R34expand}.
\end{proof}

\begin{lemma}
Let $l\geq 1$, $\delta>0$ and let $  \mathfrak{F}\in C^\infty(D_{\Hrad^1}(\mathcal T_{\omega_*},\delta),\Hrad^1(\R^3,\widetilde \C))$ satisfy \eqref{Ferror}--\eqref{eq:sigma(i)}.
Then, we have
\begin{align}
\mathfrak F^*\mathbf R_6&=\sum_{j=2}^6\mathbf R_j+S^0_{\sigma(l)+1,l+2}. \label{R567expand}
\end{align}

\end{lemma}

\begin{proof}
For $\mathfrak F^* \mathbf R_6$, we have, recalling $v(t,s)=sS^1_{0,0}+t (\widetilde r+S^1_{0,0})$,
\begin{align*}
&\mathfrak F^* \int_0^1\int_0^1(1-t)^2\int_{\R^3}g'''(|v(t,s)|^2)\<v(t,s),\sigma_1 S^1_{0,0}\>_{\C^2}\<v(t,s),\sigma_1 S^1_{1,1}\>_{\C^2}\<v(t,s),\sigma_1 \widetilde r\>_{\C^2}^2\,dxdsdt\\&=\mathbf R_6 +\int_{\R^3}\beta_3(|S^1_{0,0}+\widetilde r|^2)\<\widetilde r+S^1_{0,0},\sigma_1 S^1_{0,0}\>_{\C^2}\<\widetilde r+S^1_{0,0},\sigma_1 S^1_{1,1}\>_{\C^2}\<\widetilde r+S^1_{0,0},\sigma_1 S^1_{\sigma(l),l}\>_{\C^2}^2\,dx\\&
+\int_{\R^3}\beta_3(|S^1_{0,0}+\widetilde r|^2)\<\widetilde r+S^1_{0,0},\sigma_1 S^1_{0,0}\>_{\C^2}\<\widetilde r+S^1_{0,0},\sigma_1 S^1_{1,1}\>_{\C^2}\<\widetilde r+S^1_{0,0},\sigma_1 S^1_{\sigma(l),l}\>_{\C^2}\\&\quad\quad \times\<\widetilde r+S^1_{0,0},\sigma_1 \widetilde r\>_{\C^2}\,dx\\&
=\sum_{k=2}^6\mathbf R_k +\int_{\R^3}\beta_3(|S^1_{0,0}+\widetilde r|^2)\<S^1_{\sigma(l)+1,l+2},\sigma_1 \widetilde r\>_{\C^2}\,dx+\int_{\R^3}\beta_3(|S^1_{0,0}+\widetilde r|^2)s^1_{2\sigma(l)+1,2l+1} dx\\&
=\sum_{k=2}^6\mathbf R_k+S^0_{\sigma(l)+1,l+2},
\end{align*}
where we used the convention   in Definition \ref{def:beta} in the first equality when we replace $g'''$ with $\beta_3$ and erase the integral of $t$ and $s$.
Furthermore, in the last equality, we have used  \eqref{expandbeta} as in \eqref{r2epandpr1} and \eqref{r2epandpr2}.
Here, since \eqref{expandbeta} is used, we are again using the convention   in Definition \ref{def:beta}.

By the above equality, we have \eqref{R567expand}.
\end{proof}

\begin{lemma}
Let $l\geq 1$, $\delta>0$ and let $  \mathfrak{F}\in C^\infty(D_{\Hrad^1}(\mathcal T_{\omega_*},\delta),\Hrad^1(\R^3,\widetilde \C))$ satisfy \eqref{Ferror}--\eqref{eq:sigma(i)}.
Then, we have
\begin{align}
\mathfrak F^*\mathbf R_7&=\sum_{j=2}^7\mathbf R_j+S^0_{\sigma(l)+1,l+2}.\label{R7expand}
\end{align}
\end{lemma}

\begin{proof}
Since we have
\begin{equation}\label{nabla4E}
\begin{aligned}
&\nabla^4 E_P(u)(v_1,v_2,v_3)=g'(|u|^2)\sum_{\tau\in \mathfrak S_3}\<v_{\tau(1)}, \sigma_1 v_{\tau(2)}\>_{\C^2}v_{\tau(3)}\\&\quad+2g''(|u|^2)\sum_{\tau\in \mathfrak S_3}\(\<u, \sigma_1 v_{\tau(1)}\>_{\C^2}\<u,\sigma_1 v_{\tau(2)}\>v_{\tau(3)}  +\< v_{\tau(1)}, \sigma_1 v_{\tau(2)}\>_{\C^2}\<u, \sigma_1 v_{\tau(3)}\>_{\C^2}u \)\\&\quad+8g'''(|u|^2)\<u, \sigma_1 v_1\>_{\C^2}\<u,\sigma_1 v_2\>_{\C^2}\<u,\sigma_1 v_3\>_{\C^2}u,
\end{aligned}
\end{equation}
where $\mathfrak S_3$ is the permutation group of $\{1,2,3\}$, it suffices to consider the contribution of each terms of the r.h.s.\ of \eqref{nabla4E}.
For the first term of the r.h.s.\ of \eqref{nabla4E}, we have
\begin{align}
&\mathfrak F^*\int_0^1\int_0^1(1-t)^2g'(|v(t,s)|^2)\<S^1_{0,0},\sigma_1\widetilde r\>|\widetilde r|^2\,dx dsdt\label{R71}\\&=
\int_0^1\int_0^1(1-t)^2g'(|\mathfrak F^* v(t,s)|^2)\<S^1_{0,0},\sigma_1\(\widetilde r+S^1_{\sigma(l),l}\)\>_{\C^2}|\widetilde r+S^1_{\sigma(l),l}|^2\,dx dsdt\nonumber\\&
=\int_0^1\int_0^1(1-t)^2g'(|\mathfrak F^* v(t,s)|^2)\<S^1_{0,0},\sigma_1\widetilde r\>_{\C^2}|\widetilde r|^2\,dx dsdt
+\mathbf R_2+\mathbf R_3\nonumber\\&\quad+\int_{\R^3}\beta_1(|S^1_{0,0}+\widetilde r|^2)\<S^1_{2\sigma(l),2l},\sigma_1 \widetilde r\>_{\C^2}\,dx+\int_{\R^3}\beta_1(|S^1_{0,0}+\widetilde r|^2)s^1_{3\sigma(l),3l}\nonumber\\&=
\int_0^1\int_0^1(1-t)^2g'(|\mathfrak F^* v(t,s)|^2)\<S^1_{0,0},\sigma_1\widetilde r\>_{\C^2}|\widetilde r|^2\,dx dsdt
+\mathbf R_2+\mathbf R_3+S^0_{2\sigma(l),2l+1},\nonumber
\end{align}
where we have used the convention   in Definition \ref{def:beta} in the second equality when we replace $g'$ with $\beta_1$ and erase the integral of $t$ and $s$.
Furthermore, in the last equality, we have used  \eqref{expandbeta} as in \eqref{r2epandpr1} and \eqref{r2epandpr2}.
Also, notice that $1\geq k-2$ for $k=3$ so it is compatible with the definition of $\mathbf R_k$ in Definition \ref{def:Rk}.

For the second term of the r.h.s.\ of \eqref{nabla4E}, considering first term with $\tau=\mathrm{id}$, we have
\begin{align}
&\mathfrak F^*\int_0^1\int_0^1(1-t)^2g''(|v(t,s)|^2)\<v(t,s),\sigma_1 S^1_{0,0}\>_{\C^2}\<v(t,s),\sigma_1\widetilde r\>_{\C^2}|\widetilde r|^2\,dx dsdt\label{R72}\\&=
\int_0^1\int_0^1(1-t)^2g''(|\mathfrak F^*v(t,s)|^2)\<  \mathfrak{F}^*v(t,s),\sigma_1 S^1_{0,0}\>_{\C^2}\<\mathfrak F^*v(t,s),\sigma_1\(\widetilde r+S^1_{\sigma(l),l}\)\>_{\C^2}\nonumber\\&\quad\quad \times|\widetilde r+S^1_{\sigma(l),l}|^2\,dx dsdt\nonumber\\&
=\int_0^1\int_0^1(1-t)^2g''(|\mathfrak F^*v(t,s)|^2)\<  \mathfrak{F}^*v(t,s),\sigma_1 S^1_{0,0}\>_{\C^2}\<\mathfrak F^*v(t,s),\sigma_1\widetilde r\>_{\C^2}|\widetilde r|^2\,dx dsdt+\sum_{k=2}^4\mathbf R_k\nonumber\\&
\quad+\int_{\R^3}\beta_2(|S^1_{0,0}+\widetilde r|^2)\<S^1_{2\sigma(l),2l},\sigma_1\widetilde r\>_{\C^2}\,dx+\int_{\R^3}\beta_2(|S^1_{0,0}+\widetilde r|^2)s^1_{3\sigma(l),3l}\nonumber\\&
=\int_0^1\int_0^1(1-t)^2g''(|\mathfrak F^*v(t,s)|^2)\<  \mathfrak{F}^*v(t,s),\sigma_1 S^1_{0,0}\>_{\C^2}\<\mathfrak F^*v(t,s),\sigma_1\widetilde r\>_{\C^2}|\widetilde r|^2\,dx dsdt\nonumber\\&
\quad+\sum_{k=2}^4\mathbf R_k+S^0_{2\sigma(l),2l+1},\nonumber
\end{align}
where we have used the convention   in Definition \ref{def:beta} in the second equality when we replace $g''$ with $\beta_2$ and erase the integral of $t$ and $s$.
Furthermore, in the last equality, we have used  \eqref{expandbeta} as in \eqref{r2epandpr1} and \eqref{r2epandpr2}.
Also, notice that $2\geq k-2$ for $k=3,4$ so it is compatible with the definition of $\mathbf R_k$ in Definition \ref{def:Rk}.

The other terms in the second term of the r.h.s.\ of \eqref{nabla4E} will have the same estimate.
Finally, for the third term of the r.h.s.\ of \eqref{nabla4E}, we have
\begin{align}
&\mathfrak F^* \int_0^1\int_0^1 (1-t)^2 g'''(|v(t,s)|^2)\<v(t,s),\sigma_1 S^1_{0,0}\>_{\C^2}\<v(t,s),\sigma_1\widetilde r\>^3\,dxdsdt\label{R73}\\&=
\int_0^1\int_0^1 (1-t)^2 g'''(|\mathfrak F^*v(t,s)|^2)\<\mathfrak F^*v(t,s),\sigma_1 S^1_{0,0}\>_{\C^2}\<\mathfrak F^*v(t,s),\sigma_1\(\widetilde r+S^1_{\sigma(l),l}\)\>^3\,dxdsdt\nonumber\\&=
\int_0^1\int_0^1 (1-t)^2 g'''(|\mathfrak F^*v(t,s)|^2)\<\mathfrak F^*v(t,s),\sigma_1 S^1_{0,0}\>_{\C^2}\<\mathfrak F^*v(t,s),\sigma_1\widetilde r \>^3\,dxdsdt\nonumber\\&
\quad+\sum_{k=2}^6\mathbf R_k+S^0_{2\sigma(l),2l+1},\nonumber
\end{align}
where we have used the same argument as above.
Adding the first terms of r.h.s.\ in \eqref{R71}, \eqref{R72} and \eqref{R73}, we have $\mathbf R_7$.
Therefore, we have the conclusion.
\end{proof}

Combining \eqref{R2expand}, \eqref{R34expand}, \eqref{R567expand} and \eqref{R7expand}, we obtain Lemma \ref{lem:Rexpand}.

\subsection{Proof of Lemma \ref{lem:errorstzbound}}

\label{sec:errorstzbound}

In this subsection we prove Lemma \ref{lem:errorstzbound}.
Since $\mathbf R=S^0_{0,N+1}+\sum_{k=2}^7 \mathbf R_k$, we estimate each term in Lemmas \ref{lem:boundR0}, \ref{lem:boundR2}, \ref{lem:boundR345} and \ref{lem:boundR67}.

\begin{lemma}\label{lem:boundR0}
Let $s=0,1$. Assume $|\omega-\omega_*|\lesssim \epsilon$ and $|\lambda|+\|\widetilde r\|_{H^1}\lesssim \epsilon^{3/2}$.
Then, we have
\begin{align}\label{eq:boundR0}
\|\widetilde\nabla_{\widetilde r}S^0_{0,N+1}\|_{W^{s,6/5}}\lesssim \epsilon^{\frac32 N}.
\end{align}
\end{lemma}

\begin{proof}
By Lemma \eqref{difsymbol}, we have $\sigma_3\widetilde \nabla_{\widetilde r}S^0_{0,N+1}=S^1_{0,N}$.
Thus,
\begin{align}\label{contriS}
\| \sigma_3\widetilde \nabla_{\widetilde r}S^0_{0,N+1}\|_{W^{1,6/5}}= \| S^1_{0,N}\|_{W^{1,6/5}}\lesssim \epsilon^{3N/2}.
\end{align}
Therefore, we have \eqref{eq:boundR0}.
\end{proof}

\begin{lemma}\label{lem:boundR2}
Let $s=0,1$. Assume $|\omega-\omega_*|\lesssim \epsilon$ and $|\lambda|+\|\widetilde r\|_{H^1}\lesssim \epsilon^{3/2}$.
Then, we have
\begin{align}\label{eq:boundR2}
\|\widetilde\nabla_{\widetilde r}\mathbf R_2\|_{W^{s,6/5}}\lesssim \epsilon\|\widetilde r\|_{W^{s,6}}.
\end{align}
\end{lemma}

\begin{proof}
Let $f=\<S^2_{1,0}\widetilde r,\sigma_1 \widetilde r\>$.
Then, since $$\<\widetilde \nabla_{\widetilde r}f,\sigma_1\widetilde \eta\>=\<S^2_{1,0}\widetilde r,\widetilde \eta\> + \<d_{\widetilde r}S^2_{1,0}(\widetilde r,\widetilde \eta),\sigma_1\widetilde r\>,$$ and $d_{\widetilde r}S^2_{1,0}(\widetilde r,\cdot) = S^2_{1,1}$, we have
\begin{align*}
\|\widetilde \nabla_{\widetilde r}f\|_{\Sigma^s}&\lesssim \sup_{\|\widetilde \eta\|_{\Sigma^{-s}}\leq 1}\<\widetilde \nabla_{\widetilde r}f,\sigma_1\widetilde \eta\>\leq \|S^2_{1,0}\widetilde r \|_{\Sigma^s}+\sup_{\|\widetilde \eta\|_{\Sigma^{-s}}\leq 1}\|d_{\widetilde r}S^2_{1,0}(\widetilde r,\widetilde \eta)\|_{\Sigma^s}\|\widetilde r\|_{\Sigma^{-s}}\\&\lesssim_s \epsilon \|\widetilde r\|_{\Sigma^{-s}}.
\end{align*}
In particular, we have
\begin{align}\label{contriR21}
\|\widetilde\nabla_{\widetilde r} f\|_{W^{1,6/5}}\lesssim  \epsilon \|\widetilde r\|_{L^{6}}.
\end{align}
We next consider the term $ \widetilde \nabla_{\widetilde r}\int_{\R^3}\beta_n(|S^1_{0,0}+\widetilde r|^2)s^1_{1,0}|\widetilde r|^2\,dx$ in $\widetilde \nabla_{\widetilde r}\mathbf R_2$.
Since
\begin{align*}
\widetilde \nabla_{\widetilde r}\int_{\R^3}\beta_n(|S^1_{0,0}+\widetilde r|^2)s^1_{1,0}|\widetilde r|^2\,dx=&\beta_n(|S^1_{0,0}+\widetilde r|^2)s^1_{1,0}\widetilde r+S^2_{0,0}\beta_n(|S^1_{0,0}+\widetilde r|^2)|\widetilde r|^2\\&+(1+S^2_{0,0})\beta_{n+1}(|S^1_{0,0}+\widetilde r|^2)s^1_{1,0}|\widetilde r|^2(S^1_{0,0}+\widetilde r),
\end{align*}
where are using \eqref{expandbeta} and the convention of Definition \ref{def:beta},
we have
\begin{align}
&\| \widetilde \nabla_{\widetilde r}\int_{\R^3}\beta_n(|S^1_{0,0}+\widetilde r|^2)s^1_{1,0}|\widetilde r|^2\,dx\|_{W^{s,6/5}}\lesssim \epsilon\|\beta_n(|S^1_{0,0}+\widetilde r|^2)\widetilde r\|_{W^{s,6/5}}\nonumber\\&\quad+\|\beta_n(|S^1_{0,0}+\widetilde r|^2)|\widetilde r|^2\|_{L^{6/(2\tilde p+2)}}+\epsilon\|\beta_{n+1}(|S^1_{0,0}+\widetilde r|^2)|\widetilde r|^2 (S^1_{0,0}+\widetilde r)\|_{W^{s,6/5}}\lesssim \epsilon \|\widetilde r\|_{W^{s,6}},\label{contriR22}
\end{align}
where $\tilde p=\max(p,1)$ and we have used $L^{6/(2\tilde p+2)}\hookrightarrow\Sigma^{-2}$ coming from $\Sigma^2\hookrightarrow L^2\cap L^\infty$.

The contribution of the term $\sigma_3 \widetilde \nabla_{\widetilde r}\int_{\R^3}\beta_n(|S^1_{0,0}+\widetilde r|^2)\<S^1_{1,0},\sigma_1\widetilde r\>_{\C^2}\<S^1_{0,0},\sigma_1\widetilde r\>_{\C^2}\,dx$ also follows the same estimate as \eqref{contriR22}.
Therefore, from \eqref{contriR21} and \eqref{contriR22}, we have \eqref{eq:boundR2}.
\end{proof}

\begin{lemma}\label{lem:boundR345}
Let $k=3,4,5$.
Let $s=0,1$. Assume $|\omega-\omega_*|\lesssim \epsilon$ and $|\lambda|+\|\widetilde r\|_{H^1}\lesssim \epsilon^{3/2}$.
Then, we have
\begin{align*}
\|\widetilde\nabla_{\widetilde r}\mathbf R_k\|_{W^{s,6/5}}\lesssim \epsilon\|\widetilde r\|_{W^{s,6}}.
\end{align*}
\end{lemma}

\begin{proof}
We consider the term $\int_{\R^3}\beta_n(|S^1_{0,0}+\widetilde r|^2)\<S^1_{0,0},\sigma_1 \widetilde r\>^i |\widetilde r|^{2j}\,dx$ in $\mathbf R_k$ for $k=3,4,5$.
Here, recall that $n\geq k-2$ (see Definition \ref{def:Rk}).
Taking $\widetilde \nabla_{\widetilde r}$ and using \eqref{expandbeta} with the convention of Definition \ref{def:beta}, we have
\begin{align}
&\widetilde \nabla_{\widetilde r}\int_{\R^3}\beta_n(|S^1_{0,0}+\widetilde r|^2)\<S^1_{0,0},\sigma_1 \widetilde r\>_{\C^2}^i |\widetilde r|^{2j}\,dx=
\beta_n(|S^1_{0,0}+\widetilde r|^2)\<S^1_{0,0},\sigma \widetilde r\>_{\C^2}^i|\widetilde r|^{2(j-1)}\widetilde r \nonumber
\\&\quad+ \beta_n(|S^1_{0,0}|^2)\<S^1_{0,0},\sigma \widetilde r\>_{\C^2}^{i-1}|\widetilde r|^{2j}S^1_{0,0}
+S^2_{0,0}\(\beta_n(|S^1_{0,0}+\widetilde r|^2)\<S^1_{0,0},\sigma \widetilde r\>_{\C^2}^i|\widetilde r|^{2j}\widetilde r\) \nonumber
\\&\quad+(1+S^2_{0,0})\beta_{n+1}(|S^1_{0,0}+\widetilde r|^2)\<S^1_{0,0},\sigma_1\widetilde r\>_{\C^2}^i|\widetilde r|^{2j}(S^1_{0,0}+\widetilde r),\label{contriRk1}
\end{align}
where if $i=0$ or $j=0$, the terms with $i-1$ or $j-1$ do not exist.

For the first term of r.h.s.\ of \eqref{contriRk1}, if $s=0$,
\begin{align}\label{contriRk2}
\|\beta_n(|S^1_{0,0}+\widetilde r|^2)\<S^1_{0,0},\sigma_1 \widetilde r\>_{\C^2}^i|\widetilde r|^{2(j-1)}\widetilde r\|_{L^{6/5}}\lesssim \| (1+|S^1_{0,0}+\widetilde r|^{\max(2(p-n),0)})s^1_{0,0}|\widetilde r|^{k-1}\|_{L^{6/5}}.
\end{align}
If $k=4,5$ since $p<2$ and $n\geq2$, we have
\begin{align}
\|\beta_n(|S^1_{0,0}+\widetilde r|^2)\<S^1_{0,0},\sigma_1 \widetilde r\>_{\C^2}^i|\widetilde r|^{2(j-1)}\widetilde r\|_{L^{6/5}}&\lesssim \| s^1_{0,0}|\widetilde r|^{k-1}\|_{L^{6/5}}\lesssim \| |\widetilde r|^{k-1}\|_{L^{6/(k-1)}}\nonumber\\&
\leq \|\widetilde r\|_{L^6}^{k-1}\lesssim \epsilon \|\widetilde r\|_{L^6}.\label{contriRk3}
\end{align}
The same estimate holds for the case $k=3$ with $n\geq 2$ or $n=1$ and $p \leq 1$.
For the remaining case $k=3$, $n=1$ and $1<p<2$, we have
\begin{align}
&\|\beta_n(|S^1_{0,0}+\widetilde r|^2)\<S^1_{0,0},\sigma_1 \widetilde r\>_{\C^2}^i|\widetilde r|^{2(j-1)}\widetilde r\|_{L^{6/5}}\lesssim \| (1+|S^1_{0,0}+\widetilde r|^{2p-2})s^1_{0,0}|\widetilde r|^2\|_{L^{6/5}}\nonumber\\&\lesssim \| |\widetilde r|^2\|_{L^{3}}+\||S^1_{0,0}+\widetilde r|^{2p-2}|\widetilde r|^2 \|_{L^{6/2p}}\lesssim \epsilon \|\widetilde r\|_{L^6}.\label{contriRk4}
\end{align}
The case $s=1$ for the first term of the r.h.s.\ of \eqref{contriRk1}, since we have
\begin{align}
&\left|\nabla_x\(\beta_n(|S^1_{0,0}+\widetilde r|^2)\<S^1_{0,0},\sigma_1 \widetilde r\>_{\C^2}^i |\widetilde r|^{2(j^1)}\widetilde r\)\right|
\lesssim \left|\beta_n(|S^1_{0,0}+\widetilde r|^2)\right| |S^1_{0,0}| |\widetilde r|^{k-2}|\nabla_x \widetilde r|\nonumber\\&\quad+
\left|\beta_n(|S^1_{0,0}+\widetilde r|^2)\right| |S^1_{0,0}| |\widetilde r|^{k-1}
+\left|\beta_{n+1}(|v|^2)|v|\right| |S^1_{0,0}| |\widetilde r|^{k-1}\(S^1_{0,0}+|\nabla_x \widetilde r|\),\label{contriRk4.1}
\end{align}
where $v=S^1_{0,0}+\widetilde r$,
we can bound the $L^{6/5}$ norm of the first term in the r.h.s.\ of \eqref{contriRk4.1} by a similar manner to \eqref{contriRk2}, \eqref{contriRk3} and \eqref{contriRk4}, and we have
\begin{align}
\|\beta_n(|S^1_{0,0}+\widetilde r|^2) |S^1_{0,0}| |\widetilde r|^{k-2}\nabla_x \widetilde r\|_{L^{6/5}}\lesssim \epsilon\|\nabla_x \widetilde r\|_{L^6}.
\end{align}
The second term is the same as what appeared in \eqref{contriRk2}, so we can bound it by $\epsilon\|\widetilde r\|_{L^6}$.
The $L^{6/5}$ norm of the third term of the r.h.s.\ of \eqref{contriRk4.1} can be bounded by
\begin{align*}
&\|\beta_{n+1}(|v|^2)|v| |S^1_{0,0}| |\widetilde r|^{k-1}\(S^1_{0,0}+|\nabla_x \widetilde r|\)\|_{L^{6/5}}\lesssim \| |S^1_{0,0}|(1+|\widetilde r|^{\max(0,2p-2n-1)})|\widetilde r|^{k-1}(1+|\nabla_x\widetilde r|)\|_{L^{6/5}}\\&
\lesssim \|\widetilde r\|_{L^6}^{k-1}(1+\|\nabla_x\widetilde r\|_{L^6})+\|\widetilde r\|_{L^6}^{k-1+\max(0,2p-2n-1)}(1+\|\nabla_x\widetilde r\|_{L^6}),
\end{align*}
where the last term is only needed in the case $k=3$, $n=1$ and $p>1$.
Thus, we have
\begin{align}
	&\|\beta_n(|S^1_{0,0}+\widetilde r|^2)\<S^1_{0,0},\sigma_1 \widetilde r\>_{\C^2}^i|\widetilde r|^{2(j-1)}\widetilde r\|_{W^{s,6/5}}\lesssim \epsilon \|\widetilde r\|_{W^{s,6}}.\label{contriRk5}
\end{align}

The second term of the r.h.s.\ of \eqref{contriRk1} can be handled by similar manner.
The third term will be easier because we have a cutoff $S^2_{0,0}$ so we omit it.
The fourth term also can be handled similarly as the first term taking care of the decay property of $\beta_n$.
Therefore, we have the conclusion.
\end{proof}

\begin{lemma}\label{lem:boundR67}
Let $s=0,1$. Let $k=6,7$.
Assume $|\omega-\omega_*|\lesssim \epsilon$ and $|\lambda|+\|\widetilde r\|_{H^1}\lesssim \epsilon^{3/2}$.
Then, we have
\begin{align*}
\|\widetilde\nabla_{\widetilde r}\mathbf R_k\|_{W^{s,6/5}}\lesssim \epsilon\|\widetilde r\|_{W^{s,6}}.
\end{align*}
\end{lemma}

\begin{proof}
Since the proof will be similar to the proof of Lemma \ref{lem:boundR345}, we will omit it.
\end{proof}

By Lemmas \ref{lem:boundR0}, \ref{lem:boundR2}, \ref{lem:boundR345} and \ref{lem:boundR67}, we obtain Lemma \ref{lem:errorstzbound}.

\section*{Acknowledgments}
M.M. was supported by the JSPS KAKENHI Grant Numbers JP15K17568, JP17H02851 and JP17H02853.

%
%

Department of Mathematics and Geosciences,  University
of Trieste, via Valerio  12/1  Trieste, 34127  Italy.
{\it E-mail Address}: {\tt scuccagna@units.it}

Department of Mathematics and Informatics,
Faculty of Science,
Chiba University,
Chiba 263-8522, Japan.
{\it E-mail Address}: {\tt maeda@math.s.chiba-u.ac.jp}

\end{document}